\documentclass{amsart}

\usepackage{amssymb}
\usepackage{amsmath}
\usepackage{amsthm}
\usepackage{nicefrac}
\usepackage[dvipsnames]{xcolor}
\usepackage{tikz}
\usepackage{enumitem}
\usepackage{hyperref}
\usetikzlibrary{arrows}
\usepackage{stackengine}
\usepackage{scalerel}

\newtheorem{theorem}{Theorem}[section]
\newtheorem{lemma}[theorem]{Lemma}
\newtheorem{proposition}[theorem]{Proposition}
\newtheorem{corollary}[theorem]{Corollary}

\theoremstyle{definition}
\newtheorem{definition}[theorem]{Definition}
\newtheorem{situation}[theorem]{Situation}
\newtheorem{example}[theorem]{Example}
\newtheorem{question}[theorem]{Question}

\theoremstyle{remark}
\newtheorem{remark}[theorem]{Remark}
\newtheorem{claim}[theorem]{Claim}

\hypersetup{colorlinks=true,citecolor=Blue,linkcolor=Blue,urlcolor=Cyan}

\newcommand{\sh}{\mathcal}
\newcommand{\spec}{\mathrm{Spec}\mathop{}}
\newcommand{\Proj}{\mathrm{Proj}\mathop{}}
\newcommand{\Pic}{\mathrm{Pic}}
\newcommand{\zz}{\mathbb{Z}}
\newcommand{\cc}{{\mathbb{C}}}
\newcommand{\aline}{\mathbb{A}^1}
\newcommand{\cs}{\mathbb{A}^*}
\newcommand{\suf}{+}
\newcommand{\fuf}{\natural}
\newcommand{\tfuf}{\mp}
\newcommand{\hsymb}{\star}
\newcommand{\graph}{\Gamma}
\newcommand{\blocks}{{\mathcal B}}
\newcommand{\dg}{{\mathcal E}}

\begin{document}

	\title[Small resolutions of moduli spaces of scaled curves]{Small resolutions of moduli spaces of scaled curves}
	\author[A. Zahariuc]{Adrian Zahariuc}
	\date{}
	\address{Department of Mathematics and Statistics, University of Windsor, 401 Sunset Ave, Windsor, ON, N9B 3P4, Canada}
	\email{adrian.zahariuc@uwindsor.ca}
	\thanks{\includegraphics[scale=0.3]{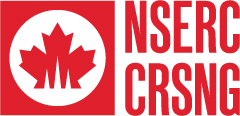} We acknowledge the support of the Natural Sciences and Engineering Research Council of Canada (NSERC), RGPIN-2020-05497. Cette recherche a \'{e}t\'{e} financ\'{e}e par le Conseil de recherches en sciences naturelles et en g\'{e}nie du Canada (CRSNG), RGPIN-2020-05497.}
	
	\maketitle
	
	\begin{abstract}
	We construct small resolutions of the moduli space $\overline{Q}_n$ of stable scaled $n$-marked lines of Ma'u--Woodward and Ziltener and of the moduli space $\overline{P}_n$ of stable $n$-marked ${\mathbb G}_a$-rational trees introduced in earlier work. The resolution of $\overline{P}_n$ is an augmented wonderful variety associated to the graphic matroid of a complete graph. The resolution of $\overline{Q}_n$ is a further blowup, also a wonderful model of an arrangement in ${\mathbb P}^{n-1}$.
	\end{abstract}
	
	\section{Introduction}\label{section: introduction}

	\subsection{Statement of the results} We work over the field ${\mathbb C}$ of complex numbers.
	
	The moduli space of stable scaled lines with marked points constructed by Ma'u and Woodward \cite[\S10]{[MW10]} and Ziltener \cite{[Zi14]}, which we will denote by $\overline{Q}_n$ (also denoted by $Q_n$, $\overline{Q}_n({\mathbb C})$, or $\overline{\mathcal M}_{n,1}({\mathbb A})$ in the literature), is a remarkable compactification of the space of $n$-tuples of distinct points on an affine line modulo translation, i.e. 
	\[ (z_1,\ldots,z_n) \sim (z'_1,\ldots,z'_n) \quad \text{if} \quad z'_1-z_1 = \cdots = z'_n-z_n. \] 
	It is reminiscent of $\overline{M}_{0,n}$, but the curves in question have some partially rigidifying structure given by an extra marked point and a section of $\overline{\omega} = {\mathbb P}(\omega \oplus {\sh O})$, which satisfy some conditions, and generalize $[1:0] \in {\mathbb P}^1$ and $dz$ respectively for smooth curves. This space is important in the context of stable gauged maps \cite{[GSW18]} and other contexts \cite{[IKLPR23]}. In \cite{[Za21]}, a simpler version of $\overline{Q}_n$ denoted $\overline{P}_n$ was considered, in which the section of $\overline{\omega}$ is replaced by a section of $\omega^\vee$, and the marked points labeled $1,\ldots,n$ may coincide, making it a ${\mathbb G}_a^{n-1}$-equivariant compactification of ${\mathbb A}^n/{\mathbb G}_a$ cf. \cite{[HT99]}, where ${\mathbb G}_a = (\cc,+)$ acts diagonally. These spaces are reviewed in \S\ref{subsection: review of Q and P}.
	
	Both $\overline{Q}_n$ and $\overline{P}_n$ are singular for $n \geq 4$, but only mildly and in codimension 3, so they are intriguingly close to being nonsingular. Thus, it is very natural to ask if there are there good ways to resolve their singularities. In this paper, we answer this question by constructing small resolutions of singularities of $\overline{Q}_n$ and $\overline{P}_n$. Here, a `small resolution of singularities' means a proper birational morphism from a smooth variety, whose exceptional locus has codimension at least $2$ (see also Remark \ref{remark: not IH small}). 
	
	The varieties which resolve $\overline{Q}_n$ and $\overline{P}_n$ can be obtained by blowing up certain linear subspaces of ${\mathbb P}^{n-1}$, all of which are intersections of some of the hyperplanes $ X_i = X_j$ and $X_1+\cdots+X_n=0$. 
	The construction is the following. For any subset $S$ of $[n]=\{1,\ldots,n\}$ with at least $2$ elements, we have a \emph{diagonal} of ${\mathbb P}^{n-1}$,
	\[ \{[X_1:\cdots:X_n] \in {\mathbb P}^{n-1}: X_i=X_j \text{ if $i,j \in S$} \}. \] 
	\emph{Polydiagonals} of ${\mathbb P}^{n-1}$ are arbitrary intersections of diagonals, including ${\mathbb P}^{n-1}$ itself. Let 
	$H$ be the hyperplane $X_1+\cdots+X_n = 0$. Let ${\mathcal D}iag$ be the set of diagonals, and
	\begin{equation}\label{equation: An-1 mod scaling}
		{\mathcal P}oly_H = \left\{ P \cap H : \quad P \neq \{[1:\cdots:1]\} \text{ is a polydiagonal} \right\}.
	\end{equation}
	Both ${\mathcal P}oly_H$ and ${\mathcal D}iag \cup {\mathcal P}oly_H$ are building sets in the sense of \cite{[DP95], [Li09]}. The former is trivially so, being closed under (non-empty) intersections, while the latter follows from the fact that any polydiagonal is the transverse intersection of the minimal diagonals which contain it.
	
	\begin{definition}\label{definition: initial definition of Wn and Tn}
	Let $W_n$ and $W'_n$ be the De Concini--Procesi wonderful models of ${\mathcal P}oly_H$ and ${\mathcal D}iag \cup {\mathcal P}oly_H$ respectively.
	\end{definition}
	
	It turns out that $W_n$ is an example of an \emph{augmented wonderful variety}, namely, the one associated to the graphic matroid of the complete graph $K_n$ and its natural representation by a positive root system of $\mathfrak{sl}_n$ (\S\ref{subsection: rearranged definition}). Augmented wonderful varieties, introduced and studied by Huh, Wang, and others \cite{[HW17], [BHMPW20a], [BHMPW20b], [EHL22], [LLPP22]}, play a significant role in the major emerging area between algebraic geometry and matroid theory. 
	
	Let $\cc$ act diagonally on $\cc^n$ and 
	\[ \mathrm{F}(n,\cc) := \{ (x_1,\ldots,x_n) \in \cc^n: x_i \neq x_j, \text{ for all }i \neq j\} \subset \cc^n. \] 
	The open strata $P_n \subset \overline{P}_n$ and $Q_n \subset \overline{Q}_n$ where the curves are smooth are isomorphic to $\cc^n/\cc$ respectively $\mathrm{F}(n,\cc)/\cc$ (see \S\ref{subsection: review of Q and P}), so $\overline{P}_n$ and $\overline{Q}_n$ are (singular) compactifications of $\cc^n/\cc$ and $\mathrm{F}(n,\cc)/\cc$ respectively. On the other hand, $W_n$ and $W'_n$ are nonsingular compactifications of $\cc^n/\cc$ and $\mathrm{F}(n,\cc)/\cc$ respectively with simple normal crossing boundary divisors (see \S\ref{subsection: rearranged definition}). 
	
	
	\begin{theorem}\label{theorem: small resolution theorem}
		There exists a small resolution $\gamma:W_n \to \overline{P}_n$ if $n \geq 4$.
	\end{theorem}

	\begin{theorem}\label{theorem: small resolution theorem 2}
		There exists a small resolution $\eta:W'_n \to \overline{Q}_n$ if $n \geq 4$.
	\end{theorem}
	
	If $n \leq 3$, then $\overline{P}_n \simeq W_n$ and $\overline{Q}_n \simeq W'_n$.
	
	It is a well-known fact, which follows from work of Kapranov \cite[Theorem 4.3.3]{[Ka93]}, that $\overline{M}_{0,n}$ is the minimal wonderful model of the projectivization of the $A_{n-2}$ arrangement. Theorems \ref{theorem: small resolution theorem} and \ref{theorem: small resolution theorem 2} can be thought of as a version for scaled curves of this fact, with the notable difference that the wonderful models are instead the small resolutions of the moduli spaces. The results also fit naturally with (and refine) the construction of the augmented wonderful variety as a resolution of singularities of a matroid Schubert variety in this case, and with the spaces of flower curves and cactus flower curves recently introduced in \cite{[IKLPR23]} (Remark \ref{remark: flower curves}).
	
	\subsection{Polydiagonal degenerations and outline of proof} The idea to prove Theorems \ref{theorem: small resolution theorem} and \ref{theorem: small resolution theorem 2} is to obtain the maps to $\overline{P}_n$ and $\overline{Q}_n$ from the representability of the moduli functors, by constructing suitable families of curves over $W_n$ and $W'_n$. Hence, our arguments will focus on $W_n$ and $W'_n$ rather than $\overline{P}_n$ and $\overline{Q}_n$.
		
	For any (smooth, projective) variety $X$, we will describe two constructions:
	\begin{itemize}
		\item the \emph{polydiagonal degeneration of $X^n$}, obtained by blowing up $\aline \times X^n$ along all polydiagonals of its central fiber $\{0\} \times X^n \subset \aline \times X^n$; and
		\item the \emph{polydiagonal degeneration of the Fulton--MacPherson compactification $X[n]$}, obtained by \emph{additionally} blowing up all diagonals of $\aline \times X^n$ (that is, the direct product of $\aline$ with all diagonals of $X^n$).
	\end{itemize}
	The name is inspired by Ulyanov's polydiagonal compactification $X \langle n \rangle$ \cite{[Ul02]}. (A common generalization of the two constructions is likely possible, but we have decided against further complicating matters.) The polydiagonal degenerations are constructed using Li's theory of wonderful compactifications \cite{[Li09]}, and still parametrize degenerations of $X$ in the style of Fulton--MacPherson, by Theorem \ref{theorem: polydiagonal degeneration}. A new feature is the fact that these degenerations of $X$ are, in some suitable sense, scaled. The scaling is defined in terms of a line bundle which is roughly the $\dim X$-th root of the dualizing sheaf (\S\ref{section: naive FM spaces}).
	
	The case $\dim X \geq 2$ is not required in the proofs of Theorems \ref{theorem: small resolution theorem} and \ref{theorem: small resolution theorem 2}, though it is included for its own sake, and for potential generalizations (Remark \ref{remark: generalization to any dimension}). In fact, one could avoid the use of polydiagonal degenerations and the auxiliary curve $X$ in the proofs of the main results altogether, at the expense of a less streamlined and natural construction (see also Remark \ref{remark: final remark}).
	
	When $X=C$ is a curve, $W_n$ is a codimension $2$ subvariety of the polydiagonal degeneration of $C^n$, specifically, the fiber over $(0,p,p,\ldots,p) \in \aline \times C^n$ of the birational morphism from the polydiagonal degeneration of $C^n$ to $\aline \times C^n$, for some arbitrary $p \in C$. If we restrict the natural universal family to $W_n$, we can erase the components isomorphic to $C$, and we obtain an \emph{unstable} family of $n$-marked ${\mathbb G}_a$-rational trees (stable $n$-marked ${\mathbb G}_a$-rational trees is the name given in \cite{[Za21]} to the curves parametrized by $\overline{P}_n$). We stabilize it and obtain a map $W_n \to \overline{P}_n$ by \cite[Theorem 1.3]{[Za21]}, which turns out to be the desired small resolution. The idea for $\overline{Q}_n$ is similar, replacing the polydiagonal degeneration of  $C^n$ with that of $C [n]$.
		
	\subsection*{Structure of the paper} The paper is organized as follows.
	\begin{enumerate}
		\item [\S\ref{section: discussion}] contains a brief review of $\overline{Q}_n$ and $\overline{P}_n$, and some further discussion.
		\item [\S\ref{section: naive FM spaces}] focuses on a certain line bundle (which we call the \emph{root relative dualizing sheaf}) defined on families of Fulton--MacPherson degenerations of a variety. The `scaling' structure in \S\ref{section: polydiagonal degenerations} will be defined in terms of this line bundle. 
		\item [\S\ref{section: combinatorial language}] contains combinatorial language and preliminaries. 
		\item [\S\ref{section: polydiagonal degenerations}] discusses polydiagonal degenerations and their universal families. The main result is Theorem \ref{theorem: polydiagonal degeneration}.
		\item [\S\ref{section: small resolutions}] explains how the picture in \S\ref{section: polydiagonal degenerations} provides the desired families of curves over $W_n$ and $W'_n$, then checks that the classifying morphisms are small resolutions of $\overline{P}_n$ respectively $\overline{Q}_n$.
	\end{enumerate}

	\section{Discussion of the spaces and resolutions}\label{section: discussion}
	
	\subsection{Review of $\overline{Q}_n$ and $\overline{P}_n$}\label{subsection: review of Q and P} 
	\subsubsection*{Review of $\overline{P}_n$}
	Although $\overline{Q}_n$ predates $\overline{P}_n$, we review $\overline{P}_n$ first because it is simpler. What we denote by $\overline{P}_n$ in this paper would be denoted by $\overline{P}_{n,{\mathbb C}}$ in \cite[\S1.2]{[Za21]}. Given a noetherian $\cc$-scheme $S$ (for the purposes of this paper, it suffices to consider only the case when $S$ is a variety, although \cite[Theorem 1.3]{[Za21]} applies more generally), an $S$-point of $\overline{P}_n$ corresponds to data $(C,x_1,\ldots,x_n,x_\infty,v)$ such that:
	\begin{enumerate}
		\item $\pi:C \to S$ is a prestable (nodal) curve over $S$ of arithmetic genus $0$;
		\item $x_1,\ldots,x_n,x_\infty: S \to C$ are sections of $C \to S$, such that, for any geometric point $s$ of $S$, $x_i(s) \neq x_\infty(s)$ for $i=1,\ldots,n$ and $x_1(s),\ldots,x_n(s),x_\infty(s)$ are nonsingular points of $C_s$;
		\item\label{item: vector field} $v$ is a global section of $\omega^\vee_{C/S}(-2x_\infty(S))$, which does \emph{not} vanish at $x_1(s),\ldots,x_n(s)$ for any $s \in S$; and
		\item for any geometric point $s$ of $S$, any terminal component of $C_s$ contains at least one of $x_1(s),\ldots,x_n(s)$, while any non-terminal component of $C_s$ contains at least $3$ nodes of $C_s$ if it does not contain $x_\infty(s)$, respectively at least $2$ nodes if it contains $x_\infty(s)$.
	\end{enumerate}
	Note that $x_1(s),\ldots,x_n(s)$ are allowed to coincide. However, $x_1(s),\ldots,x_n(s)$ are forced to live on terminal components of $C_s$ (leaves of the dual tree) by condition \ref{item: vector field}. We consider the component containing $x_\infty(s)$ to be the root of the dual tree. Please see \cite[Figure 2]{[Za21]} for a picture. With the obvious changes over ${\mathbb C}$, $\overline{P}_n$ is a fine moduli space by \cite[Theorem 1.3]{[Za21]}.
	
	By \cite[Theorems 1.3 and 1.4]{[Za21]}, $\overline{P}_n$ is a ${\mathbb G}_a^n/{\mathbb G}_a$-equivariant compactification of ${\mathbb C}^n/{\mathbb C}$, where $\cc \hookrightarrow \cc^n$ diagonally. Specifically, the open subset $P_n \subset \overline{P}_n$ where the curves are irreducible is isomorphic to ${\mathbb C}^n/{\mathbb C}$, and the open immersion $\cc^n/\cc \hookrightarrow \overline{P}_n$ maps the coset of $(x_1,\ldots,x_n) \in \cc^n$ to 
	\[ \left( {\mathbb P}^1,[x_1:1],\ldots,[x_n:1],[1:0],\frac{d}{dx} \right) \] in $\overline{P}_n$. To further elucidate the `modulo translation' nature of $\overline{P}_n$, we note that data $(C,x_\infty,v)$, where $C \to S$ is a \emph{smooth} genus $0$ curve, $x_\infty$ is a section, and nonzero $v \in \Gamma(C,\omega^\vee_{C/S}(-2x_\infty(S)))$, with the natural notion of pullback, is in fact the stack $B{\mathbb G}_a$. Furthermore, considering a $\cc$-point of $\overline{P}_n$ where the curve $C$ is singular, if $Y \not\ni x_\infty$ is a terminal irreducible component of $C$ containing a unique node $q$ of $C$, then $\omega_C^\vee(-2x_\infty)$ restricts on $Y$ to $\omega_Y^\vee(-q)$, but $v$ vanishes on all non-terminal components by a simple calculation, so we can think of $v|_Y$ as a global section of $\omega_Y^\vee(-2q)$, thus giving an affine structure on $Y$ by a reiteration of the smooth case.
	
	The universal curve over $\overline{P}_n$ is isomorphic to $\overline{P}_{n+1}$. This claim is inherent in the inductive construction of $\overline{P}_n$ in \cite{[Za21]} -- see the claim in \cite[\S6.2.3]{[Za21]} that $\overline{P}_n$ is obtained by running Construction 6.4 in \emph{loc. cit.}
	
	The moduli space $\overline{P}_n$ has a natural stratification by the dual graphs of the curves. The open stratum $P_n$ was discussed above. It can be shown that the codimension $1$ strata (which will be denoted by $P_n[\rho]$ in the proof of Proposition \ref{proposition: isomorphism in codimension 2}) consist of `comb' curves (that is, curves whose dual tree is a star, with the center being the root component) with at least $2$ teeth, such that the backbone contains only the marking $x_\infty$, the remaining markings are distributed on the teeth (at least one marking on each tooth), and the vector field $v$ vanishes everywhere on the backbone, but not on the teeth -- see Figure \ref{figure: boundary of q}.
	
	\subsubsection*{Review of $\overline{Q}_n$}
	
	For more in-depth accounts of $\overline{Q}_n$, we refer the reader to \cite[\S10]{[MW10]} (where $\overline{Q}_n$ is denoted $Q_n({\mathbb C})$), as well as \cite[\S2]{[GSW17], [GSW18]} and \cite[Example 4.2.(d)]{[Wo15]} (where $\overline{Q}_n$ is denoted $\overline{\mathcal M}_{n,1}({\mathbb A})$). Our notation will be slightly non-standard compared to these sources: we will index the special marked point by $\infty$ instead of $0$, and denote the scaling (see below) by $\sigma$. 
	
	By comparison with $\overline{P}_n$, the additional feature of $\overline{Q}_n$ is that it disallows the points $x_1,\ldots,x_n$ to coincide. If some of these points collide, then the curve `sprouts out' new components in precisely the same manner it would for $\overline{{\mathcal M}}_{g,n}$ or Fulton--MacPherson spaces of curves. However, this also blows up the vector field, which needs to be replaced by a section $\sigma$ of the compactified log-(co)tangent bundle ${\mathbb P}(\omega_C \oplus {\sh O}_C)$ called `scaling'. \emph{The conventions are reversed}: $0$ vector field is infinite scaling and `infinite vector field' is $0$ scaling. 
	
	Following \cite[Definition 2.4 and Theorem 2.5]{[GSW17]}, if $S$ is a $\cc$-scheme (though, in this paper, we will only use the case when $S$ is a variety), an $S$-point of $\overline{Q}_n$ corresponds to data $(C,x_1,\ldots,x_n,x_\infty,\sigma)$ such that:
	\begin{enumerate}
		\item $\pi:C \to S$ is a prestable curve over $S$ of arithmetic genus $0$;
		\item $x_1,\ldots,x_n,x_\infty: S \to C$ are smooth pairwise disjoint sections of $C \to S$;
		\item\label{item: scaling} $\sigma$ is a global section of the ${\mathbb P}^1$-bundle ${\mathbb P}(\omega_{C/S} \oplus {\sh O}_C) = {\mathbb P}({\sh O}_C \oplus \omega^\vee_{C/S}) \to C$ such that $1/\sigma$ is a regular local section of $\omega_C^\vee(-2x_\infty(S))$ in a neighbourhood of $x_\infty(S)$;
		\item for any geometric point $s$ of $S$, the following conditions are satisfied:
		\begin{enumerate}
			\item The scaling $\sigma(s)$ is infinite at $x_\infty(s)$, and finite (possibly $0$) at $x_1(s),\ldots,x_n(s)$.
			\item The scaling $\sigma_s$ is monotone on $C_s$ in the following sense: for each terminal component $Y$ of $C_s$, there is a canonical non-self-crossing path of components $Y_0,\ldots,Y_k = Y$ starting with the root component $Y_0$ (the one containing $x_\infty$), such that the scaling $\sigma_s$ is finite and non-zero on at most one of these components -- let us call it the \emph{transition} component $Y_\alpha$ -- and the scaling is infinite on $Y_i$ for all $i<\alpha$, respectively $0$ on $Y_i$ for all $i>\alpha$.
			\item The natural stability condition equivalent to the absence of nontrivial automorphisms holds: any irreducible component of $C_s$ either contains at least $3$ special points (markings or nodes of $C_s$), or contains at least $2$ special points and has non-degenerate (finite and non-zero) scaling.
		\end{enumerate} 
	\end{enumerate}
	(The second part of condition \ref{item: scaling} is required e.g. to rule out the case when $\sigma$ has two isolated simple poles.) Thus, if $(C,x_1,\ldots,x_n,x_\infty,\sigma)$ corresponds to a $\cc$-point of $\overline{Q}_n$, there are three types of irreducible components of $C$: components with infinite scaling, (nonzero) finite scaling, and $0$ scaling. Unless explicitly stated otherwise, by finite scaling, we mean non-degenerate scaling (not just finite, but also nonzero). 
	
	Let $\mathrm{F}(n,\cc) = \{ (x_1,\ldots,x_n) \in \cc^n: x_i \neq x_j, \text{ for all }i \neq j\} \subset \cc^n$. The diagonal action of $(\cc,+)$ on $\cc^n$ restricts to an action on $\mathrm{F}(n,\cc)$, and we obtain an open immersion $\mathrm{F}(n,\cc)/\cc \hookrightarrow \overline{Q}_n$ which maps the coset of $(x_1,\ldots,x_n)$ to 
	\[ ({\mathbb P}^1,[x_1:1],\ldots,[x_n:1],[1:0],dx). \] 
	Its image is the open stratum $Q_n \subset \overline{Q}_n$ consisting of smooth curves. It can be shown that there are two types of codimension $1$ strata on $\overline{Q}_n$ -- see Figure \ref{figure: boundary of q}. One (which will be denoted by $Q_n[\rho]$ in the proof of Proposition \ref{proposition: isomorphism in codimension 2 for q}) is completely analogous to the codimension $1$ strata of $\overline{P}_n$ discussed above, with the sole additional requirement that the marking are pairwise distinct; note that the root/backbone has infinite scaling, and the teeth have finite scaling. In the other type of codimension 1 stratum (which will be denoted by $Q_n[S]$ in the proof of Proposition \ref{proposition: isomorphism in codimension 2 for q}), the curves have two irreducible components: the component containing $x_\infty$ has finite scaling, while the other component has $0$ scaling and contains at least two of the marked points $x_1,\ldots,x_n$.
	
	\begin{figure}[h]	
	\begin{center}
		\begin{tikzpicture}[scale = 0.5]
			\draw (0,0) circle (1); \draw[fill=white] (0,-1) circle (3pt) node[anchor = south] {$\infty$};
			\draw (1.3,0.5) circle (0.4); \draw[fill=black] (1.7,0.5) circle (3pt); \draw[fill=black] (1.5,0.82) circle (3pt);
			\draw (-1.3,0.5) circle (0.4);  \draw[fill=black] (-1.7,0.5) circle (3pt); \draw[fill=black] (-1.3,0.1) circle (3pt);
			\draw[fill=black] (-1.3,0.9) circle (3pt); 
			\draw (1.3,-0.5) circle (0.4); \draw[fill=black] (1.3,-0.9) circle (3pt); \draw[fill=black] (-1.3,-0.9) circle (3pt);
			\draw (-1.3,-0.5) circle (0.4);
			\draw (0,1.4) circle (0.4); \draw[fill=black] (0,1.8) circle (3pt);
			
			\draw (6,0) circle (1); \draw[fill=white] (6,-1) circle (3pt) node[anchor = south] {$\infty$};
			\draw[fill=black] (5.2,0.6) circle (3pt); \draw[fill=black] (6,1) circle (3pt); \draw[fill=black] (5,0) circle (3pt);
			\draw[fill=black] (5.2,-0.6) circle (3pt); \draw[fill=black] (5.2,-0.6) circle (3pt); \draw[fill=black] (6.8,-0.6) circle (3pt);
			\draw (7.5,0.7) circle (0.67); \draw[fill=black] (7.5,1.37) circle (3pt); \draw[fill=black] (7.5,0.03) circle (3pt);  \draw[fill=black] (8.17,0.7) circle (3pt);
			
			\node at (-5,0) {finite scaling}; \draw [->,dotted] (-3,0) to (-2,0.5); \draw [->,dotted] (-3,0) to (-2,-0.5);
			
			\node at (-4,2) {infinite scaling}; \draw [->,dotted] (-2,2) to (-0.7,0.8);
			
			\node at (12,-0.5) {finite scaling}; \draw [->,dotted] (10,-0.5) to (7,-0.5); 
			\node at (12,0.5) {zero scaling}; \draw [->,dotted] (10,0.5) to (8.4,0.5);
		\end{tikzpicture} 
	\caption{The two types of boundary divisors of $\overline{Q}_n$. The picture on the left also corresponds to a boundary divisor of $\overline{P}_n$.}
	\label{figure: boundary of q}
	\end{center}
	\end{figure}
	
	The space $\overline{Q}_n$ is constructed as a projective variety in \cite[\S10]{[MW10]}. The equivalence with the modular point of view is stated in \cite[Example 4.2.(d)]{[Wo15]}, a point of view discussed in more detail in \cite[\S2]{[GSW17], [GSW18]}.
	
	Both $\overline{Q}_n$ and $\overline{P}_n$ are normal local complete intersection (irreducible) projective varieties of dimension $n-1$, which are singular in codimension $3$ for $n \geq 4$. We will justify or reference the less obvious facts in the previous sentence only as needed.

	\subsubsection*{Relation between $\overline{Q}_n$ and $\overline{P}_n$}
	
	Given a scaled curve $(C,x_1,\ldots,x_n,x_\infty,\sigma) \in \overline{Q}_n(\cc)$, there is a morphism $f:C \to C'$ which contracts all $0$-scaling components of $C$, with a one-sided inverse $C' \hookrightarrow C$. If $D \subset C'$ is the image of the components contracted by $f$, then $\omega_C|_{C'} \simeq \omega_{C'}(D)$ and, in a Zariski neighbourhood of $D$ in $C'$, $\sigma|_{C'}$ is a meromorphic section of $\omega_C|_{C'}$ with simple poles at $D$, and it follows that $v:=1/(\sigma|_{C'})$ can be regarded as a global section of $\omega^\vee_C$. Then, it is easy to see that $(C',f(x_1),\ldots,f(x_n),f(x_\infty),v) \in \overline{P}_n(\cc)$, so we have defined a function $\overline{Q}_n(\cc) \to \overline{P}_n(\cc)$. This function is indeed a morphism, by Proposition \ref{proposition: morphism from Q to P} below. The proposition will not be used in the paper, so we will sketch the proof, which is rather technical, in Appendix \ref{appendix: contracting components}.
	
	\begin{proposition}\label{proposition: morphism from Q to P}
		There exists a morphism $\tau:\overline{Q}_n \to \overline{P}_n$ which induces the function above on $\cc$-points. 
	\end{proposition}
	
	\subsubsection*{Relations to other spaces}
	
	Some notable relations between $\overline{Q}_n$ or $\overline{P}_n$ and other known spaces are summarized below.
	
	\begin{remark}\label{remark: relation to spaces 1}
	By \cite[Theorem 1.4]{[Za21]}, $\overline{P}_n$ is a degeneration of the Losev--Manin space $\overline{L}_n$ \cite{[LM00]} compatible with the natural actions of ${\mathbb G}_a^{n-1}$ and ${\mathbb G}_m^{n-1}$ respectively. This relies on the degeneration of ${\mathbb G}_m$ to ${\mathbb G}_a$ by $\alpha \to 0$ in $x *_\alpha y = x+y+\alpha xy$ ($x,y \neq \nicefrac{-1}{\alpha}$). In \cite{[IKLPR23]}, it is shown that $\overline{Q}_n$ is a degeneration of $\overline{M}_{0,n+2}$.
	\end{remark}
	
	\begin{remark}\label{remark: relation to spaces 2}
	This remark and the follow-up Remark \ref{remark: flower curves} are based on discussions with Joel Kamnitzer. Let $\overline{\mathfrak{t}}_n$ be the moduli space of flower curves in \cite{[IKLPR23]}, which is also a matroid Schubert variety. Recall that $\overline{P}_{n+1}$ is the universal curve over $\overline{P}_n$ (see the review of $\overline{P}_n$ above). Composing, we obtain maps $\delta_{ij}:\overline{P}_n \to \overline{P}_2 \simeq {\mathbb P}^1$, which forget all markings except any pair $i \neq j$, and stabilize the resulting curves. Taken together, these give a map $ (\delta_{ij})_{i \neq j}: \overline{P}_n \to ({\mathbb P}^1)^{n^2-n} $, whose image is easily checked to be precisely $\overline{\mathfrak{t}}_n$, for instance, by restricting to the open stratum ${\mathbb C}^n/{\mathbb C}$. Thus, we obtain a morphism
	\[ \delta: \overline{P}_n \to \overline{\mathfrak t}_n. \]
	Geometrically, this operation contracts the `interior' of a curve corresponding to some ${\mathbb C}$-point of $\overline{P}_n$ to a point: the leaves turn into petals and the tree shrinks to a flower. For instance, it can be shown that the fiber of $\delta$ over $\infty^{n^2-n} \in \overline{\mathfrak t}_n$ is a Weil divisor on $\overline{P}_n$ isomorphic to $\overline{M}_{0,n+1}$. 
		
	In \cite{[IKLPR23]}, it is shown that there is a birational morphism $\overline{Q}_n \to \overline{F}_n$, where $\overline{F}_n$ is the space of cactus flower curves introduced in \emph{loc. cit.}
	\end{remark}
	
	\subsection{The wonderful models $W_n$ and $W'_n$}\label{subsection: rearranged definition} 
	It is convenient to slightly rearrange Definition \ref{definition: initial definition of Wn and Tn}. If we embed $\cc \hookrightarrow \cc^n$ diagonally and consider the isomorphism 
	$ \cc^n \cong \cc \oplus \cc^\perp \cong \cc \oplus \cc^n/\cc $ relative to the standard inner product $\langle u,v \rangle = \sum_{i=1}^n u_i v_i$ and the corresponding isomorphism ${\mathbb P}^{n-1} \cong {\mathbb P}(\cc \oplus \cc^n/\cc)$, then the hyperplane $H = \{X_1+\cdots+X_n = 0\}$ is simply ${\mathbb P} ( 0 \oplus \cc^n / \cc) $, and the arrangements considered in \S\ref{section: introduction} correspond to
	\begin{equation} 
	\begin{aligned}
	{\mathcal D}iag &\simeq \{ {\mathbb P}( {\mathbb C} \oplus D/{\mathbb C} ) : \text{$D$ is a diagonal of ${\mathbb C}^n$} \} \quad \text{and} \\ 
	{\mathcal P}oly_H &\simeq \{ {\mathbb P}( 0 \oplus P/{\mathbb C} ) : \text{$P$ is a polydiagonal of ${\mathbb C}^n$, }\dim P \geq 2 \}.
	\end{aligned}
	\end{equation}
	Note that ${\mathbb P}^{n-1} \setminus H$ is naturally identified with $\cc^n / \cc$, so the complements of the building sets ${\mathcal P}oly_H$ and ${\mathcal D}iag \cup {\mathcal P}oly_H$ used to construct $W_n$ and $W'_n$ are naturally identified with $\cc^n/\cc$ and $\mathrm{F}(n,\cc)/\cc$ respectively. 

	Now we can also justify the claim that $W_n$ is an augmented wonderful variety. In the setup of \cite[Remark 2.13]{[BHMPW20a]}, let us take $V \subset \cc^{n \choose 2}$ cut out by 
	\[ x_{\{a,b\}} - x_{\{a,c\}} + x_{\{b,c\}} = 0 \text{ for all } 1 \leq a<b<c \leq n. \] 
	Clearly, $V \cong {\mathbb C}^n/{\mathbb C}$, with ${\mathbb C} \hookrightarrow {\mathbb C}^n$ diagonally. With the notation in loc. cit., $M$  is the graphic matroid of $K_n$, the linear subspaces $H_F$ correspond to elements of ${\mathcal P}oly_H$, and hence $X_V = W_n$. The specific setup considered essentially corresponds to representing $M$ by a positive root system of $\mathfrak{sl}_{n}$.
	
	By \cite[Theorem 1.3]{[Li09]}, $W'_n$ can be constructed by first blowing up the elements of ${\mathcal P}oly_H$ and then (the proper transforms of) the elements of ${\mathcal D}iag$, so there is a natural birational morphism $W'_n \to W_n$.
		
	\subsection{Further remarks and questions} To conclude the discussion of how the various spaces mentioned above relate to each other, we note the following.
	
	\begin{remark}\label{remark: flower curves}
		By Remark \ref{remark: relation to spaces 2} and Theorem \ref{theorem: small resolution theorem}, there are morphisms
		\begin{equation}\label{equation: factorization of resolution} W_n \xrightarrow{\gamma} \overline{P}_n \xrightarrow{\delta} \overline{\mathfrak t}_n, \end{equation}
	whose composition $W_n \to \overline{\mathfrak{t}}_n$ must equal the (special case of the) known resolution of the matroid Schubert variety by the augmented wonderful variety \cite{[HW17]}, since all spaces involved are compactifications of ${\mathbb C}^n/{\mathbb C}$ and all maps involved are the identity on these open subsets. Thus, Theorem \ref{theorem: small resolution theorem} provides a factorization of $W_n \to \overline{\mathfrak{t}}_n$ by inserting the moduli space $\overline{P}_n$ in the middle. 
	
	Similarly, for distinct marked points, there is a composition
	$$ W'_n \xrightarrow{\eta} \overline{Q}_n \to \overline{F}_n, $$ 
	where $\overline{F}_n$ is the space of cactus flower curves in \cite{[IKLPR23]}, cf. Remark \ref{remark: relation to spaces 2}.
	\end{remark} 
	
	\begin{remark}
	A further remark regarding the spaces in \eqref{equation: factorization of resolution}: both $\overline{\mathfrak t}_n$ and $\overline{P}_n$ deform to $\overline{L}_n$ (\cite{[IKLPR23]} and \cite{[Za21]}), but $W_n$ does not deform to $\overline{L}_n$ for $n \geq 4$, because $W_n$ and $\overline{L}_n$ are smooth but not diffeomorphic. For instance, it can be checked that $\mathrm{rank}\mathop{}H^2(W_n,\zz) = B_n-1 \neq 2^n-n-1 = \mathrm{rank}\mathop{} H^2(\overline{L}_n,\zz)$ for $n \geq 4$, where $B_n$ is the $n$-th Bell number. Thus, although $W_n$ fixes the singularities of both $\overline{\mathfrak t}_n$ and $\overline{P}_n$, it loses at least one of their interesting properties. 
	\end{remark}
	
	\begin{question}
		Are there factorizations similar to \eqref{equation: factorization of resolution} for other matroids?
	\end{question}
	
	Theorems \ref{theorem: small resolution theorem} and \ref{theorem: small resolution theorem 2} do not provide a clear method to read the local geometry of the small resolutions from combinatorial data. However, the case $n=4$ discussed in Example \ref{example: case n=4} below is quite simple.
	
	\begin{example}\label{example: case n=4}
		As explained at the end of \cite{[MW10]}, the threefold $\overline{Q}_4$ has $3$ ordinary double points, namely, the points where the curves look like this 
		\begin{center}
		\begin{tikzpicture}[scale = 0.5]
			\draw (0,0) circle (1); \draw[fill=white] (0,-1) circle (3pt) node[anchor = south] {$\infty$};
			\draw (-2,0) circle (1); 
			\draw (2,0) circle (1);
			\draw (3.3,0.5) circle (0.4); \draw[fill=black] (3.7,0.5) circle (3pt) node[anchor = west] {$k$};
			\draw (3.3,-0.5) circle (0.4); \draw[fill=black] (3.7,-0.5) circle (3pt) node[anchor = west] {$\ell$};
			\draw (-3.3,0.5) circle (0.4); \draw[fill=black] (-3.7,0.5) circle (3pt) node[anchor = east] {$i$};
			\draw (-3.3,-0.5) circle (0.4); \draw[fill=black] (-3.7,-0.5) circle (3pt) node[anchor = east] {$j$};
		\end{tikzpicture} 
		\end{center}
		cf. \cite[Figure 19]{[MW10]}, and, in fact, the same holds for $\overline{P}_4$. The exceptional set of $W_4 \to \overline{P}_4$ (respectively $W'_4 \to \overline{Q}_4$) is the intersection of the preimage by $W_4 \to {\mathbb P}^3$ (respectively $W'_4 \to {\mathbb P}^3$) of the $3$ points of the form $[\pm 1:\pm 1:\pm 1: \pm 1]$ in $H$ (the $3$ dots in Figure \ref{figure: 6 lines} which are intersections of only $2$ lines) with the proper transform of $H$. Our resolutions are locally a usual small resolution of $xy=zt$.
	\end{example}
		
		\begin{figure}[h]
			\begin{center}
				\begin{tiny}
					\begin{tikzpicture}
						\def\xa{-0.5}; \def\ya{0}; \def\xb{1}; \def\yb{0}; \def\xc{0}; \def\yc{1}; \def\xd{0.8}; \def\yd{0.8};
						
						\def\rr{-3}; \def\ss{2};
						\def\xp{\xa}; \def\yp{\ya}; \def\xq{\xb}; \def\yq{\yb};
						\draw  (\rr*\xp+\xq-\rr*\xq,\rr*\yp+\yq-\rr*\yq) -- (\ss*\xp+\xq-\ss*\xq,\ss*\yp+\yq-\ss*\yq);
						
						\def\rr{-1.5}; \def\ss{1.5};
						\def\xp{\xa}; \def\yp{\ya}; \def\xq{\xc}; \def\yq{\yc};
						\draw  (\rr*\xp+\xq-\rr*\xq,\rr*\yp+\yq-\rr*\yq) -- (\ss*\xp+\xq-\ss*\xq,\ss*\yp+\yq-\ss*\yq);
						
						\def\rr{-1}; \def\ss{1.5};
						\def\xp{\xa}; \def\yp{\ya}; \def\xq{\xd}; \def\yq{\yd};
						\draw  (\rr*\xp+\xq-\rr*\xq,\rr*\yp+\yq-\rr*\yq) -- (\ss*\xp+\xq-\ss*\xq,\ss*\yp+\yq-\ss*\yq);
						
						\def\rr{-1}; \def\ss{1.5};
						\def\xp{\xb}; \def\yp{\yb}; \def\xq{\xc}; \def\yq{\yc};
						\draw  (\rr*\xp+\xq-\rr*\xq,\rr*\yp+\yq-\rr*\yq) -- (\ss*\xp+\xq-\ss*\xq,\ss*\yp+\yq-\ss*\yq);
						
						\def\rr{-2}; \def\ss{1.5};
						\def\xp{\xb}; \def\yp{\yb}; \def\xq{\xd}; \def\yq{\yd};
						\draw  (\rr*\xp+\xq-\rr*\xq,\rr*\yp+\yq-\rr*\yq) -- (\ss*\xp+\xq-\ss*\xq,\ss*\yp+\yq-\ss*\yq);
						
						\def\rr{-5}; \def\ss{2};
						\def\xp{\xc}; \def\yp{\yc}; \def\xq{\xd}; \def\yq{\yd};
						\draw  (\rr*\xp+\xq-\rr*\xq,\rr*\yp+\yq-\rr*\yq) -- (\ss*\xp+\xq-\ss*\xq,\ss*\yp+\yq-\ss*\yq);
						
						\draw[fill=white] (\xa,\ya) circle (2pt) node[anchor = north] {$234|1$}; \draw[fill=white] (\xb,\yb) circle (2pt) node[anchor = north] {$134|2$}; \draw[fill=white] (\xc,\yc) circle (2pt) node[anchor = south] {$124|3$}; \draw[fill=white] (\xd,\yd) circle (2pt) node[anchor = south] {$123|4$}; 
						
						\draw[fill=white] (0.43,0.57) circle (2pt) node[anchor = east] {$14|23$}; \draw[fill=white] (0.5,2) circle (2pt) node[anchor = east] {$13|24$}; \draw[fill=white] (4,0) circle (2pt) node[anchor = south] {$12|34$};
						
						\def\px{7}; \def\py{0.5}; \def\ax{0}; \def\ay{2};  \def\bx{1}; \def\by{-1};
						
						\draw (\px,\py) -- (\px+\ax,\py+\ay) -- (\px+\ax+\bx,\py+\ay+\by) -- (\px+\bx,\py+\by) -- (\px,\py);
						
						\def\bx{\px-1}; \def\by{\py+0.5};
						
						\def\rx{0.3}; \def\ry{1.2}; 
						\draw (\bx,\by) -- node[anchor = south] {$124$} (\px+\rx,\py+\ry);
						\draw[fill=white] (\px+\rx,\py+\ry) circle (1pt);
						
						\def\rx{0.7}; \def\ry{0.5}; 
						\draw (\bx,\by) -- node[anchor = south] {$123$} (\px+\rx,\py+\ry);
						\draw[fill=white] (\px+\rx,\py+\ry) circle (1pt);
						
						\def\rx{0.2}; \def\ry{0.2}; 
						\draw (\bx,\by) -- node[anchor = north] {$234$} (\px+\rx,\py+\ry);
						\draw[fill=white] (\px+\rx,\py+\ry) circle (1pt);
						
						\def\rx{0.8}; \def\ry{-0.5}; 
						\draw (\bx,\by) -- node[anchor = north] {$134$} (\px+\rx,\py+\ry);
						\draw[fill=white] (\px+\rx,\py+\ry) circle (1pt);
						
						\draw[fill=white] (\bx,\by) circle (2pt) node[anchor = east] {$1234$};
						
						\node at (\px+1.2,\py-0.5) {$H$};
					\end{tikzpicture}
					\caption{For $n=4$: the arrangement ${\mathcal P}oly_H$ if $H$ is the screen or paper (left); the lines and point in ${\mathcal D}iag$ (right).}
					\label{figure: 6 lines}
				\end{tiny}
			\end{center}
		\end{figure}

	The geometry of the resolutions and their fibers will be studied in more detail in future work. For instance, the fibers of $\gamma:W_n \to \overline{P}_n$ will be shown to be non-trivial precisely over points of $\overline{P}_n$ where the corresponding curve is isomorphic to that in Example \ref{example: case n=4} (and the $n$ marked points are distributed among the $4$ terminal components in any possible way) plus further degenerations of such configurations. More notably, we expect the fibers of $\gamma$ to be isomorphic to fibers of the map from Ulyanov's polydiagonal compactification to the Fulton--MacPherson compactification \cite[\S6]{[Ul02]}, though fitting together in families in rather new ways. 
	
	\begin{remark}\label{remark: not IH small}
	In this paper, we are using the phrase `small resolution of singularities' in the classical sense (\S\ref{section: introduction}). There is a more restrictive notion of small resolution which, for instance, is used in intersection homology \cite[\S6.2]{[GM83]}. Our resolutions are not small in this sense if $n \geq 7$. This is proved in Remark \ref{remark: not IH small proof}.
	\end{remark}
		
	\section{Naive Fulton--MacPherson degeneration spaces}\label{section: naive FM spaces}
	
	\subsection{Naive FM spaces} Let $X$ be a smooth projective variety of dimension $d>0$. The constructions in this section are related to the Fulton--MacPherson compactification $X[n]$ and its `universal family' $X[n]^+$. The reader is referred to \cite{[FM94]} for their construction and properties.
	
	Kim, Kresh, and Oh introduced the notion of \emph{Fulton--MacPherson degeneration space} (or simply \emph{FM space}) of $X$ over an arbitrary base \cite[Definition 2.1.1]{[KKO14]} to construct an alternate compactification of the space of maps from curves to $X$. FM spaces are families which \'e{tale} locally on the base look like (pullbacks of) $X[n]^+ \to X[n]$ with the marked points ignored. Please see \cite[Figure 5]{[FM94]} or \cite[Figure 1]{[KKO14]} for pictures of the fibers. Such spaces will come up in \S\ref{section: polydiagonal degenerations}, though, for our purposes, we may work with the more elementary class of degenerations below.
	
	Let $S$ be a smooth quasi-projective variety. Naive FM spaces of $X$ over $S$ are obtained starting from the trivial family $S \times X \to S$, by repeatedly blowing up smooth sections of the family above smooth divisors on the base. (We say that a section of a morphism is \emph{smooth} if its image is contained in the open subset of the source of the morphism where the morphism is a smooth.) 
		
	\begin{definition}\label{definition: naive FM space}
		A pair of morphisms $(W \xrightarrow{\pi} S,W \xrightarrow{\xi} X)$ is a \emph{naive FM space} of $X$ over $S$ if there exists a sequence of birational morphisms
		\[ W = W_k \xrightarrow{\beta_k} W_{k-1} \xrightarrow{\beta_{k-1}} \cdots \xrightarrow{\beta_1} W_0 = S \times X \]
		and, for each $i=0,\ldots,k-1$, a smooth divisor $D_i \subset S$ and a smooth section $\sigma_i: S \to W_i$ of $W_i \to S$, such that $\beta_{i+1}$ is the blowup of $W_i$ at $\sigma_i(D_i)$, and $\pi$ and $\xi$ are the compositions of $\beta_1 \cdots \beta_k$ with the projections of $S \times X$ to its two factors.
	\end{definition}
	
	\begin{lemma}\label{lemma: only flatness}
		If $(W \xrightarrow{\pi} S, W \xrightarrow{\xi} X)$ is a naive FM space of $X$ over $S$, then $\pi$ is flat, and the fiber of $\pi$ over any ${\mathbb C}$-point of $S$ is isomorphic to a fiber of $X[n]^+ \to X[n]$ for some suitably chosen $n$.
	\end{lemma}
	
	\begin{proof}
		Continuing to use the notation in Definition \ref{definition: naive FM space}, let $\pi_i:W_i \to S$ be the composition of  $\beta_1 \cdots \beta_i$ with the projection $S \times X \to S$, for $i=0,\ldots,k$. We will write $W_{i,s} := \pi_i^{-1}(s)$, for any $s \in S$. Let ${\sh N}$ (respectively ${\sh N}'$) the normal bundle of the closed immersion by $\sigma_i$ of $D_i$ into $W_i$ (respectively $\sigma_i^{-1}(D_i)$), and ${\sh T} = {\sh T}_{W_i/S}$ the relative tangent sheaf of $W_i$ over $S$. Note that ${\sh N}' \simeq \sigma_i^*{\sh T}_{W_i/S}$, and we have a short exact sequence
		\begin{equation}\label{equation: ses which will split locally}
			0 \to \sigma_i^*{\sh T} \to {\sh N} \to {\sh O}_S(D_i)|_{D_i} \to 0.
		\end{equation}
		Recall that $W_{i+1}$ is the blowup of $W_i$ along $\sigma_i(D_i)$. Then, $\pi_{i+1}^{-1}(D_i)$ is the union of $R$, the proper transform of $\pi_i^{-1}(D_i)$ relative to $\beta_{i+1}$, or equivalently, its blowup along $\sigma_i(D_i)$, with the exceptional divisor $P = {\mathbb P}{\sh N}$ of $\beta_{i+1}$. Moreover, $P \cap R = E:= {\mathbb P}{\sh N}'$, the exceptional divisor of $R \to \pi_i^{-1}(D_i)$. 
		
		Then, for any $s \in D_i$, $W_{i+1,s}$ is the union of the blowup $\mathrm{Bl}_{\sigma_i(s)} W_{i,s}$ of $W_{i,s}$ at $\sigma_i(s)$ with the projective space $P_s$. More specifically, $W_{i+1,s}$ is the gluing (pushout) of $\mathrm{Bl}_{\sigma_i(s)} W_{i,s}$ with $P_s$ along $E_s$, with $E_s$ embedded inside $\mathrm{Bl}_{\sigma_i(s)} W_{i,s}$ as the exceptional divisor, and inside $P_s$ as the hyperplane $ {\mathbb P}{\sh N}'_s \subset P_s$. On the other hand, if $s \notin D_i$, we obviously have $W_{i+1,s} \cong W_{i,s}$.  
		Then, using the description of the fibers of $X[n]^+ \to X[n]$ in \cite[p. 194]{[FM94]}, it follows by induction on $i$ that the fibers of $W_i \to S$ are isomorphic to such fibers, for suitably chosen $n$.
		
		Finally, the flatness of $\pi$ follows from the result above concerning the fibers and the `miracle flatness' theorem \cite[Theorem 23.1]{[Ma87]}.
	\end{proof}
	
	\subsection{The root relative dualizing sheaf} If $(W \xrightarrow{\pi} S, W \xrightarrow{\xi} X)$ is a naive FM space, then it can be checked inductively with standard arguments that
	\begin{equation} \Pic (W) = \Pic (S \times X) \oplus F_{W,X/S}, \end{equation}
	where $F_{W,X/S}$ is a free abelian group freely generated by the irreducible components of the exceptional locus of $W \to S \times X$, and $\Pic(S \times X)$ is identified with its image by $(\pi,\xi)^*$. Note also that $\pi$ is lci because $W$ and $S$ are smooth varieties.
	
	\begin{lemma}\label{lemma: defining property of root}
		For any naive FM space $(W \xrightarrow{\pi} S, W \xrightarrow{\xi} X)$, there exists a unique line bundle $\sqrt[d]{\omega}_{W,X/S} \in F_{W,X/S}$ such that 
		\begin{equation}\label{equation: defining equation of root repeated}  \omega_{W/S} = \xi^*\omega_X \otimes \sqrt[d]{\omega}_{W,X/S}^{\otimes d} . \end{equation}
		Moreover, if $(W \xrightarrow{\pi} S, W \xrightarrow{\xi} X)$ is a naive FM space, $D \subset S$ is a smooth divisor, $\sigma: S \to W$ is a smooth section of $\pi$, and $\beta:W' \to W$ is the blowup of $W$ at $\sigma(D)$ with exceptional divisor $E$, then $(W' \xrightarrow{\pi \beta} S, W' \xrightarrow{\xi \beta} X)$ is a naive FM space, and
		\begin{equation}\label{equation: inductive definition of root}
			\sqrt[d]{\omega}_{W',X/S} = {\sh O}_{W'}(E) \otimes \beta^* \sqrt[d]{\omega}_{W,X/S}.
		\end{equation}
	\end{lemma}
	
	\begin{proof} Uniqueness is clear, since $F_{W,X/S}$ is torsion-free. To prove existence, we turn $\sqrt[d]{\omega}_{S \times X,X/S} = {\sh O}_{S \times X}$ and \eqref{equation: inductive definition of root} into the inductive definition of the desired line bundle, and check \eqref{equation: defining equation of root repeated}. Since ${\sh O}_{W'}(E) \in F_{W',X/S}$ and $\beta^*(F_{W,X/S}) \subseteq F_{W',X/S}$, we have $\sqrt[d]{\omega}_{W',X/S} \in F_{W',X/S}$. By the well-known 
	$ \omega_W = \omega_{W/S} \otimes \pi^*\omega_S$ and $\omega_{S \times X} = \omega_S \boxtimes \omega_X$,
	formula \eqref{equation: defining equation of root repeated} can be rearranged equivalently as $\omega_W = (\pi,\xi)^*\omega_{S \times X} \otimes \sqrt[d]{\omega}_{W,X/S}^{\otimes d}$, which follows inductively from \eqref{equation: inductive definition of root} and $K_{W'} = \beta^*K_W + dE$. 
	\end{proof}
	
	\begin{definition}
	The line bundle $\sqrt[d]{\omega}_{W,X/S}$ in Lemma \ref{lemma: defining property of root} is the \emph{root relative dualizing sheaf} of $(W \to S, W \to X)$. If $d=1$, we write $\omega_{W,X/S}$ instead of $\sqrt[1]{\omega}_{W,X/S}$.
	\end{definition}
	
	In some sense, $\sqrt[d]{\omega}$ is simply the line bundle associated to the exceptional divisor of $W \to S \times X$, though components of the exceptional locus often need to be taken with multiplicities greater than 1. (There is also a distinguished section $\epsilon$ of $\sqrt[d]{\omega}$. We will never use it in any arguments, though we will make another parenthetical reference to it later.) It is almost certain that the root relative dualizing sheaf can be defined for FM spaces in the sense of \cite{[KKO14]}, although the required verifications seem not entirely trivial, and will be discussed elsewhere. 
		
	\begin{lemma}\label{lemma: lifting vector fields naively}
		Let $(W \xrightarrow{\pi} S, W \xrightarrow{\xi} X)$ be a naive FM space, $D \subset S$ a smooth divisor, $\sigma: S \to W$ a smooth section of $\pi$, and $\beta:W' \to W$ the blowup of $W$ at $\sigma(D)$ with exceptional divisor $E$. Then, for any
		$$ \psi \in \mathrm{Hom}_{{\sh O}_W} \left( \sqrt[d]{\omega}_{W,X/S}, {\sh I}_{\sigma(D),W} \right), $$ 
		there exists a unique
		$$ \psi' \in \mathrm{Hom}_{{\sh O}_{W'}} \left( \sqrt[d]{\omega}_{W',X/S}, {\sh O}_{W'} \right) $$ 
		such that the restrictions of $\psi$ and $\psi'$ to $W \backslash \sigma(D)$ and $W' \backslash E$ respectively map to each other under the isomorphism $W \backslash \sigma(D) \cong W' \backslash E$. Moreover, the restriction of $\psi'$ to $E$ is equal to $0$ if and only if $\mathrm{Im}(\psi) \subseteq {\sh I}_{\sigma(D),W}^2$.
	\end{lemma}
	
	\begin{proof}
		The lemma follows from \eqref{equation: inductive definition of root} and Corollary \ref{corollary: twisting for blowup}.
	\end{proof}
	
	\section{Combinatorial language}\label{section: combinatorial language}
	
	This section is a summary of the combinatorial notation used in \S\ref{section: polydiagonal degenerations} and \S\ref{section: small resolutions}.
		
	\subsection{Partitions and trees}\label{subsection: partitions and trees} 
	Our account is very similar to that in \cite[\S2]{[Ul02]}, although some small differences are significant in practice. 
	
	Let $L_n$ be the lattice of partitions of $[n] = \{1,\ldots,n\}$, with the (inverse, according to some conventions, though we will follow \cite{[Ul02]}) refinement partial ordering $\rho_1 \leq \rho_2$ if each block of $\rho_2$ is contained in a block of $\rho_1$. Then 
	$\bot = 12\cdots n = \min L_n$ and $\top = 1|2|\cdots|n = \max L_n$. For each $\rho \in L_n$, $\sim_\rho$ denotes the corresponding equivalence relation on $[n]$, that is, $i \sim_\rho j$ if $i$ and $j$ belong to the same block of $\rho$, and $\blocks(\rho)$ is the set of blocks of $\rho$. For $\rho_1,\rho_2 \in L_n$, $\rho_1 \vee \rho_2$ and $\rho_1 \wedge \rho_2$ are the join and the meet respectively, for instance,  $12|34|56 \vee 123|456 = 12|3|4|56$ and $12|34|56 \wedge 123|456 = 123456$. 
	
	\begin{remark}\label{remark: bijection of complement}
	The partitions of $[n+1]$ which have $\{n+1\}$ as a block are in bijection with $L_n$, so there is a natural injective function 
	$\iota:L_n \hookrightarrow L_{n+1}$, and 
	\begin{equation}  L_{n+1} \backslash L_n \simeq U_n := \{(\rho,B):\rho \in L_n \text{ and } B \in \blocks(\rho)\}. \end{equation}
	By abuse of notation, we will sometimes omit $\iota$, as we did above.
	\end{remark}
	
	A \emph{chain} in a poset is a subset consisting of mutually comparable elements. As in \cite{[Ul02]}, the interplay between chains in $L_n$ and leveled trees is essential.
	
	\begin{definition}\label{definition: leveled tree representations}
		Let ${\mathcal H}=\{\rho_1 < \rho_2 < \cdots < \rho_k\} \neq \emptyset$ be a chain in $L_n$.
		
		The \emph{leveled tree representation of ${\mathcal H}$ with phantom vertices included} $T({\mathcal H})$ is a rooted tree whose root $\star$ is the only level-$0$ vertex, whose set of level-$i$ vertices is $\blocks(\rho_i)$, and in which edges correspond to inclusions of blocks on consecutive levels (and all level-$1$ vertices are adjacent to the root). Level-$k$ vertices are leaves. Finally, to each leaf $B \in \blocks(\rho_k)$ we attach legs (half-edges) indexed by the elements of $B$.
		
		The \emph{leveled tree representation of ${\mathcal H}$ with phantom vertices excluded} $T^s({\mathcal H})$ is obtained from $T({\mathcal H})$, by replacing each maximal chain in $T^s({\mathcal H})$ with internal vertices of degree $2$ (and vertex endpoints -- leg ends are forbidden), which doesn't contain $\star$ as an internal vertex, with an edge and erasing the internal vertices.	
	\end{definition}
	
	\begin{figure}[h]
		\begin{center}
			\begin{tiny}
				\begin{tikzpicture}[scale=0.5]
					\node at (-1,-5.3) {\normalsize{$T({\mathcal H})$}};
					\def\rr{3pt}
					\fill[black] (-1,0) circle (\rr) node[anchor = south] {$\star$};
					\fill[black] (-3,-1) circle (\rr) node[anchor = south] {$123$};
					\draw (-3,-1) -- (-1,0);
					\fill[black] (-4,-2) circle (\rr) node[anchor = east] {$12$};
					\draw (-4,-2) -- (-3,-1);
					\fill[black] (-2,-2) circle (\rr) node[anchor = east] {$3$};
					\draw (-2,-2) -- (-3,-1);
					\fill[black] (-4,-3) circle (\rr) node[anchor = east] {$12$};
					\draw (-4,-3) -- (-4,-2);
					\fill[black] (-2,-3) circle (\rr) node[anchor = east] {$3$};
					\draw (-2,-3) -- (-2,-2);
					\fill[black] (-4.5,-4) circle (0pt) node[anchor = north] {$1$};
					\draw (-4.5,-4) -- (-4,-3);
					\fill[black] (-3.5,-4) circle (0pt) node[anchor = north] {$2$};
					\draw (-3.5,-4) -- (-4,-3);
					\fill[black] (-2,-4) circle (0pt) node[anchor = north] {$3$};
					\draw (-2,-4) -- (-2,-3);
					
					\fill[black] (1,-1) circle (\rr) node[anchor = south] {$45678$};
					\draw (1,-1) -- (-1,0);
					\fill[black] (-0.5,-2) circle (\rr) node[anchor = east] {$45$};
					\draw (-0.5,-2) -- (1,-1);
					
					\fill[black] (-0.5,-3) circle (\rr) node[anchor = east] {$45$};
					\draw (-0.5,-3) -- (-0.5,-2);
					
					\fill[black] (-1,-4) circle (0pt) node[anchor = north] {$4$};
					\draw (-1,-4) -- (-0.5,-3);
					\fill[black] (0,-4) circle (0pt) node[anchor = north] {$5$};
					\draw (0,-4) -- (-0.5,-3);
					
					\fill[black] (1,-2) circle (\rr) node[anchor = east] {$67$};
					\draw (1,-2) -- (1,-1);
					
					\fill[black] (0.5,-3) circle (\rr) node[anchor = east] {$6$};
					\draw (0.5,-3) -- (1,-2);
					\fill[black] (0.5,-4) circle (0pt) node[anchor = north] {$6$};
					\draw (0.5,-4) -- (0.5,-3);
					\fill[black] (1.5,-3) circle (\rr) node[anchor = west] {$7$};
					\draw (1.5,-3) -- (1,-2);
					\fill[black] (1.5,-4) circle (0pt) node[anchor = north] {$7$};
					\draw (1.5,-4) -- (1.5,-3);
					
					\fill[black] (2.5,-2) circle (\rr) node[anchor = south] {$8$};
					\draw (2.5,-2) -- (1,-1);
					
					\fill[black] (2.5,-3) circle (\rr) node[anchor = west] {$8$};
					\draw (2.5,-3) -- (2.5,-2);
					
					\fill[black] (2.5,-4) circle (0pt) node[anchor = north] {$8$};
					\draw (2.5,-4) -- (2.5,-3);
					
					\def\aa{9}
					
					\node at (\aa-1,-5.3) {\normalsize{$T^s({\mathcal H})$}};
					
					\fill[black] (\aa-1,0) circle (\rr) node[anchor = south] {$\star$};
					\fill[black] (\aa-3,-1) circle (\rr) node[anchor = south] {$123$};
					\draw (\aa-3,-1) -- (\aa-1,0);
					\fill[black] (\aa-4,-3) circle (\rr) node[anchor = east] {$12$};
					\draw (\aa-4,-3) -- (\aa-3,-1);
					\fill[black] (\aa-2,-3) circle (\rr) node[anchor = east] {$3$};
					\draw (\aa-2,-3) -- (\aa-3,-1);
					\fill[black] (\aa-4.5,-4) circle (0pt) node[anchor = north] {$1$};
					\draw (\aa-4.5,-4) -- (\aa-4,-3);
					\fill[black] (\aa-3.5,-4) circle (0pt) node[anchor = north] {$2$};
					\draw (\aa-3.5,-4) -- (\aa-4,-3);
					\fill[black] (\aa-2,-4) circle (0pt) node[anchor = north] {$3$};
					\draw (\aa-2,-4) -- (\aa-2,-3);
					
					\fill[black] (\aa+1,-1) circle (\rr) node[anchor = south] {$45678$};
					\draw (\aa+1,-1) -- (\aa-1,0);
					
					\fill[black] (\aa-0.5,-3) circle (\rr) node[anchor = east] {$45$};
					\draw (\aa-0.5,-3) -- (\aa+1,-1);
					
					\fill[black] (\aa-1,-4) circle (0pt) node[anchor = north] {$4$};
					\draw (\aa-1,-4) -- (\aa-0.5,-3);
					\fill[black] (\aa+0,-4) circle (0pt) node[anchor = north] {$5$};
					\draw (\aa+0,-4) -- (\aa-0.5,-3);
					
					\fill[black] (\aa+1,-2) circle (\rr) node[anchor = east] {$67$};
					\draw (\aa+1,-2) -- (\aa+1,-1);
					
					\fill[black] (\aa+0.5,-3) circle (\rr) node[anchor = east] {$6$};
					\draw (\aa+0.5,-3) -- (\aa+1,-2);
					\fill[black] (\aa+0.5,-4) circle (0pt) node[anchor = north] {$6$};
					\draw (\aa+0.5,-4) -- (\aa+0.5,-3);
					\fill[black] (\aa+1.5,-3) circle (\rr) node[anchor = west] {$7$};
					\draw (\aa+1.5,-3) -- (\aa+1,-2);
					\fill[black] (\aa+1.5,-4) circle (0pt) node[anchor = north] {$7$};
					\draw (\aa+1.5,-4) -- (\aa+1.5,-3);
					
					\fill[black] (\aa+2.5,-3) circle (\rr) node[anchor = west] {$8$};
					\draw (\aa+2.5,-3) -- (\aa+1,-1);
					
					\fill[black] (\aa+2.5,-4) circle (0pt) node[anchor = north] {$8$};
					\draw (\aa+2.5,-4) -- (\aa+2.5,-3);
					
				\end{tikzpicture}
				\caption{${\mathcal H}= \{123|45678,12|3|45|67|8,12|3|45|6|7|8\}$. We are \emph{not} contracting all the way to the leg (visible for legs $3,6,7,8$), in contrast to \cite[Figure 7]{[Ul02]}. Instead, we contract only up to the corresponding leaf, even when there is a single leg attached to it.}
				\label{figure: trees figure}
			\end{tiny}
		\end{center}
	\end{figure}
	
	Later in the paper, we will see that $T({\mathcal H})$ relates to $W_n$ and $T^s({\mathcal H})$ relates to $\overline{P}_n$ (Remark \ref{remark: relation between W curve and P curve}). In the context of $W_n$ and/or $\overline{P}_n$, only chains ${\mathcal H}$ which contain the partition $\bot$ will occur. If $\bot \in {\mathcal H}$, then the root vertex $\star$ in $T({\mathcal H})$ (respectively $T^s({\mathcal H})$) has degree $1$, so removing it gives a tree $T({\mathcal H}) - \star$ (respectively $T^s({\mathcal H}) - \star$), in which we will regard the unique level 1 vertex to be the new root.
	
	We state here two ad hoc combinatorial lemmas needed in \S\ref{section: small resolutions}. 
	
	\begin{lemma}\label{lemma: spider has short legs}
	Let ${\mathcal H} \subseteq L_n$ be a chain such that $\bot \in {\mathcal H}$ and $|{\mathcal H}| \geq 2$. If $T^s({\mathcal H})-\star$ is a star after erasing the legs (that is, $K_{1,m}$ ignoring the legs) with the center adjacent to $\star$ in $T^s({\mathcal H})$, then $|{\mathcal H}|=2$ and $T^s({\mathcal H})=T({\mathcal H})$.
	\end{lemma}
	
	\begin{proof}
		The number of non-root, non-leaf vertices of degree larger than $2$ is the same in $T({\mathcal H})$ and $T^s({\mathcal H})$, so it's at most $1$ in $T({\mathcal H})$ too. Every level of $T({\mathcal H})$ except level $0$ and level $|{\mathcal H}|$ contains at least one such vertex, so $|{\mathcal H}|-1 \leq 1$.
	\end{proof}
	
	\begin{lemma}\label{lemma: binary tree}
		Let $\rho_j$ be the partition of $[n]$ into the fibers of the function $[n] \to \zz$ given by $x \mapsto \left\lfloor 2^{-j}(x-1) \right\rfloor$. Let $k=\lceil \log_2 n \rceil$ and ${\mathcal H} =\{\bot = \rho_k, \ldots, \rho_0\}$. Then,
		\begin{enumerate}
			\item $T^s({\mathcal H}) - \star$ is binary with one leg attached to each leaf; and
			\item $T^s({\mathcal H}') = T^s({\mathcal H})$, for any chain ${\mathcal H}' \supseteq {\mathcal H}$ in $L_n$.
		\end{enumerate}
	\end{lemma}
	
	\begin{proof}
		The proof is entirely elementary.
	\end{proof}
	
	\subsection{Nested sets}\label{subsection: Fulton--MacPherson nests} A family of sets is \emph{nested} if any two of them are either disjoint or comparable, that is, one is contained in the other. The following fact is well-known and implicit in \cite{[FM94]}, so we will regard it as a definition.
	
	\begin{definition}
		Given a finite set $S$ and a nested family ${\mathcal F}$ of subsets of $S$ with at least two elements each, there exists a unique (up to isomorphism) rooted tree $R_S({\mathcal F})$ with $|{\mathcal F}|+1$ vertices and legs indexed by $S$, such that ${\mathcal F}$ coincides with the family of subsets of $S$ which consist of all legs attached to vertices that are descendants of (or equal to) a given non-root vertex. Moreover, if $v$ if a non-root vertex of $R_S({\mathcal F})$, the total number of edges and legs with endpoint $v$ is at least $3$.
	\end{definition}
	
	When $S$ is clear from context, we will write $R({\mathcal F})$ instead of $R_S({\mathcal F})$.
	
	Let $Z_n$ be the set of subsets of $[n]$ with at least two elements ($|Z_n|=2^n-n-1$), and
	$ Z_{n+1}^* = Z_{n+1} \backslash \{\{1,n+1\},\ldots,\{n,n+1\}\}$. 
	The families $Z_n$ and $Z_{n+1}^*$ come up in the construction of the Fulton--MacPherson space \cite{[FM94]}, where they index the blowups required to obtain $X[n]$ and its universal family $X[n]^+$ respectively.
	
	We can combine the objects above and those in \S\ref{subsection: partitions and trees} as follows. A pair $({\mathcal H},{\mathcal F})$ consisting of a chain ${\mathcal H} \subseteq L_n$ and a nested family ${\mathcal F} \subseteq Z_n$ is \emph{comparable} if each member of ${\mathcal F}$ is contained in some block of $\max {\mathcal H}$. To each comparable pair $({\mathcal H},{\mathcal F})$, we associate rooted trees $T({\mathcal H},{\mathcal F})$ and $T^s({\mathcal H},{\mathcal F})$ as follows. 
	
	\begin{definition}\label{definition: combined leveled tree representations}
		Let $({\mathcal H},{\mathcal F})$ be a comparable pair as above. For each $B \in \blocks(\max {\mathcal H})$, let 
		$ {\mathcal F}|_B = \{ S \in {\mathcal F}: S \subseteq B \}$, and, in the rooted tree $T({\mathcal H})$ (respectively $T^s({\mathcal H})$) in Definition \ref{definition: leveled tree representations}, we replace the leaf corresponding to $B$ and the legs attached to this leaf with the rooted tree $R_B({\mathcal F}|_B)$, in such a way that the old leaf is replaced with the root of $R_B({\mathcal F}|_B)$. The resulting rooted tree with legs is denoted $T({\mathcal H},{\mathcal F})$ (respectively $T^s({\mathcal H},{\mathcal F})$).
	\end{definition}
	
	A not-root vertex of $T({\mathcal H},{\mathcal F})$ (respectively $T^s({\mathcal H},{\mathcal F})$) has \emph{zero scaling} if it does not come from a vertex of $T({\mathcal H})$ (resp. $T^s({\mathcal H},{\mathcal F})$), \emph{finite scaling} if it corresponds to a leaf of $T({\mathcal H})$ (resp. $T^s({\mathcal H},{\mathcal F})$), or \emph{infinite scaling} if it corresponds to a non-leaf vertex of $T({\mathcal H})$ (resp. $T^s({\mathcal H},{\mathcal F})$). 
	
	\begin{example} The case ${\mathcal H}=\{1234|5678,13|24|5678\}$, ${\mathcal F} = \{\{1,3\},\{6,7,8\},\{6,7\}\}$ is shown below. 
		\begin{center}
			\begin{tiny}
				\begin{tikzpicture}[scale=0.5]
					\def\rr{3pt}
					\def\xx{0};
					
					\fill[black] (-1+\xx,0) circle (\rr) node[anchor = south] {$\star$};
					\fill[black] (-3+\xx,-1) circle (\rr) node[anchor = south] {$1234$};
					\fill[black] (1+\xx,-1) circle (\rr) node[anchor = south] {$5678$};		
					\draw (-3+\xx,-1) -- (-1+\xx,0) -- (1+\xx,-1);
					
					\fill[black] (-4+\xx,-2) circle (\rr) node[anchor = east] {$13$};
					\fill[black] (-2+\xx,-2) circle (\rr) node[anchor = east] {$24$};
					\draw (-2+\xx,-2) -- (-3+\xx,-1) -- (-4+\xx,-2);
					
					\fill[black] (-4.5+\xx,-3) circle (0pt) node[anchor = north] {$1$};
					\fill[black] (-3.5+\xx,-3) circle (0pt) node[anchor = north] {$3$};
					\draw (-3.5+\xx,-3) -- (-4+\xx,-2) -- (-4.5+\xx,-3);
					
					\draw[dashed] (-4+\xx,-2.5) circle (0.9);
					
					\def\ax{0}; \def\ay{-3};
					
					\fill[black] (-4+\xx,-2+\ay) circle (\rr) node[anchor = south] {$R(\{1,3\})$};
					\fill[black] (-4+\xx,-3+\ay) circle (\rr);
					\draw (-4+\xx,-3+\ay) -- (-4+\xx,-2+\ay);
					\fill[black] (-4.5+\xx,-4+\ay) circle (0pt) node[anchor = north] {$1$};
					\fill[black] (-3.5+\xx,-4+\ay) circle (0pt) node[anchor = north] {$3$};
					\draw (-3.5+\xx,-4+\ay) -- (-4+\xx,-3+\ay) -- (-4.5+\xx,-4+\ay);

					\fill[black] (-2.5+\xx,-3) circle (0pt) node[anchor = north] {$2$};
					\fill[black] (-1.5+\xx,-3) circle (0pt) node[anchor = north] {$4$};
					\draw (-1.5+\xx,-3) -- (-2+\xx,-2) -- (-2.5+\xx,-3);
					
					\draw[dashed] (-2+\xx,-2.5) circle (0.9);

					\fill[black] (-2+\xx,-2+\ay) circle (\rr) node[anchor = south] {$R(\emptyset)$};
					\fill[black] (-2.5+\xx,-3+\ay) circle (0pt) node[anchor = north] {$2$};
					\fill[black] (-1.5+\xx,-3+\ay) circle (0pt) node[anchor = north] {$4$};
					\draw (-1.5+\xx,-3+\ay) -- (-2+\xx,-2+\ay) -- (-2.5+\xx,-3+\ay);

					\fill[black] (1+\xx,-2) circle (\rr) node[anchor = east] {$5678$};
					\draw (1+\xx,-2) -- (1+\xx,-1);
					
					\draw[dashed] (1+\xx,-2.8) circle (1.2);
					
					\fill[black] (0+\xx,-3) circle (0pt) node[anchor = north] {$5$}; \draw (0+\xx,-3) -- (1+\xx,-2);
					\fill[black] (0.6+\xx,-3) circle (0pt) node[anchor = north] {$6$}; \draw (0.6+\xx,-3) -- (1+\xx,-2);
					\fill[black] (1.4+\xx,-3) circle (0pt) node[anchor = north] {$7$}; \draw (1.4+\xx,-3) -- (1+\xx,-2);
					\fill[black] (2+\xx,-3) circle (0pt) node[anchor = north] {$8$}; \draw (2+\xx,-3) -- (1+\xx,-2);
					
					\fill[black] (1+\xx,-2+\ay) circle (\rr) node[anchor = south] {$R(\{6,7,8\},\{6,7\})$};	
					\fill[black] (0+\xx,-3+\ay) circle (0pt) node[anchor = north] {$5$}; \draw (0+\xx,-3+\ay)--(1+\xx,-2+\ay);
					\fill[black] (2+\xx,-3+\ay) circle (\rr); \draw (2+\xx,-3+\ay)--(1+\xx,-2+\ay);
					\fill[black] (3+\xx,-4+\ay) circle (0pt) node[anchor = north] {$8$}; \draw (2+\xx,-3+\ay)--(3+\xx,-4+\ay);
					\fill[black] (1+\xx,-4+\ay) circle (\rr); \draw (2+\xx,-3+\ay)--(1+\xx,-4+\ay);
					
					\fill[black] (0+\xx,-5+\ay) circle (0pt) node[anchor = north] {$6$}; \draw (0+\xx,-5+\ay)--(1+\xx,-4+\ay);
					\fill[black] (2+\xx,-5+\ay) circle (0pt) node[anchor = north] {$7$}; \draw (2+\xx,-5+\ay)--(1+\xx,-4+\ay);
					
					\def\xx{9};
					
					\fill[black] (-1+\xx,0) circle (\rr) node[anchor = south] {$\star$};
					\fill[black] (-3+\xx,-1) circle (\rr);
					\fill[black] (1+\xx,-1) circle (\rr);		
					\draw (-3+\xx,-1) -- (-1+\xx,0) -- (1+\xx,-1);
					
					\fill[black] (-4+\xx,-2) circle (\rr);
					\fill[black] (-2+\xx,-2) circle (\rr);
					\draw (-2+\xx,-2) -- (-3+\xx,-1) -- (-4+\xx,-2);
					
					\def\ay{0};
					
					\fill[black] (-4+\xx,-2+\ay) circle (\rr);
					\fill[black] (-4+\xx,-3+\ay) circle (\rr);
					\draw (-4+\xx,-3+\ay) -- (-4+\xx,-2+\ay);
					\fill[black] (-4.5+\xx,-4+\ay) circle (0pt) node[anchor = north] {$1$};
					\fill[black] (-3.5+\xx,-4+\ay) circle (0pt) node[anchor = north] {$3$};
					\draw (-3.5+\xx,-4+\ay) -- (-4+\xx,-3+\ay) -- (-4.5+\xx,-4+\ay);

					\fill[black] (-2.5+\xx,-3) circle (0pt) node[anchor = north] {$2$};
					\fill[black] (-1.5+\xx,-3) circle (0pt) node[anchor = north] {$4$};
					\draw (-1.5+\xx,-3) -- (-2+\xx,-2) -- (-2.5+\xx,-3);

					\fill[black] (-2+\xx,-2+\ay) circle (\rr);
					
					\fill[black] (1+\xx,-2) circle (\rr);
					\draw (1+\xx,-2) -- (1+\xx,-1);
					
					\fill[black] (1+\xx,-2+\ay) circle (\rr);	
					\fill[black] (0+\xx,-3+\ay) circle (0pt) node[anchor = north] {$5$}; \draw (0+\xx,-3+\ay)--(1+\xx,-2+\ay);
					\fill[black] (2+\xx,-3+\ay) circle (\rr); \draw (2+\xx,-3+\ay)--(1+\xx,-2+\ay);
					\fill[black] (3+\xx,-4+\ay) circle (0pt) node[anchor = north] {$8$}; \draw (2+\xx,-3+\ay)--(3+\xx,-4+\ay);
					\fill[black] (1+\xx,-4+\ay) circle (\rr); \draw (2+\xx,-3+\ay)--(1+\xx,-4+\ay);
					
					\fill[black] (0+\xx,-5+\ay) circle (0pt) node[anchor = north] {$6$}; \draw (0+\xx,-5+\ay)--(1+\xx,-4+\ay);
					\fill[black] (2+\xx,-5+\ay) circle (0pt) node[anchor = north] {$7$}; \draw (2+\xx,-5+\ay)--(1+\xx,-4+\ay);
					
					\draw [dashed] (4,-2.5) -- (13,-2.5); \draw [dashed] (4,-1.5) -- (13,-1.5);
					
					\node at (13,-0.5) {$\infty$};
					\node at (13,-2) {finite};
					\node at (13,-3.5) {$0$};
					
					\node at (9,-6.5) {\normalsize{$T({\mathcal H},{\mathcal F})$}};
					
					\def\xx{19};
					
					\fill[black] (-1+\xx,0) circle (\rr) node[anchor = south] {$\star$};
					\fill[black] (-3+\xx,-1) circle (\rr);
					\draw (-3+\xx,-1) -- (-1+\xx,0);
					
					\fill[black] (-4+\xx,-2) circle (\rr);
					\fill[black] (-2+\xx,-2) circle (\rr);
					\draw (-2+\xx,-2) -- (-3+\xx,-1) -- (-4+\xx,-2);
					
					\def\ay{0};
					
					\fill[black] (-4+\xx,-2+\ay) circle (\rr);
					\fill[black] (-4+\xx,-3+\ay) circle (\rr);
					\draw (-4+\xx,-3+\ay) -- (-4+\xx,-2+\ay);
					\fill[black] (-4.5+\xx,-4+\ay) circle (0pt) node[anchor = north] {$1$};
					\fill[black] (-3.5+\xx,-4+\ay) circle (0pt) node[anchor = north] {$3$};
					\draw (-3.5+\xx,-4+\ay) -- (-4+\xx,-3+\ay) -- (-4.5+\xx,-4+\ay);

					\fill[black] (-2.5+\xx,-3) circle (0pt) node[anchor = north] {$2$};
					\fill[black] (-1.5+\xx,-3) circle (0pt) node[anchor = north] {$4$};
					\draw (-1.5+\xx,-3) -- (-2+\xx,-2) -- (-2.5+\xx,-3);

					\fill[black] (-2+\xx,-2+\ay) circle (\rr);
					
					\fill[black] (1+\xx,-2) circle (\rr);
					\draw (1+\xx,-2) -- (\xx-1,0);
					
					\fill[black] (1+\xx,-2+\ay) circle (\rr);	
					\fill[black] (0+\xx,-3+\ay) circle (0pt) node[anchor = north] {$5$}; \draw (0+\xx,-3+\ay)--(1+\xx,-2+\ay);
					\fill[black] (2+\xx,-3+\ay) circle (\rr); \draw (2+\xx,-3+\ay)--(1+\xx,-2+\ay);
					\fill[black] (3+\xx,-4+\ay) circle (0pt) node[anchor = north] {$8$}; \draw (2+\xx,-3+\ay)--(3+\xx,-4+\ay);
					\fill[black] (1+\xx,-4+\ay) circle (\rr); \draw (2+\xx,-3+\ay)--(1+\xx,-4+\ay);
					
					\fill[black] (0+\xx,-5+\ay) circle (0pt) node[anchor = north] {$6$}; \draw (0+\xx,-5+\ay)--(1+\xx,-4+\ay);
					\fill[black] (2+\xx,-5+\ay) circle (0pt) node[anchor = north] {$7$}; \draw (2+\xx,-5+\ay)--(1+\xx,-4+\ay);
					
					\draw [dashed] (14,-2.5) -- (23,-2.5); \draw [dashed] (14,-1.5) -- (23,-1.5);
					
					\node at (23,-0.5) {$\infty$};
					\node at (23,-2) {finite};
					\node at (23,-3.5) {$0$};
					
					\node at (19,-6.5) {\normalsize{$T^s({\mathcal H},{\mathcal F})$}};
					
				\end{tikzpicture}
			\end{tiny}
		\end{center}
	\end{example}
	
	\subsection{Enumerations}\label{subsection: enumerations} 
		The sets and partitions discussed above will index certain sequences of blowups in the constructions in \S\ref{section: polydiagonal degenerations}. Thus, order is important, so let us fix notation for some frequently used enumerations. In \S\ref{section: polydiagonal degenerations}, we will use a graph $\graph$ with $n$ vertices, which is either complete or has no edges.
			\begin{enumerate}
				\item[${\mathfrak l}$] is an increasing enumeration $\bot,\ldots,\top$ of $L_n$: first the partition with $1$ block, then the partitions with $2$ blocks, etc.;
				\item[${\mathfrak l}^\star$] is a corresponding enumeration of $L_{n+1} \backslash L_n$ by Remark \ref{remark: bijection of complement}: first all pairs whose first entry is the first partition in ${\mathfrak l}$, then all pairs whose first entry is the second partition in ${\mathfrak l}$, etc.;
				\item[${\mathfrak L}$] is an increasing enumeration of $L_{n+1}$;
				\item[${\mathfrak e}$] is a Fulton--MacPherson enumeration \cite[page 196]{[FM94]} of $Z_n$ if $\graph$ is complete, respectively empty if $\graph$ has no edges;
				\item[${\mathfrak e}^\star$] is a Fulton--MacPherson enumeration of $ Z_{n+1}^* \backslash Z_n$
				if $\graph$ is complete (which essentially means non-increasing cardinalities), respectively empty if $\graph$ has no edges.
			\end{enumerate}

		\section{Polydiagonal degenerations}\label{section: polydiagonal degenerations}
		
		Throughout \S\ref{section: polydiagonal degenerations}, $X$ is a smooth projective variety of dimension $d > 0$. For each partition $\rho \in L_n$, we have a \emph{polydiagonal} of $X^n$
	\begin{equation*} \Delta_\rho = \{(x_1,\ldots,x_n) \in X^n: \text{ $x_i=x_j$ if $i \sim_\rho j$} \}. \end{equation*}
	By abuse of notation, if $S \subseteq [n]$, we have a \emph{diagonal}
	\begin{equation*} \Delta_S = \{(x_1,\ldots,x_n) \in X^n: \text{ $x_i=x_j$ if $i, j \in S$} \}. \end{equation*}
	As usual, $\mathrm{F}(X,n) = X^n \backslash \bigcup_{|S| \geq 2} \Delta_S = X^n \backslash \bigcup_{|S| = 2} \Delta_S$ is the configuration space.
	
	In \S\ref{section: polydiagonal degenerations} and \S\ref{section: small resolutions}, we will rely heavily on Li's theory of wonderful compactifications \cite{[Li09]}. It is presumed that the reader is familiar with the results of \cite{[Li09]}, and with the way in which the results of \cite{[FM94]} can be phrased in the language of \cite{[Li09]}, which is sketched in \cite[\S4.2]{[Li09]}. If ${\mathcal G}$ is a building set in $Y$, we will sometimes use the notation $ X \times {\mathcal G} = \{ X \times S: S \in {\mathcal G} \}$, which is a building set in $X \times Y$. We will use similar notations ${\mathcal G} \times Z$, $X \times {\mathcal G} \times Z$, etc. for the building sets in $Y \times Z$, $X \times Y \times Z$, etc.

	\subsection{Construction of polydiagonal degenerations}\label{subsection: polydiagonal degenerations} In \S\ref{section: polydiagonal degenerations}, we construct the `polydiagonal degenerations' of $X^n$ and of the Fulton--MacPherson compactification $X[n]$ \cite{[FM94]}. The author owes the reader an explanation and apology regarding exposition:  the decision to carry out the two constructions simultaneously in the interest of keeping the size of the paper reasonable has the major expository downside that many of the objects considered are trivial or empty in the case of $X^n$, and are of relevance only in the case of $X[n]$. It might be possible to find a natural common generalization by considering Kuperberg--Thurston compactifications \cite{[KT99]}, but this seems to be a further complication. Nevertheless, the exposition is quite similar to what it would have been if we only considered $X[n]$. To help the reader navigate the technicalities, we note that the entire point of the construction essentially amounts to items \ref{item: combinatorics item in polydiagonal theorem} and \ref{item: scaling in polydiagonal theorem} in Theorem \ref{theorem: polydiagonal degeneration}.
	
	Let $\graph$ be a graph with $n$ vertices labeled $1,\ldots,n$, and assume that either:
	\begin{enumerate}
		\item $\graph$ has no edges, that is, $\graph = K_n^c$; or
		\item $\graph$ is complete, that is, $\graph = K_n$.
	\end{enumerate}
	Let $F(X,\graph) = X^n \backslash \bigcup_{e \in \mathrm{Edges}(\graph)} \Delta_e$, which is thus either $X^n$ or ${\mathrm F}(n,X)$,
	$$
	X[\graph] = 
	\begin{cases}
	X^n & \text{if $\graph$ has no edges} \\
	X[n] & \text{if $\graph$ is complete,}
	\end{cases}
	$$
	and
	$$
	\dg_\graph = 
	\begin{cases}
	\emptyset & \text{if $\graph$ has no edges} \\
	\{\Delta_S: \mathop{} S \subseteq [n], |S| \geq 2\} & \text{if $\graph$ is complete}.
	\end{cases}
	$$
	Then, $X[\graph]$ is the wonderful compactification \cite[Theorem 1.2]{[Li09]} of the building set \cite[Definition 2.2]{[Li09]} $\dg_\graph$ in $X^n$ (allowing empty building sets), by e.g. \cite[\S4.2]{[Li09]}. 
	
	We will construct a degeneration $X[[\graph]]$ of $X[\graph]$, that is, a variety $X[[\graph]]$ and a (flat, projective) morphism $X[[\graph]] \to \aline$
	whose restriction to the preimage of $\cs = \aline \backslash \{0\}$ coincides with the projection $\cs \times X[\graph] \to \cs$. Note that
	\begin{equation} {\mathcal P}_\graph = (\aline \times  \dg_\graph) \cup \{ \{0\} \times \Delta_\rho: \rho \in L_n \} \end{equation}
	is a building set in $\aline \times X^n$. Indeed, if $\graph = K_n^c$, then ${\mathcal P}_\graph$ coincides with its induced arrangement, and if $\graph = K_n$ (and $n \geq 2$), then the arrangement induced by ${\mathcal P}_\graph$ is $ \{ \aline \times \Delta_\rho: \rho \in L_n \setminus \{ \top \} \}  \cup \{ \{0\} \times \Delta_\rho: \rho \in L_n \}$, and the intersection $\aline \times \Delta_\rho = \bigcap_{B \in \blocks(\rho)} (\aline \times \Delta_B)$ in $\aline \times X^n$ is transverse.
	
	\begin{definition}\label{definition: polydiagonal degeneration}
	The \emph{polydiagonal degeneration} $X[[\graph]]$ of $X[\graph]$ is the wonderful compactification of ${\mathcal P}_\graph$.
	\end{definition}
	
	If $\graph$ is complete, there are two types of `boundary' divisors on $X[[\graph]]$, but only one if $\graph$ has no edges. Namely,
	\begin{enumerate}
	\item let $D^v_\rho \subset X[[\graph]]$ be the divisor corresponding to $\{0\} \times \Delta_\rho$; and
	\item let $D^h_S \subset X[[\graph]]$ be the divisor corresponding to $\aline \times \Delta_S$ if $\graph$ is complete.
	\end{enumerate}
	Recall that the $\dg_{K_n}$-nests correspond to nested subfamilies of $Z_n$ (\cite{[FM94]} and \cite{[Li09]}). In our setup, the situation is similar. 
	
	\begin{lemma}\label{lemma: nests in polydiagonal degeneration}
	Any ${\mathcal P}_\graph$-nest \cite[Definition 2.3]{[Li09]} is a union of two sets as follows:
	\begin{enumerate}
		\item the polydiagonals in $\{0\} \times X^n$ corresponding to a chain ${\mathcal H} \subseteq L_n$, and
		\item the products of $\aline$ with diagonals in $X^n$ which form a $\dg_\graph$-nest, such that any subset of $[n]$ corresponding to one of these diagonals of $X^n$ is contained in some block of $\max {\mathcal H}$. (Thus, this is empty if $\graph$ has no edges.)
	\end{enumerate}
	Conversely, any subset of ${\mathcal P}_\graph$ of this form is a ${\mathcal P}_\graph$-nest. 
	\end{lemma}
	
	\begin{proof}
		If $\graph =K_n^c$, the claim is trivial, since ${\mathcal P}_\graph$ coincides with the arrangement it induces. When $\graph = K_n$, any flag in the induced arrangement is of the form
		$ \{0\} \times \Delta_{\rho_1} \subset \cdots \subset \{0\} \times \Delta_{\rho_k} \subset \aline \times \Delta_{\rho_{k+1}} \subset \cdots \subset \aline \times \Delta_{\rho_\ell}$, 
		with $\rho_1<\cdots<\rho_k \leq \rho_{k+1}<\cdots < \rho_\ell$ in $L_n$, $k \leq \ell$, and $\rho_\ell \neq \top$ if $k \neq \ell$, and the argument is similar to that in the usual Fulton--MacPherson setup.
	\end{proof}
	
	Thus, the ${\mathcal P}_\graph$-nests are in bijective correspondence with the chains in $L_n$ if $\graph=K_n^c$, respectively with the set of comparable (\S\ref{subsection: Fulton--MacPherson nests}) pairs $({\mathcal H},{\mathcal F})$ if $\graph = K_n$. 
	
	\subsection{Universal families}\label{subsection: universal families} Most of \S\ref{section: polydiagonal degenerations} is devoted to constructing a family of degenerations of $X$ over $X[[\graph]]$. In fact, we will construct two such families, $X[[\graph]]^\fuf$ and $X[[\graph]]^\suf$. The `correct' one is $X[[\graph]]^\suf$, but it is convenient for many technical reasons to have both when $\graph$ is complete; when $\graph$ has no edges, they coincide. 
	
	First, let us review the universal family over $X[\graph]$. Let $\Delta'_T \subseteq X^{n+1}$ be the diagonal corresponding to $T \subseteq [n+1]$, and $\Delta'_\varpi \subseteq X^{n+1}$ the polydiagonal corresponding to $\varpi \in L_{n+1}$. Note that
	\begin{equation} 
	\dg_\graph^+ = (\dg_\graph \times X) \cup \{ \Delta'_{S \cup \{n+1\}}: \mathop{} \Delta_S \in \dg_\graph \} 
	\end{equation}
	is a building set in $X^{n+1}$. Then, the wonderful compactification $X[\graph]^+$ of $\dg_\graph^+$ is the universal family over $X[\graph]$, either $X^{n+1}$ or $X[n]^+$, cf. \cite{[FM94]} for the latter.
	
	To construct the universal families over the polydiagonal degenerations, let
		\begin{equation*}
			\begin{aligned}
				{\mathcal U}^\fuf_{\graph} &= (\aline \times \dg_\graph \times X) \cup \{\{0\} \times \Delta'_\varpi : \mathop{} \varpi \in L_{n+1} \}, \\
				{\mathcal U}^\suf_{\graph} &= (\aline \times \dg_\graph^+) \cup \{\{0\} \times \Delta'_\varpi : \mathop{} \varpi \in L_{n+1} \} . \\
			\end{aligned}
		\end{equation*}
	Note that ${\mathcal U}^\fuf_\graph$ and ${\mathcal U}^\suf_\graph$ are building sets in $\aline \times X^{n+1}$ (this is similar to ${\mathcal P}_\graph$ being a building set), and that ${\mathcal U}^\fuf_\graph = {\mathcal U}^\suf_\graph$ if $\graph$ has no edges. 
	
	\begin{definition}
		Let $X[[\graph]]^\fuf$ and $X[[\graph]]^\suf$ be the wonderful compactifications of ${\mathcal U}^\fuf_\graph$ and ${\mathcal U}^\suf_\graph$ respectively. 
	\end{definition}
	
	\begin{proposition}\label{proposition: commutative diagram}
		There exists a unique commutative diagram
		\begin{center}
			\begin{tikzpicture}[scale = 1.4]
				
				\node (a) at (4,0) {$\aline \times X [\graph] \times X$};
				\node (b) at (0.5,0) {$X [[\graph]] \times X$};
				
				\node (d) at (4,1) {$\aline \times X [\graph]^+$};
				\node (e) at (1.5,1) {$X [[\graph]]^\suf$};
				\node (f) at (-0.5,1) {$X [[\graph]]^\fuf$};
				
				\draw [->] (b) --node[above] {$\upsilon_\graph \times \mathrm{id}_X$} (a); 
				\draw [->] (e) -- (d); 
				\draw [->] (e) --node[above] {$\xi_\graph$}  (f); 
				\draw [->] (d) -- (a); \draw [->] (e) -- (b); \draw [->] (f) -- (b); 
			\end{tikzpicture}
		\end{center}
		over $\aline \times X^{n+1}$. (All varieties in the diagram are $\aline \times X^{n+1}$-schemes naturally.)
	\end{proposition} 
	
	When $\graph$ has no edges, $\xi_\graph$ and the downward arrow are isomorphisms. 
	
	\begin{proof}
		The diagram is necessarily commutative and unique since all varieties are canonically birational to $\aline \times X^{n+1}$. Hence, it suffices to construct any morphisms over $\aline \times X^{n+1}$.
		
		\begin{lemma}\label{lemma: enumerations satisfy Li's condition}
			Any of the following enumerations of a building set:
			\begin{enumerate}
			\item the enumeration of ${\mathcal P}_\graph$ corresponding to $\mathfrak{el}$ (see \S\ref{subsection: enumerations});
			\item the enumeration of ${\mathcal U}_\graph^\suf$ corresponding to $\mathfrak{lel^\star e^\star}$; or
			\item the enumeration of ${\mathcal U}_\graph^\suf$ corresponding to $\mathfrak{ee^\star L}$
			\end{enumerate}
			satisfies condition $(*)$ in \cite[Theorem 1.3]{[Li09]}.
		\end{lemma}
		
		\begin{proof}
		It is well-known that the FM enumerations ${\mathfrak e}$ and ${\mathfrak e}{\mathfrak e}^\star$ (in the context of $X^n$, though taking a direct product with $\aline$ doesn't change anything) satisfy condition $(*)$ in \cite[Theorem 1.3]{[Li09]}. We will focus on the two enumerations $\mathfrak{lel^\star e^\star}$ and $\mathfrak{ee^\star L}$ and the case $\graph = K_n$, since the other cases are similar but quite trivial. We will check that, given any `initial segment' $\mathfrak{s}$ in either enumeration ($\mathfrak{lel^\star e^\star}$ or $\mathfrak{ee^\star L}$), and any subset of indices $J$ in ${\mathfrak s}$, whose corresponding subvarieties intersect at  $F_J$, the minimal subvarieties in ${\mathfrak s}$ which contain $F_J$ intersect transversally. If $J$ contains only sets, no subvariety indexed by a partition can contain $F_J$, so we can invoke $(*)$ for the FM enumeration. If $J$ contains at least one partition, $F_J$ must have already appeared on the list of subvarieties indexed by ${\mathfrak s}$, so the requirement is satisfied in this case as well.
		\end{proof}
		
		In light of Lemma \ref{lemma: enumerations satisfy Li's condition}, \cite[Theorem 1.3]{[Li09]} gives sequences of morphisms:
		\begin{enumerate}
			\item $X[[\graph]] \to \aline \times X[\graph] \to \aline \times X^n$;
			\item $X[[\graph]]^\suf \to X[[\graph]]^\fuf \to X[[\graph]] \times X \to \aline \times X^{n+1}$;
			\item $X[[\graph]]^\suf \to \aline \times X[\graph]^\suf \to \aline \times X[\graph] \times X \to \aline \times X^{n+1}$
		\end{enumerate}
		corresponding respectively to the three enumerations in Lemma \ref{lemma: enumerations satisfy Li's condition}. For instance, for the first one, we start with $\aline \times X^n$ and blow up in order until we get $X[[\graph]]$ according to \cite[Theorem 1.3]{[Li09]}, but at some point in the sequence of blowups we will see $\aline \times X[\graph]$. The other two sequences of morphisms are analogous. 
	\end{proof}
	
	\begin{theorem}\label{theorem: polydiagonal degeneration}
		The following hold.
		\begin{enumerate}
			\item\label{item: item 1 in theorem: polydiagonal degeneration}
			Both $X[[\graph]]^\fuf$ and $X[[\graph]]^\suf$ (with the maps to $X[[\graph]]$ and $X$ obtained by composing the `diagonal' maps in Proposition \ref{proposition: commutative diagram} with the projections to the factors) are naive FM spaces of $X$ over $X[[\graph]]$.
			\item There exist unique smooth sections $ x^\suf_1,\ldots,x^\suf_n:X[[\graph]] \to X[[\graph]]^\suf$ of the map $X[[\graph]]^\suf \to X[[\graph]]$, such that the diagram
			\begin{center}
			\begin{tikzpicture}
				\node (a) at (0.5,0) {$\cs \times X^n$};
				\node (b) at (4,0) {$\cs \times \mathrm{F}(X,\graph)$};
				\node (c) at (7,0) {$X[\graph]$};
				\node (au) at (0.5,1.2) {$\cs \times X^n \times X$};
				\node (bu) at (4,1.2) {$\cs \times \mathrm{F}(X,\graph) \times X$};
				\node (cu) at (7,1.2) {$X[\graph]^\suf$};
				
				\draw[left hook->] (bu) -- (au); \draw[right hook->] (bu) -- (cu);
				\draw[left hook->] (b) -- (a); \draw[right hook->] (b) -- (c);
				\draw[->] (a) -- (au); \draw[->] (b) -- (bu); \draw[->] (c) -- node[right] {$x_i$} (cu);
			\end{tikzpicture}
			\end{center}
			where the left vertical arrow is the graph of the projection to the $i$-th $X$ factor, is commutative. Moreover,
			\begin{enumerate}
				\item\label{item: disjoint sections} $x_i^\suf(z) \neq x_j^\suf(z)$, for all $z \in X[[\graph]]$ and $i \neq j$ adjacent in $\graph$; and
				\item $x^\fuf_i = \xi_\graph \circ x^\suf_i$ are smooth sections of $X[[\graph]]^\fuf \to X[[\graph]]$.
			\end{enumerate}
			\item\label{item: combinatorics item in polydiagonal theorem} Let $z \in X[[\graph]](\cc)$ in the fiber of $0 \in \aline$. By Lemma \ref{lemma: nests in polydiagonal degeneration}, ${\mathcal H} = \{\rho \in L_n: z \in D_\rho^v \}$ is a chain in $L_n$ (we are implicitly relying on \cite[Theorem 1.2]{[Li09]}).
			\begin{enumerate}
				\item\label{item: combinatorics empty case item in polydiagonal theorem} Assume that $\graph$ has no edges. Then, the dual tree of $X [[\graph]]^\fuf_z$ is $T({\mathcal H})$, and so is the dual tree of $X [[\graph]]^\suf_z$, since the two coincide. 
				\item\label{item: combinatorics complete case item in polydiagonal theorem} Assume that $\graph$ is complete, and let ${\mathcal F} = \{S \subseteq [n]: |S| \geq 2, z \in D_S^h \}$. By Lemma \ref{lemma: nests in polydiagonal degeneration}, ${\mathcal H}$ and ${\mathcal F}$ are comparable (\S\ref{subsection: Fulton--MacPherson nests}). Then,
				\begin{enumerate}
					\item the dual tree of $X [[\graph]]^\fuf_z$ is $T({\mathcal H})$; and
					\item the dual tree of $X [[\graph]]^\suf_z$ is $T({\mathcal H},{\mathcal F})$.
				\end{enumerate}
				Moreover, if $\Lambda$ is an irreducible component of $X [[\graph]]^\suf_z$, then $\xi_\graph|_\Lambda$ is generically finite onto its image if and only if the corresponding vertex of $T({\mathcal H},{\mathcal F})$ has infinite or finite (non-zero) scaling (\S\ref{subsection: Fulton--MacPherson nests}), and, in that case, $\xi_\graph$ maps $\Lambda$ birationally to an irreducible component of $X [[\graph]]^\fuf_z$.
			\end{enumerate}
			\item\label{item: scaling in polydiagonal theorem} By item \ref{item: item 1 in theorem: polydiagonal degeneration} and Lemma \ref{lemma: defining property of root}, we may speak of the root relative dualizing sheaf on $X[[\graph]]^\fuf$. Then, there exists a (unique) homomorphism
			$$ \psi: \sqrt[d]{\omega} = \sqrt[d]{\omega}_{X[[\graph]]^\fuf,X/X[[\graph]]} \to {\sh O}_{ X[[\graph]]^\fuf } $$
			which maps to $t \in \cc[t,t^{-1}]$ under the restriction map
			$$ H^0( X[[\graph]]^\fuf ,  \sqrt[d]{\omega}^\vee ) \to H^0(\cs \times_{\aline} X[[\graph]]^\fuf,  \sqrt[d]{\omega}^\vee ) = H^0 ({\sh O}_{\cs \times X[\graph] \times X}) \cong { \mathbb C} [t,t^{-1}]. $$
			\item The homomorphisms $(x_1^\fuf)^*\psi,\ldots,(x^\fuf_n)^*\psi$ are isomorphisms.
		\end{enumerate}
	\end{theorem}
	
	(A parenthetical remark on $\psi$ for context: if $\epsilon$ is the distinguished section of $\sqrt[d]{\omega}$ alluded to in \S\ref{section: naive FM spaces}, and $\varpi:X[[\graph]]^\fuf \to X[[\graph]]$ is the projection, then it can be shown that $\psi \otimes \epsilon = \varpi^*t$.)
	
	\subsection{Proof of Theorem \ref{theorem: polydiagonal degeneration}}\label{subsection: proof of theorem}
	The proof is a long induction, which follows the sequence of blowups indexed by $\mathfrak{lel^\star e^\star}$ (Lemma \ref{lemma: enumerations satisfy Li's condition}) from $\aline \times X^{n+1}$ to $X[[\graph]]^\suf$. We split it into $3$ parts as follows:
	\begin{itemize}
		\item [I] the blowups indexed by $L_n$, and, if $\graph = K_n$, those indexed by $Z_n$; 
		\item [II] the blowups indexed by $L_{n+1} \backslash L_n \simeq U_n$ (Remark \ref{remark: bijection of complement}); 
		\item [III] only in the case $\graph = K_n$, the blowups indexed by $Z_{n+1}^*\backslash Z_n$. 
	\end{itemize}
	The three parts correspond respectively to the segments $\mathfrak{le}$, ${\mathfrak l}^\star$, and ${\mathfrak e}^\star$. 
	
	\subsubsection*{Proof of Theorem \ref{theorem: polydiagonal degeneration} -- part I} After performing the blowups of $\aline \times X^{n+1}$ indexed by $\mathfrak{le}$, we obtain $X[[\graph]] \times X$.
	
	Let $x_i^0: X[[\graph]] \to X[[\graph]] \times X$ be the graph of the composition of $X[[\graph]] \to X^n$ with the projection to the $i$-th factor $X^n \to X$, and $\Sigma_{i}^0 \subset X[[\graph]] \times X$ its image. For each $\varpi \in L_{n+1}$, let $G_\varpi^0 \subset X[[\graph]] \times X$ be the dominant transform of $\{0\} \times \Delta'_\varpi$ in the sequence of blowups from $\aline \times X^{n+1}$ to $X[[\graph]] \times X$. If $\varpi \in L_{n+1} \backslash L_n$ and $\varpi$ corresponds to $(\rho,B) \in U_n$ by Remark \ref{remark: bijection of complement}, we will also write $G^0_{\rho,B}$ for $G^0_\varpi$. If $\graph$ is complete, for each $T \in Z_{n+1}^*$, let $G^0_T \subset X[[\graph]] \times X$ be the dominant transform of $\aline \times \Delta'_T$ in the sequence of blowups from $\aline \times X^{n+1}$ to $X[[\graph]] \times X$.
	
	\begin{lemma}\label{lemma: more about after round 1 of blowups}
		\begin{enumerate}
			\item\label{item: item 1 in lemma: more about after round 1 of blowups} If $\varpi = \iota(\rho)$ with $\rho \in L_n$, then $G_\varpi^0 = D^v_\rho \times X$.
			\item\label{item: item 2 in lemma: more about after round 1 of blowups} If $\graph = K_n$ and $T \in Z_{n+1}^*$ does not contain $n+1$, then $G_T^0 = D^h_T \times X$.
			\item\label{item: item 3 in lemma: more about after round 1 of blowups} If  $(\rho,B) \in U_n$, then $ G_{\rho,B}^0 = x_i^0(D^v_\rho)$ for all $i \in B$.
			\item\label{item: item 4 in lemma: more about after round 1 of blowups} If $\graph = K_n$ and $T \in Z_{n+1}^*$ contains $n+1$, then $G^0_T = x_i^0(D^h_{T \backslash \{n+1\}})$ for all $i \in T$.
		\end{enumerate}
	\end{lemma}
	
	\begin{proof}
		Claims \ref{item: item 1 in lemma: more about after round 1 of blowups} and \ref{item: item 2 in lemma: more about after round 1 of blowups} are trivial. Claims \ref{item: item 3 in lemma: more about after round 1 of blowups} and \ref{item: item 4 in lemma: more about after round 1 of blowups} follow inductively using Lemma \ref{lemma: technical fact about dominant transforms}.
		\end{proof} 
	
	\subsubsection*{Proof of Theorem \ref{theorem: polydiagonal degeneration} -- part II} The morphism $X[[\graph]]^\fuf \to X[[\graph]] \times X$ was obtained explicitly as a sequence of blowups with smooth centers indexed by $L_{n+1} \backslash L_n$, by Proposition \ref{proposition: commutative diagram} and Lemma \ref{lemma: enumerations satisfy Li's condition}.
	
	\begin{definition}\label{definition: trimming}
	Given a rooted tree $T$ with legs attached only to leaves, a set of vertices $S$ is a \emph{star subset} if the shortest path from the root to $v$ is contained in $S$, for all $v \in S$. The \emph{trimming} of $T$ relative to $S$ is the subgraph whose set of vertices is $S$, whose edges are the edges of $T$ with both endpoints in $S$, and whose legs (half-edges) are obtained from the legs of $T$ with endpoints in $S$, and from the edges of $T$ with only one endpoint in $S$. 
	\end{definition}
	
	In particular, if ${\mathcal H} \subseteq L_n$ is a chain, and $T({\mathcal H})$ is its representation described in Definition \ref{definition: leveled tree representations}, then any initial segment in the enumeration ${\mathfrak l}^\star$ specifies a trimming of $T({\mathcal H})$, since the subset of vertices of $T({\mathcal H})$ whose label is in this initial segment is a star subset. 
	
	\begin{claim}\label{claim: big induction in round 2 of blowing up}
		Let $W$ be a blowup of $X[[\graph]] \times X$ in the sequence of blowups from $X[[\graph]] \times X$ to $X[[\graph]]^\fuf$. Let $\Sigma_i \subset W$ be the dominant transform of $\Sigma_i^0$.
		\begin{enumerate}
			\item\label{item: naive FM space in lemma: big induction in round 2 of blowing up} $(W \to X[[\graph]], W \to X)$ is a naive FM space of $X$ over $X[[\graph]]$.
			\item\label{item: sections remain sections in lemma: big induction in round 2 of blowing up} The composition $\Sigma_i \hookrightarrow W \to X[[\graph]]$ is an isomorphism.
			\item\label{item: sections are smooth in lemma: big induction in round 2 of blowing up} The morphism $W \to X[[\graph]]$ is smooth at all points of $\Sigma_i$.
		\end{enumerate}
		According to item \ref{item: sections remain sections in lemma: big induction in round 2 of blowing up}, $\Sigma_i$ is the image of a section $x_i:X[[\graph]] \to W$. Let $G_{\rho,B} \subset W$ be the dominant transform of $G_{\rho,B}^0$, for $(\rho,B) \in U_n$.
		\begin{enumerate}\setcounter{enumi}{3}
			\item\label{item: loci to be blown up are sections in lemma: big induction in round 2 of blowing up} $G_{\rho,B} = x_i(D^v_\rho)$ if $i \in B$ and the dominant transform of the polydiagonal corresponding to $(\rho,B)$ hasn't been blown up yet.
			\item\label{item: combinatorics in lemma: big induction in round 2 of blowing up} Let $z \in X[[\graph]]_0(\cc)$ and ${\mathcal H} = \{\rho \in L_n: z \in D^v_\rho\}$, which is a chain by Lemma \ref{lemma: nests in polydiagonal degeneration}. Then the dual tree of $W_z$ is the trimming of $T({\mathcal H})$ induced by the initial segment of ${\mathfrak l}^\star$ corresponding to $W$ (Definitions \ref{definition: leveled tree representations} and \ref{definition: trimming}). In particular, if a vertex of the trimmed tree has a leg towards a block which contains the number $i \in [n]$, then $x_i(z)$ lives on the component of $W_z$ which corresponds to the vertex. 
		\end{enumerate}
	\end{claim}
	
	The maps $W \to X[[\graph]]$ and $W \to X$ are obtained as the compositions of the blowdown $W \to X[[\graph]] \times X$ with the projections to the factors.
	
	\begin{proof}
		We proceed inductively. The base case $W = X[[\graph]] \times X$ is trivial, with the exception of item \ref{item: loci to be blown up are sections in lemma: big induction in round 2 of blowing up}, which follows from Lemma \ref{lemma: more about after round 1 of blowups}. We make the following assumption: the initial segment of ${\mathfrak l}^\star$ corresponding to $W$ is a union of fibers of $L_{n+1} \backslash L_n \to L_n$. Thus, at the next steps we have to blow up all loci $G_{\rho_W,B}$ with $B \in \blocks(\rho_W)$ for fixed $\rho_W \in L_n$. Lemma \ref{lemma: extreme shuffling} below guarantees among other things that there is no loss of generality in doing so.
		
		\begin{lemma}\label{lemma: extreme shuffling}
			If $B_1,B_2 \in \blocks(\rho_W)$ and $B_1 \neq B_2$, then $G_{\rho_W,B_1} \cap G_{\rho_W,B_2} = \emptyset$.
		\end{lemma}
	
		\begin{proof}
			Let $U_n \ni \varpi_i \leftrightarrow (\rho_W,B_i)$ by Remark \ref{remark: bijection of complement}. The enumeration $\mathfrak{lel^\star e^\star}$ is
			\begin{equation*}
				\underbrace{\ldots}_{{\mathfrak l}}\underbrace{\ldots}_{{\mathfrak e}}\underbrace{\ldots\varpi_1 \wedge \varpi_2 \ldots|\ldots\varpi_1\ldots\varpi_2\ldots}_{{\mathfrak l}^\star}\underbrace{\ldots}_{{\mathfrak e}^\star},
			\end{equation*}
			where the bar `$|$' marks the current state at $W$. Let us shuffle as follows:
			\begin{equation}\label{equation: extreme shuffling}
				\varpi_1\wedge\varpi_2\underbrace{\ldots \overbrace{ \textcolor{gray}{\varpi_1\wedge\varpi_2} }^{\text{missing}}\ldots}_{\text{${\mathfrak l}^\star$ before $|$}}\underbrace{\ldots}_{{\mathfrak l}}\underbrace{\ldots}_{{\mathfrak e}}| \underbrace{\ldots\varpi_1\ldots\varpi_2\ldots}_{ \text{ ${\mathfrak l}^\star$ after $|$} }  \underbrace{\ldots}_{{\mathfrak e}^\star}.
			\end{equation}
			Then, nothing has changed after the current state, and, more importantly, the (enumeration of ${\mathcal U}_\graph^+$ corresponding to) enumeration \eqref{equation: extreme shuffling} satisfies condition $(*)$ in \cite[Theorem 1.3]{[Li09]}. This can be checked by an argument similar to Lemma \ref{lemma: enumerations satisfy Li's condition}. Indeed, if the chosen intersection is not contained in the $0$-fiber, then the property to check is clear from the fact that the FM enumeration satisfies Li's condition; if the intersection is contained in the $0$-fiber, then it must have already appeared on our list, which can be checked by a straightforward case analysis. Then the lemma follows from \cite[Theorem 1.3]{[Li09]} and the fact that, since $\varpi_1,\varpi_2$ are not comparable in $L_{n+1}$, blowing up $\{0\} \times \Delta'_{\varpi_1 \wedge \varpi_2}$ in $\aline \times X^{n+1}$ separates (the proper transforms of) $\{0\} \times \Delta'_{\varpi_1}$ and $\{0\} \times \Delta'_{\varpi_2}$. 
		\end{proof}
		
		Assume that $G_{\rho_W,B_W}$ is to be blown up next. By items \ref{item: sections are smooth in lemma: big induction in round 2 of blowing up} and \ref{item: loci to be blown up are sections in lemma: big induction in round 2 of blowing up}, $G_{\rho_W,B_W}$ is a smooth section over a smooth divisor in $X[[\graph]]$. Moreover, we claim that
		\begin{equation}\label{equation: intersection cases}
			x_i(X[[\graph]]) \cap G_{\rho_W,B_W} = \begin{cases}
				x_i(D^v_{\rho_W}), & \text{if $i \in B_W$} \\
				\emptyset, & \text{if $i \notin B_W$}. \\
			\end{cases}
		\end{equation}
		Indeed, the case $i \in B_W$ follows from item \ref{item: loci to be blown up are sections in lemma: big induction in round 2 of blowing up}, whereas the case $i \notin B_W$ follows from item \ref{item: loci to be blown up are sections in lemma: big induction in round 2 of blowing up} and Lemma \ref{lemma: extreme shuffling}: if $B'$ is the block of $\rho_W$ which contains $i$, then
		\begin{equation*}
			\begin{aligned}
				x_i(X[[\graph]]) \cap G_{\rho_W,B_W} &= x_i(X[[\graph]]) \cap (D^v_{\rho_W} \times_{X[[\graph]]} W) \cap G_{\rho_W,B_W} \\
				&= x_i(D^v_{\rho_W}) \cap G_{\rho_W,B_W} = G_{\rho_W,B'} \cap G_{\rho_W,B_W} = \emptyset.
			\end{aligned}
		\end{equation*}
		In either case, $x_i(X[[\graph]]) \cap G_{\rho_W,B_W}$ is a possibly empty smooth effective Cartier divisor on $x_i(X[[\graph]])$, so $x_i(X[[\graph]])$ remains the image of a section after blowing up, call it $x_i':X[[\graph]] \to W'$, where $W' \to W$ is the blowup. Note also that $W' \to X [[\graph]]$ is smooth at these sections. Thus, we've already checked items \ref{item: naive FM space in lemma: big induction in round 2 of blowing up}, \ref{item: sections remain sections in lemma: big induction in round 2 of blowing up}, and \ref{item: sections are smooth in lemma: big induction in round 2 of blowing up}. Item \ref{item: loci to be blown up are sections in lemma: big induction in round 2 of blowing up} is clear since the equality of the generic points remains true if $(\rho,B) \neq (\rho_W,B_W)$. Item \ref{item: combinatorics in lemma: big induction in round 2 of blowing up} follows by construction and \eqref{equation: intersection cases}.
	\end{proof}
	
	Let $x_i^\tfuf:X [[\graph]] \to X [[\graph]]^\fuf$ and $\Sigma^\fuf_i \subset X[[\graph]]^\fuf$ denote the section $x_i$ and the subvariety $\Sigma_i$ in Claim \ref{claim: big induction in round 2 of blowing up} at the final step $W = X [[\graph]]^\fuf$. It will turn out that $x_i^\tfuf = x_i^\fuf$, but we don't know this yet, so we use the notation $x_i^\tfuf$ for now.
	
	The same induction as in Claim \ref{claim: big induction in round 2 of blowing up} allows us to define the homomorphism
	\begin{equation}\label{equation: final scaling}
		\psi:\sqrt[d]{\omega}_{X[[\graph]]^\fuf,X/X[[\graph]]} \to {\sh O}_{X[[\graph]]^\fuf}.
	\end{equation}
	Specifically, we start with $\psi_{X[[\graph]] \times X}$ on $X[[\graph]] \times X$ corresponding to the pullback of the regular function $t$ on $\aline$ along $X[[\graph]] \times X \to \aline$. The property we need to check inductively simultaneously with constructing the homomorphism is the following.
	
	\begin{claim}
		If $\psi_W:\sqrt[d]{\omega}_{W,X/X[[\graph]]} \to {\sh O}_{W}$ is the map on $W$, and the blowup indexed by $(\rho,B)$ hasn't been performed yet, then 
		\begin{equation}\label{equation: inclusion of image in lone item} \mathrm{Im}(\psi_W) \subseteq {\sh I}_{G_{\rho,B},W}, \end{equation} 
		where $G_{\rho,B}$ is the dominant transform of $G^0_{\rho,B}$ on $W$.
	\end{claim}
	
	\begin{proof}
		Equation \eqref{equation: inclusion of image in lone item} ensures that Lemma \ref{lemma: lifting vector fields naively} can be applied to produce $\psi_{W'}$, where $W'$ is the next blowup in the sequence. The base case of \eqref{equation: inclusion of image in lone item} is trivial by construction. It is clear that \eqref{equation: inclusion of image in lone item} continues to hold on $W'$ because the desired vanishing obviously continues to hold at least at the generic point of $G_{\rho,B}$.
	\end{proof}
	
	This completes the construction of \eqref{equation: final scaling}. For simplicity, we write $\sqrt[d]{\omega}$ instead of $\sqrt[d]{\omega}_{X[[\graph]]^\fuf,X/X[[\graph]]}$ for the rest of this section.
	
	\begin{lemma}\label{lemma: scaling is nonvanishing at markings}
		The map $(x_i^\tfuf)^*\psi:(x_i^\tfuf)^*\sqrt[d]{\omega} \to {\sh O}_{X[[\graph]]}$ is an isomorphism.
	\end{lemma}
	
	\begin{proof}
		We proceed inductively again.
		
		\begin{claim}\label{claim: lone item 2 about scaling following lemma} Let $W$ be a blowup of $X[[\graph]] \times X$ in the sequence of blowups in Claim \ref{claim: big induction in round 2 of blowing up} and $x_{i,W}:X[[\graph]] \to W$ the sections on $W$. If the blowup indexed by $(\rho,B)$ has been performed already, then the restriction of $\psi_W$ to $x_{i,W}(D^v_\rho)$ is not identically $0$ for any $i \in B$.
		\end{claim}
		
		\begin{proof}
		The base case is vacuous. Let $W'$ be next variety in the sequence, obtained by performing the blowup $\beta:W' \to W$ corresponding to $(B,\rho)$, with exceptional divisor $E \subset W'$. It suffices to check Claim \ref{claim: lone item 2 about scaling following lemma} for the pair $(B,\rho)$ only. It is clear that $\mathrm{Im}(\psi_W)$ is not contained in ${\sh I}^2_{x_{i,W}(D^v_\rho),W}$ -- if it was the case, we could remove the exceptional loci of all blowups so far and find that the same holds at least in a suitable open in $X[[\graph]] \times X$ (note that $x_{i,W}(D^v_\rho)$ is not contained in the union of the removed exceptional divisors), but this is clearly not true given the definition of $\psi_0$. It follows from Lemma \ref{lemma: lifting vector fields naively} that
		\begin{equation}\label{equation: not 0 on exceptional at next step}
			\psi_{W'}|_E \neq 0.
		\end{equation}
		Clearly, $E \to D^v_{\rho}$ is a projective bundle, and, for any $z \in D^v_{\rho}(\cc)$,
		\begin{equation}\label{equation: restriction of dual root to fiber of exceptional}
			\sqrt[d]{\omega}_{W',X/X[[\graph]]}^\vee |_{E_z} \simeq {\sh O}_{E_z}(1),
		\end{equation}
		by Lemma \ref{lemma: defining property of root}. Moreover, $\psi_{W'}$ vanishes along the hyperplane $H \subset E_z$ where $E_z$ intersects the proper transform of $W_0$ because $\psi_W$ vanishes in a neighbourhood in $W_0$ of $x_{i,W}(D^v_\rho)$, which in turn follows from the definition of $\psi_0$ and the fact that $W_z \simeq X$ for general $z \in D^v_{\rho}(\cc)$. Also, note that $H$ is precisely the subset of $E_z$ where $W'_0$ is singular, so in particular $x_{i,W'}(z) \notin H$ by item \ref{item: sections are smooth in lemma: big induction in round 2 of blowing up} in Claim \ref{claim: big induction in round 2 of blowing up}. Thus, if $\psi_{W'}$ also vanishes at $x_{i,W'}(z)$, then it must vanish everywhere on $E_z$ by \eqref{equation: restriction of dual root to fiber of exceptional}. In conclusion, if Claim \ref{claim: lone item 2 about scaling following lemma} is false, then \eqref{equation: not 0 on exceptional at next step} is false, contradiction.
		\end{proof}
		
		Now consider the final case $W = X [[\graph]]^\fuf$. By Claim \ref{claim: lone item 2 about scaling following lemma} above and the obvious fact that $(x_i^\tfuf)^*\psi$ is an isomorphism away from the central fiber, $(x_i^\tfuf)^*\psi$ is an isomorphism away from a codimension $2$ subvariety of $X [[\graph]]$, hence an isomorphism.
	\end{proof}
	
	Claim \ref{claim: big induction in round 2 of blowing up} and Lemma \ref{lemma: scaling is nonvanishing at markings} complete the proof of Theorem \ref{theorem: polydiagonal degeneration} in the case $\graph = K_n^c$, and much of it in the case $\graph = K_n$, specifically, everything about $X[[K_n]]^\fuf$ pending establishing $x_i^\fuf = x_i^\tfuf$.
		
	\subsubsection*{Proof of Theorem \ref{theorem: polydiagonal degeneration} -- part III} 	Only some claims in the case $\graph=K_n$ still remain to be checked. Assume that $\graph = K_n$, and consider the sequence of blowups from $X[[K_n]]^\fuf$ to $X[[K_n]]^\suf$. The fact that the total space remains a naive FM degeneration space over $X[[K_n]]$, the fact that the sections $x_i$ continue to lift all the way to $X[[K_n]]^\suf$, claim \ref{item: disjoint sections}, as well as the description of the dual tree of a fiber follow by an inductive argument similar to the well-known argument for $X[n]$, which is detailed below.
	
	\begin{claim}\label{claim: big induction in round 3 of blowing up}
		Let $W$ be a variety in the sequence of blowups from $X[[\graph]]^\fuf$ to $X[[\graph]]^\suf$. Let $\Sigma_{i,W} \subset W$ be the dominant transform of $\Sigma^\fuf_i \subset X[[\graph]]^\fuf$.
		\begin{enumerate}
			\item\label{item: naive FM space in lemma: big induction in round 3 of blowing up} $(W \to X[[\graph]], W \to X)$ is a naive FM space of $X$ over $X[[\graph]]$.
			\item\label{item: sections remain sections in lemma: big induction in round 3 of blowing up} The composition $\Sigma_{i,W} \hookrightarrow W \to X[[\graph]]$ is an isomorphism.
			\item\label{item: sections are smooth in lemma: big induction in round 3 of blowing up} The morphism $W \to X[[\graph]]$ is smooth at all points of $\Sigma_{i,W}$.
		\end{enumerate}
		By item \ref{item: sections remain sections in lemma: big induction in round 2 of blowing up}, $\Sigma_{i,W}$ is the image of a section $x_{i,W}:X[[\graph]] \to W$. Let $F_S \subset W$ be the dominant transform of $\aline \times \Delta'_{S \cup \{n+1\}} \subset \aline \times X^{n+1}$, for each $S \subseteq [n]$ nonempty.
		\begin{enumerate}\setcounter{enumi}{3}
			\item\label{item: loci to be blown up are sections in lemma: big induction in round 3 of blowing up} $F_S = x_{i,W}(D^h_S)$ if $|S| \geq 2$, $i \in S$, and the dominant transform of the variety $\aline \times \Delta'_{S \cup \{n+1\}} \in {\mathcal U}_\graph^\suf$  hasn't been blown up yet.
			\item\label{item: combinatorics in lemma: big induction in round 3 of blowing up} Let $z \in X[[\graph]]_0(\cc)$, $ {\mathcal H} = \{ \rho \in L_n: z \in D^v_\rho \}$, and ${\mathcal F} = \{ S \in Z_n : z \in D^h_S \}$. Let ${\mathcal F}_W \subseteq {\mathcal F}$ be the set of $S \in {\mathcal F}$ such that the dominant transform of the variety $\aline \times \Delta'_{S \cup \{n+1\}} \in {\mathcal U}_\graph^\suf$ has already been blown up. By Lemma \ref{lemma: nests in polydiagonal degeneration}, $({\mathcal H},{\mathcal F})$ is a comparable pair in the sense of \S\ref{subsection: Fulton--MacPherson nests}, and hence so is $({\mathcal H},{\mathcal F}_W)$. Then, the dual tree of $W_z$, the fiber of $W$ over $z$, is $T({\mathcal H},{\mathcal F}_W)$. 
		\end{enumerate}
	\end{claim}
	
	\begin{proof}
		All claims in the base case $W = X[[\graph]]^\fuf$ are either trivial, or follow from the results of part II of this proof. Assume that the claim holds for $W$, and let's prove it for the next variety $W'$ in the sequence of blowups. Assume that $S_W \subseteq [n]$ is such that $W'$ is obtained from $W$ by blowing up $F_{S_W}$.
		
		\begin{lemma}\label{lemma: remaining issue}
			If $i \in [n] \backslash S_W$, then the dominant transforms of $\aline \times \Delta'_{\{i,n+1\}}$ and $\aline \times \Delta'_{S_W \cup \{n+1\}}$ on $W$ are disjoint.
		\end{lemma}
		
		\begin{proof}
			We will give an argument by reducing to the analogous statement in the Fulton--MacPherson setup. Let $Y = W_1$, the fiber of $W \to \aline$ over $1$. Then $Y$ is the variety `analogous' to $W$ in the sequence of blowups from $X[n] \times X$ to $X[n]^+$.  
			
			For each nonempty subset $S \subseteq [n]$, we have the commutative diagram
		\begin{center}
			\begin{tikzpicture}[scale = 1.2]
				\node (a) at (3,0) {$\aline \times X [n] \times X$};
				\node (b) at (0.5,0) {$X [[K_n]]^\fuf$};
				\node (d) at (3,1) {$\aline \times Y$};
				\node (e) at (0.5,1) {$W$};
				\node (d') at (3,2) {$\aline \times X [n]^+$};
				\node (e') at (0.5,2) {$X [[K_n]]^\suf$};
				\node (f) at (-1.5,1) {\stackon[-8pt]{$\aline \times \Delta'_{S \cup \{n+1\}}$}{\vstretch{1.5}{\hstretch{3.4}{\widetilde{\phantom{\;\;\;\;\;\;\;\;}}}}}};
				\node (f') at (5.2,1) {$\aline \times $\stackon[-8pt]{$\Delta'_{S \cup \{n+1\}}$}{\vstretch{1.5}{\hstretch{2.2}{\widetilde{\phantom{\;\;\;\;\;\;\;\;}}}}}};
				
				\draw [->] (b) -- (a); 
				\draw [->] (e) -- (d);  \draw [->] (e') -- (d'); 
				\draw [->] (e') -- (e);  \draw [->] (d') -- (d);
				\draw [->] (d) -- (a); \draw [->] (e) -- (b);
				\draw[right hook->] (f) -- (e); \draw[left hook->] (f') -- (d);
				\draw[->,dotted] (f) to [out=15, in = 165] (f');
			\end{tikzpicture}
		\end{center}
	in which the tall commutative square comes from Proposition \ref{proposition: commutative diagram}, while the map $W \to \aline \times Y$ can be obtained by precisely the same argument used in the proof of Proposition \ref{proposition: commutative diagram} to construct the map $X[[K_n]]^\suf \to \aline \times X[n]^+$, by simply using the suitable shorter enumeration. Commutativity in the two central columns is once again clear, because all varieties are naturally schemes over $\aline \times X^{n+1}$ by birational morphisms. The image in $\aline \times Y$ of the dominant transform on $W$ of $\aline \times \Delta'_{S \cup \{n+1\}}$ is the direct product of $\aline$ with the dominant transform on $Y$ of $\Delta'_{S \cup \{n+1\}}$. Indeed, this is trivial in the complement of the central fiber, hence the generic point of the former maps to generic point of the latter (both are irreducible), so the statement must hold. 
	
	In particular, if the dominant transforms on $Y$ of $\Delta'_{S_1 \cup \{n+1\}}$ and $\Delta'_{S_2 \cup \{n+1\}}$ are disjoint, then the dominant transforms on $W$ of $\aline \times \Delta'_{S_1 \cup \{n+1\}}$ and $\aline \times \Delta'_{S_2 \cup \{n+1\}}$ are disjoint as well. Thus, to prove the lemma, it suffices to prove that the dominant transforms of $\Delta'_{\{i,n+1\}}$ and $\Delta'_{S_W \cup \{n+1\}}$ on $Y$ are disjoint, which follows from \cite[Proposition 3.3]{[FM94]}. Indeed, if we let $S_W \cup \{n+1\}$, $\{i,n+1\}$, $n-|S_W|$ play the roles of $S'$, $T'$, $k$ from loc. cit. respectively, then $Y_k$ of loc. cit. is either our $Y$ or (more often) a blowdown of our $Y$, and we are done. 
		\end{proof}
		
		For each $i \in [n]$, we claim that
		\begin{equation}\label{equation: possible intersections in round 3}
			x_{i,W}(X[[\graph]]) \cap F_{S_W} =
			\begin{cases}
			x_{i,W}(D^h_{S_W}) & \text{if $i \in S_W$} \\
			\emptyset & \text{if $i \notin S_W$}.
			\end{cases}
		\end{equation} 
		If $i \in S_W$, \eqref{equation: possible intersections in round 3} follows from \ref{item: loci to be blown up are sections in lemma: big induction in round 3 of blowing up}. Recall that $x_{i,W}(X[[\graph]]) = \Sigma_{i,W}$ is the dominant transform of $\aline \times \Delta'_{\{i,n+1\}}$ on $W$, so, if $i \notin S_W$, \eqref{equation: possible intersections in round 3} follows from Lemma \ref{lemma: remaining issue}. Thus, the center of the next blowup is either disjoint from $\Sigma_{i,W}$, or is a Cartier divisor on $\Sigma_{i,W}$. Now the inductive step is easy: item \ref{item: naive FM space in lemma: big induction in round 3 of blowing up} for $W'$ follows from items \ref{item: naive FM space in lemma: big induction in round 3 of blowing up}, \ref{item: sections are smooth in lemma: big induction in round 3 of blowing up} and \ref{item: loci to be blown up are sections in lemma: big induction in round 3 of blowing up} for $W$; items \ref{item: sections remain sections in lemma: big induction in round 3 of blowing up} and \ref{item: sections are smooth in lemma: big induction in round 3 of blowing up} for $W'$ follow from the corresponding items for $W$ and \eqref{equation: possible intersections in round 3}, \ref{item: loci to be blown up are sections in lemma: big induction in round 3 of blowing up} is clear, since the claimed equality clearly remains true on the level of generic points, so it must still hold, and \ref{item: combinatorics in lemma: big induction in round 3 of blowing up} is clear by construction.
	\end{proof}
	
	We define $x_i^\suf$ to be $x_{i,W}$ at the final step $W=X[[K_n]]^\suf$. By construction, we have $x^\tfuf_i = \xi_\graph \circ x^\suf_i = x^\fuf_i$. It only remains to check that the sections $x_i^\suf$ are pairwise disjoint, which can be checked using the same idea as in the proof of Lemma \ref{lemma: remaining issue}. Indeed, the morphism $X[[K_n]]^\suf \to \aline \times X[n]^+$ from Proposition \ref{proposition: commutative diagram} must map the sections $x_i^\suf(X[[K_n]]) \subset X[[K_n]]^\suf$ to the Fulton--MacPherson sections (direct product with $\aline$) in $\aline \times X[n]^+$ since it trivially does so in the complement of the central fiber, and hence necessarily everywhere. Since the latter sections are pairwise disjoint, so must be the former, completing the proof.
	
	\section{The small resolutions}\label{section: small resolutions}
	
	Let $C$ be any smooth complex projective curve. It is convenient for exposition (though mathematically completely unimportant) to assume that $C$ is not rational. We will use the results of \S\ref{section: polydiagonal degenerations} for $X=C$. We will keep the notation from \S\ref{subsection: polydiagonal degenerations} and \S\ref{subsection: universal families}, but not from \S\ref{subsection: proof of theorem}.

	\subsection{Stabilizing the curves}\label{subsection: stabilizing the curves}
	By Theorem \ref{theorem: polydiagonal degeneration} and Lemma \ref{lemma: only flatness}, $C[[\graph]]^\fuf \to C[[\graph]]$ and $C[[\graph]]^\suf \to C[[\graph]]$ are prestable curves. In this section, we stabilize them (and their `scaling' structures) by contracting suitable rational bridges. A \emph{rational bridge} is an irreducible component of a nodal curve over $\spec \cc$ which is smooth, rational, and contains precisely two nodes of the curve.
	
	There are $2 \times 2 = 4$ cases in principle, since $\graph$ is either $K_n^c$ or $K_n$, and the decoration `$\hsymb$' on $C[[\graph]]^\hsymb$ is either $\fuf$ or $\suf$, but in fact only $3$ families of curves, since $C[[K_n^c]]^\fuf = C[[K^c_n]]^\suf $. The case $C[[K_n]]^\fuf$ is combinatorially similar to the two identical cases with $\graph = K_n^c$. By Theorem \ref{theorem: polydiagonal degeneration}, in all cases so far, the dual tree of any fiber $C[[\graph]]^\hsymb_z$, $z \in C[[\graph]]_0(\cc)$, is a tree of the form $T({\mathcal H})$ (Definition \ref{definition: leveled tree representations}). We refer to the (rational) irreducible components of $C[[\graph]]^\hsymb_z$ as \emph{terminal} or \emph{non-terminal} depending on whether the corresponding vertex in the dual tree is a leaf or not. We will refer to these cases as \emph{type A} cases. The remaining case of the curve $C[[K_n]]^+ \to C[[K_n]]$ (so $\hsymb = +$, $\graph = K_n$) is combinatorially the most complicated, and we will refer to it as the \emph{type B} case. In this case, the dual tree of any fiber $C[[K_n]]^\suf_z$, $z \in C[[K_n]]_0(\cc)$, is a tree of the form $T({\mathcal H},{\mathcal F})$ (Definition \ref{definition: combined leveled tree representations}), again by Theorem \ref{theorem: polydiagonal degeneration}. We refer to the irreducible components of the fiber as \emph{infinite scaling}, \emph{finite scaling}, and \emph{zero scaling} components correspondingly with the vertex in the dual tree (\S\ref{subsection: Fulton--MacPherson nests}). For convenience, if $z \notin C[[K_n]]_0$, we will sometimes refer to all rational irreducible components of $C[[K_n]]^\suf_z$ as $0$-scaling components. 
	\begin{center}
	\begin{tabular}{c|c|c|c}
		Cases & $(\graph,\hsymb)$ & combinatorics & components \\
		\hline
		Type A & $(K_n^c,\fuf)$, $(K_n^c,\suf)$, $(K_n,\fuf)$ & $T({\mathcal H})$ & terminal/non-terminal \\ 
		Type B & $(K_n,\suf)$ & $T({\mathcal H},{\mathcal F})$ & $\infty$/finite/$0$-scaling \\ 
	\end{tabular}
	\end{center}
	Note that non-terminal or infinite scaling components contain no marked points. 
	
	\begin{definition}\label{definition: components to be contracted}
		Let $\hsymb \in \{\fuf,\suf\}$ and $z \in C[[\graph]](\cc)$. Let $R \simeq {\mathbb P}^1$ be an irreducible component of $C[[\graph]]^\hsymb_z$. For $z \in C[[\graph]]_0(\cc)$, $R$ is a \emph{component to be contracted} if it is a rational bridge, and either non-terminal (Type A case) or of infinite scaling (Type B case). There are no components to be contracted for $z \notin C[[\graph]]_0(\cc)$.
	\end{definition}
	
	In particular, the components to be contracted contain no marked points.
	
	
	\begin{proposition}\label{proposition: stabilization of universal curve}
		For each symbol $\hsymb \in \{\fuf,\suf\}$, there exists a prestable curve $Y^\hsymb$ over $C[[\graph]]$ and a morphism with property R (Definition \ref{definition: rational contraction}) 
		\begin{equation} f^\hsymb:C[[\graph]]^\hsymb \to Y^\hsymb, \end{equation}
		such that, for each closed point $z \in C[[\graph]]$, $f_z^\hsymb$ contracts all components to be contracted (Definition \ref{definition: components to be contracted}) in $C[[\graph]]_z^\hsymb$, and nothing else. Moreover, $f^\hsymb x_i^\hsymb$ is a smooth section of $Y^\hsymb$ over $C[[\graph]]$.
	\end{proposition}
	
	\begin{proof}
		Let $\zeta: C[[\graph]]^\hsymb \to C[[\graph]]^\fuf$ be the map $\xi_{K_n}$ from Proposition \ref{proposition: commutative diagram} in the Type B case, respectively the identity map in the Type A cases, and $\kappa: C[[\graph]]^\hsymb \to C$ the natural morphism. Let ${\sh J}$ be a line bundle on $C$ of large degree, and
		\begin{equation}\label{equation: crucial line bundle} {\sh L} = \omega_{C[[\graph]]^\hsymb/C[[\graph]]}^{\otimes 2} \otimes \zeta^*\omega_{C[[\graph]]^\fuf/C[[\graph]]}^{\vee} \otimes {\sh O}_{C[[\graph]]^\hsymb} \left( 2\sum_{i=1}^n x_i^\hsymb(C[[\graph]]) \right) \otimes \kappa^*{\sh J}. \end{equation}
		
		\begin{lemma}\label{claim: nonnegative degrees claim}
			Let $z \in C [[\graph]](\cc)$ and $R$ an irreducible component of $C[[\graph]]_z^\hsymb$. Then, $ \deg {\sh L}|_R \geq 0$, with equality if and only if $R$ is a component to be contracted.
		\end{lemma}
		
		\begin{proof}
			This is clear if $R \simeq C$ since $\deg {\sh J}$ is large, so assume $R \simeq {\mathbb P}^1$. Let $n_R$ be the number of nodes of $C[[\graph]]_z^\hsymb$ on $R$, and $m_R$ the number of marked points on $R$.
			
			\emph{Type A cases.} Note that \eqref{equation: crucial line bundle} can be written more naturally as
			$$ {\sh L} = \omega_{C[[\graph]]^\hsymb/C[[\graph]]}\left( 2\sum_{i=1}^n x_i^\hsymb(C[[\graph]]) \right) \otimes \kappa^*{\sh J}. $$
			in this case. Note also that $z \in C[[\graph]]_0$ by necessity. We have
			\[ \deg {\sh L}|_R = n_R-2+2m_R \geq 0, \]
			with equality if and only if $R$ is to be contracted, by item \ref{item: combinatorics item in polydiagonal theorem} in Theorem \ref{theorem: polydiagonal degeneration} and Definition \ref{definition: leveled tree representations}.
			
			\emph{The type B case.} The case $z \notin C[[K_n]]_0$ is straightforward. Indeed, by \eqref{equation: crucial line bundle},
			\[ \deg {\sh L}|_R = 2(n_R-2) - 0+ 2m_R + 0 > 0, \]
			which amounts to the well-known fact that any rational irreducible component of any fiber of $C[n]^+ \to C[n]$ has at least three `special points', where both nodes and marked points are considered special. It remains to analyze the case $z \in C[[K_n]]_0$. If $R$ has infinite scaling, then, by \eqref{equation: crucial line bundle},
			$$ \deg {\sh L}|_R = 2(n_R-2) - (n_R-2)+2m_R = n_R - 2 \geq 0, $$
			with equality if and only if $R$ is to be contracted. If $R$ has finite scaling, then
			\[ \deg {\sh L}|_R = 2(n_R-2) - (-1) +2m_R = 2n_R+2m_R-3 > 0, \]
			since, if $n_R=1$ (that is, $R$ intersects no zero scaling components), then $R$ must contain at least one marked point. Finally, if $R$ has zero scaling, then
			\[ \deg {\sh L}|_R = 2(n_R-2) - 0 +2m_R = 2(n_R+m_R-2) >0, \]
			since $n_R+m_R \geq 3$ is clear from item \ref{item: combinatorics complete case item in polydiagonal theorem} in Theorem \ref{theorem: polydiagonal degeneration} and Definition \ref{definition: combined leveled tree representations}.
		\end{proof}
		
		In light of Definition \ref{definition: components to be contracted}, item \ref{item: combinatorics complete case item in polydiagonal theorem} in Theorem \ref{theorem: polydiagonal degeneration}, and especially Lemma \ref{claim: nonnegative degrees claim}, we may apply Proposition \ref{proposition: contracting bridges semi-general} to conclude the existence of $f^\hsymb$. The smoothness of $f^\hsymb x_i^\hsymb$ follows from the earlier remark that components to be contracted contain no marked points. Finally, $f^\hsymb$ has property R by Lemma \ref{lemma: has property R in general} (or Remark \ref{remark: has property R in general}). 
	\end{proof}
	
	Combinatorially, stabilizing the curves amounts to passing from $T({\mathcal H})$ or $T({\mathcal H},{\mathcal F})$ to $T^s({\mathcal H})$ or $T^s({\mathcal H},{\mathcal F})$ (please see \S\ref{section: combinatorial language}). The irreducible components of the curves $Y^\hsymb_z$ inherit the descriptions terminal/non-terminal, respectively $\infty$/finite/$0$-scaling from the corresponding components of $C[[\graph]]^\hsymb_z$. Our `finite scaling components' are the `transition components' from \cite[\S2]{[GSW17]}. 
	
	The lemma below is relevant only when $\graph$ is complete. 
	
	\begin{lemma}
		There exists a (unique) $C[[\graph]]$-morphism $\mu:Y^\suf \to Y^\fuf$ such that
		\begin{center}
		\begin{tikzpicture}
			\node (sw) at (3,1.4) {$C[[\graph]]^\fuf$}; 
			\node (nw) at (0,1.4) {$C[[\graph]]^\suf$};
			\node (se) at (3,0) {$Y^\fuf$}; 
			\node (ne) at (0,0) {$Y^\suf$};
			
			\draw [->] (nw) --node[above] {$\xi_\graph$} (sw); 
			\draw [->] (ne) --node[above] {$\mu$} (se); 
			\draw [->] (nw) --node[left] {$f^\suf$} (ne); 
			\draw [->] (sw) --node[right] {$f^\fuf$} (se);
		\end{tikzpicture}
	\end{center}
	is a commutative diagram of prestable curves over $C[[\graph]]$, and all morphisms in the diagram have property R (Definition \ref{definition: rational contraction}). Moreover, if $\graph=K_n$, for any $z \in C[[\graph]] (\cc)$, $\mu_z$ contracts all zero scaling components of $Y^\suf_z$ and nothing else.
	\end{lemma}
	
	\begin{proof}
		By item \ref{item: combinatorics complete case item in polydiagonal theorem} in Theorem \ref{theorem: polydiagonal degeneration} and Lemma \ref{lemma: has property R in general}, $\xi_\graph$ has Property R. In particular, this and Proposition \ref{proposition: stabilization of universal curve} imply
		\begin{equation}\label{equation: pushforward for BM cite} f^\fuf_* (\xi_\graph)_* {\sh O}_{C[[\graph]]^\suf} = f^\fuf_*{\sh O}_{C[[\graph]]^\fuf} = {\sh O}_{Y^\fuf}. \end{equation}
		By Proposition \ref{proposition: stabilization of universal curve}, $f^\suf_z$ is constant on the fibers of $(f^\fuf \xi_\graph)_z$ for any $z \in C[[\graph]] (\cc)$, so, with \eqref{equation: pushforward for BM cite} in mind, \cite[Lemma 2.2]{[BM96]} gives a $C[[\graph]]$-morphism $\mu:Y^\suf \to Y^\fuf$ such that $\mu f^\suf = f^\fuf \xi_\graph$. The last claim is clear by construction, and then it can be checked on fibers (Lemma \ref{lemma: has property R in general}) that $\mu$ is a has property R.
	\end{proof}
	
	From now on, we will often write $\pi_{A/B}$ for a map $A \to B$ that is `naturally' a projection. Context should leave no reasonable doubts about what the map is.
	
	We also need to check that the `scaling' descends along the maps with property R in Proposition \ref{proposition: stabilization of universal curve}, which requires some technicalities from \cite{[Za21]}. By Lemma \ref{lemma: defining property of root},
	\begin{equation}\label{equation: trivial hom isomorphism}
		{\sh H}om ( \omega_{C[[\graph]]^\fuf /C[[\graph]]}, \pi_{C[[\graph]]^\fuf/C}^*\omega_C ) = {\sh H}om (\omega_{C[[\graph]]^\fuf,C/C[[\graph]]},{\sh O}_{C[[\graph]]^\fuf} ),
	\end{equation}
	so we may think of $\psi$ (Theorem \ref{theorem: polydiagonal degeneration}) as a global section of either.
	
	\begin{proposition}\label{proposition: descending scaling}
		There exists a (unique) ${\sh O}_{Y^\fuf}$-module homomorphism
		\begin{equation*}
			\phi: \omega_{Y^\fuf/C [[\graph]]} \to \pi_{Y^\fuf/C}^*\omega_C
		\end{equation*}
		such that the restrictions of $\psi$ (seen as a global section of the left hand side of \eqref{equation: trivial hom isomorphism}) and $\phi$ to the complement of the exceptional locus of $f^\fuf$ and the isomorphic image in $Y^\fuf$ coincide. Moreover, for any $z \in C[[\graph]](\cc)$, $\phi_z$ is equal to $0$ on a rational irreducible component of $Y^\fuf_z$ if and only if that component is not terminal, and it is always nonzero at $f^\fuf(x_i^\fuf(z))$ for $i=1,\ldots,n$.
	\end{proposition}
	
	\begin{proof} By Proposition \ref{proposition: stabilization of universal curve}, \cite[Lemma 3.7]{[Za21]}, and Lemma \ref{lemma: defining property of root},
		\begin{equation}\label{equation: some calculations with omega}
			\begin{aligned}
				(f^\fuf)^!(\omega_{Y^\fuf/C [[\graph]]} \otimes \pi_{Y^\fuf/C}^*\omega_C^\vee)
				&=\omega_{C[[\graph]]^\fuf/C[[\graph]]} \otimes (f^\fuf)^*\omega_{Y^\fuf/C[[\graph]]}^\vee \otimes (f^\fuf)^*(\omega_{Y^\fuf/C [[\graph]]} \otimes \pi_{Y^\fuf/C}^*\omega_C^\vee)  \\
				&= \omega_{C [[\graph]]^\fuf/C [[\graph]]} \otimes \pi_{C [[\graph]]^\fuf/C}^*\omega_C^\vee =\omega_{C [[\graph]]^\fuf,C/C [[\graph]]}.
			\end{aligned}
		\end{equation}
		With the `<' notation from \cite[Definition 3.12]{[Za21]}, \eqref{equation: some calculations with omega} reads $(f^\fuf)^{\text{<}}(\omega^\vee_{Y^\fuf/C [[\graph]]} \otimes \pi_{Y^\fuf/C}^*\omega_C) = \omega^\vee_{C [[\graph]]^\fuf,C/C [[\graph]]}$. 
		By \cite[Proposition 3.14]{[Za21]}, we obtain a homomorphism
		\begin{equation*}
			f^\fuf_*\omega^\vee_{C [[\graph]]^\fuf,C/C [[\graph]]} \to \omega^\vee_{Y^\fuf/C [[\graph]]} \otimes \pi_{Y^\fuf/C}^*\omega_C.
		\end{equation*}
		We define $\phi$ to be the image of $\psi$ under the homomorphism above. The defining property of $\phi$ follows from \cite[Proposition 3.14]{[Za21]}, and uniqueness is trivial. The claims regarding the vanishing of $\phi_z$ follow from the analogous facts for $\psi_z$ (Theorem \ref{theorem: polydiagonal degeneration}) and the definition of $f^\fuf$ (Proposition \ref{proposition: stabilization of universal curve}).
	\end{proof}
	
	Proposition \ref{proposition: descending scaling} descends the vector field from $C[[\graph]]^\fuf$ to $Y^\fuf$, which will suffice for Theorem \ref{theorem: small resolution theorem}. However, for the purposes of Theorem \ref{theorem: small resolution theorem 2}, we need to rearrange this in terms of $C[[\graph]]^\suf$ and $Y^\suf$, which is done in Proposition \ref{proposition: descending scaling 2} below. 
	
	\begin{proposition}\label{proposition: descending scaling 2}
		Assume that $\graph$ is complete. Let $P$ be the ${\mathbb P}^1$-bundle
		$$ P = \underline{\mathrm{Proj}}_{Y^\suf} \mathrm{Sym}^* (\omega_{Y^\suf/C[[\graph]]} \oplus \pi_{Y^\suf/C}^*\omega_C) \to Y^\suf. $$
		The sections of $P$ corresponding to the projections to the two factors of the direct sum are the $0$-section and $\infty$-section respectively.
		
		Then, there exists a (unique) section $\sigma: Y^\suf \to P$ of $P$ such that $\sigma(\zeta)$ equals
		\begin{equation*}
		\begin{aligned} 
		\mu^*\phi(\zeta) &\in \spec \mathrm{Sym}^* (\mu^* (\omega_{Y^\fuf/C [[\graph]]}^\vee \otimes \pi_{Y^\fuf/C}^*\omega_C))_{\kappa(\zeta)} = \\
		& =  \spec \mathrm{Sym}^* (\omega_{Y^\suf/C [[\graph]]}^\vee \otimes \pi_{Y^\suf/C}^*\omega_C)_{\kappa(\zeta)} =P_{\kappa(\zeta)} \backslash \{ 0 \} \simeq {\mathbb A}^1_{\kappa(\zeta)}
		\end{aligned}
		\end{equation*}
		under the natural identifications, where $\zeta$ is the generic point of $Y^\suf$. 
		
		Moreover, for any $z \in C[[\graph]]_0$, $\sigma_z$ is equal to $\infty$ on all infinite scaling components of $Y^\suf_z$, to $0$ on all $0$ scaling components, and generically to neither $0$ or $\infty$ on all finite scaling components. 
	\end{proposition}
	
	\begin{proof}
		By standard facts (e.g. \cite[\href{https://stacks.math.columbia.edu/tag/0ATX}{Tag 0ATX}]{[SP]}), $ \mu^! \omega_{Y^\fuf/ C [[\graph]]} = \mu^! \pi^!_{Y^\fuf/ C [[\graph]]} {\sh O}_{C [[\graph]]} = \pi^!_{Y^\suf/ C [[\graph]]} {\sh O}_{C [[\graph]]} = \omega_{Y^\suf/ C [[\graph]]}$, 
		and then \cite[Proposition 3.10]{[Za21]} gives a homomorphism 
		$$ \mu^* \omega_{Y^\fuf/ C [[\graph]]} \to \mu^! \omega_{Y^\fuf/ C [[\graph]]} = \omega_{Y^\suf/ C [[\graph]]}. $$
		Let $d\mu$ be the dual of the map above, so $d\mu$ is essentially the relative logarithmic differential of $\mu$. Clearly, $d\mu$ is an isomorphism on the complement of the exceptional locus of $\mu$. On the other hand,
		$$ \pi^*_{Y^\suf/C} \omega_C^\vee \xrightarrow{ \mu^*(\phi^\vee) } \mu^* \omega^\vee_{Y^\fuf/C[[\graph]]} $$
		is an isomorphism on the preimage of $\mu$ of the open set on $Y^\fuf$ which consists of points which live (only on) terminal components of the respective fiber of $Y_0^\fuf \to C[[\graph]]_0$ or away from the central fiber. These two open subsets cover $Y^\suf$, so
				\begin{equation}\label{equation: surjective map giving section of P^1-bundle}
			 \pi_{Y^\suf/C}^*\omega^\vee_C \oplus \omega_{Y^\suf/C[[\graph]]}^\vee \xrightarrow{ 
				\left[ \begin{smallmatrix}
					\mu^*(\phi^\vee) \\
					d\mu
				\end{smallmatrix} \right] } \mu^* \omega^\vee_{Y^\fuf/C[[\graph]]}
		\end{equation}
		is surjective. Twisting \eqref{equation: surjective map giving section of P^1-bundle} by $\omega_{Y^\suf/C[[\graph]]} \otimes \pi_{Y^\suf/C}^*\omega_C$, we obtain a surjection
		$$ \omega_{Y^\suf/C[[\graph]]} \oplus \pi_{Y^\suf/C}^*\omega_C \to \omega_{Y^\suf/C[[\graph]]} \otimes \pi_{Y^\suf/C}^*\omega_C \otimes \mu^* \omega^\vee_{Y^\fuf/C[[\graph]]} \to 0. $$ 
		The section is then obtained from the well-known universal property, and the required properties are easy to check.
	\end{proof}
	
	\subsection{The wonderful models as subvarieties of $C[[\graph]]$}\label{subsection: wonderful models as subvarieties} In this subsection, we identify $W_n$ (respectively $W'_n$) with a codimension $2$ subvariety of $C[[\graph]]$ when $\graph$ has no edges (respectively is complete). The varieties $W_n$ and $W'_n$ were constructed in Definition \ref{definition: initial definition of Wn and Tn}, and the constructions were slightly rearranged in \S\ref{subsection: rearranged definition}. Then, we remove the components isomorphic to $C$ from the restriction to this codimension $2$ locus of the families of curves considered in \S\ref{subsection: stabilizing the curves}. In the next subsection, we will see that the resulting family of genus $0$ curves satisfies the properties in the description of the functor of points of $\overline{P}_n$ (respectively $\overline{Q}_n$). 
		
	Let $W_\graph$ be $W_n$ if $\graph$ has no edges, respectively $W'_n$ if $\graph$ is complete. Let ${\mathrm F}(\aline,\graph)$ as in \S\ref{section: polydiagonal degenerations}, namely, ${\mathbb A}^n$ if $\graph$ has no edges, respectively ${\mathrm F}(\aline, n)$ if $\graph$ is complete. Then $W_\graph$ is a compactification of ${\mathrm F}(\aline,\graph)/{\mathbb G}_a$, where ${\mathbb G}_a$ acts diagonally. 
	
	We have the following boundary divisors on $W_\graph$.
	\begin{enumerate}
		\item For each partition $\rho \in L_n^* = L_n \backslash \{\bot\}$, the divisor $D_\rho \subset W_\graph$ is the preimage of ${\mathbb P}(0 \oplus P/\cc)$, where $P$ is the polydiagonal of $\cc^n$ corresponding to $\rho$.
		\item If $\graph = K_n$, for each $S \subseteq [n]$, $|S| \geq 2$, the divisor $E_S \subset W'_n$ is the preimage of ${\mathbb P}(\cc \oplus D/\cc)$, where $D$ is the diagonal of $\cc^n$ corresponding to $S$. 
	\end{enumerate}
	The open stratum is denoted by $W_\graph^\circ \simeq {\mathrm F}(\aline,\graph)/{\mathbb G}_a$, while the codimension $1$ strata are denoted by $D_\rho^\circ$ and $E_S^\circ$, and naturally consist of all points in the respective boundary divisors which don't belong to any other boundary divisors. 
	
	Let us fix a closed point $p \in C$ and an isomorphism $T_pC \simeq \cc$.
	
	\begin{proposition}
	The variety $W_\graph$ is isomorphic to the fiber of $C[[\graph]] \to \aline \times C^n$ 
	over $(0,p,\ldots,p)$. Moreover, $D_\rho = W_\graph \cap D_\rho^v$ for any $\rho \in L_n^*$, and, if $\graph$ is complete, $E_S = W_\graph \cap D_S^h$ for any $|S| \geq 2$,
	and the intersections are transverse.
	\end{proposition}
	
	\begin{proof}
	The fiber of the blowup of $\aline \times C^n$ at $\{0\} \times \Delta_\bot$ over $(0,p,\ldots,p)$ is
	\begin{equation}\label{equation: fiber of exceptional at p} \Pi:={\mathbb P}(T_0\aline \oplus T_pC^{\oplus n}/T_pC) = {\mathbb P}(\cc \oplus \cc^n/\cc). \end{equation}
	For any $\rho \in L_n^*$ and for any $S \subseteq [n]$ with $|S| \geq 2$, $ \widetilde{\{0\} \times \Delta_\rho} \cap \Pi = {\mathbb P}(0 \oplus \Delta_{\rho,\cc^n}/\cc)$ and $\widetilde{\aline \times \Delta_S} \cap \Pi = {\mathbb P}(\cc \oplus \Delta_{S,\cc^n}/\cc)$
	transversally inside the blowup, where $\Delta_{\rho,\cc^n}$ and $\Delta_{S,\cc^n}$ are the respective polydiagonal and diagonal in $\cc^n$. Since ${\mathcal P}_\graph$ is a building set in $\aline \times C^n$ (\S\ref{subsection: polydiagonal degenerations}) and $\{0\} \times \Delta_\bot \in {\mathcal P}_\graph$ is minimal, it follows from \cite[Propositon 2.8]{[Li09]} that the proper transforms above (the latter included only if $\graph$ is complete) form a building set in the blowup, and, by \cite[Definition 2.12]{[Li09]}, the resulting wonderful compactification coincides with $C[[\graph]]$. Therefore, the claim follows from Proposition \ref{proposition: functoriality of wc}.
	\end{proof}
	
	If ${\mathcal C} \to S$ is a family of e.g. curves over a base $S$ and $T \to S$ is a morphism, we write ${\mathcal C}_T$ for the pullback $T \times_S {\mathcal C}$ to $T$.
	
	\begin{lemma}\label{lemma: deleting C}
	Let $\hsymb \in \{\fuf,\suf\}$. Then, $C [[\graph]]^\hsymb_{W_\graph}$ has an irreducible component isomorphic to $W_\graph \times C$, and $Y^\hsymb_{W_\graph}$ has an irreducible component isomorphic to $W_\graph \times C$ as well. Moreover, if $Z^\hsymb$ is the union of all the irreducible components of $Y^\hsymb_{W_\graph}$ other than the one isomorphic to $W_\graph \times C$, then $(W_\graph \times C) \cap Z^\hsymb$ is a (singular) section of $C [[\graph]]^\hsymb_{W_\graph}$ over $W_\graph$, which is smooth as a section of $Z^\hsymb$ over $W_\graph$.
	\end{lemma}
	
	(In fact, $Z^\hsymb$ is irreducible, but it is not important for now.)
	
	\begin{proof}
	Recall (from the proof of Proposition \ref{proposition: commutative diagram} or the second part of the proof of Theorem \ref{theorem: polydiagonal degeneration}) that $C [[\graph]]^\hsymb$ was obtained from $C [[\graph]] \times C$ by a sequence of blowups indexed by ${\mathfrak l}^\star$. After the first blowup in this sequence, the restriction of the family of curves to $D^v_\bot$ has a component isomorphic $D^v_\bot \times C$ and one isomorphic to a ${\mathbb P}^1$-bundle over $D^v_\bot$. However, all subsequent blowups are disjoint from the component isomorphic to $D^v_\bot \times C$, since we're always blowing up loci inside the union of the marked sections by item \ref{item: loci to be blown up are sections in lemma: big induction in round 2 of blowing up} in Claim \ref{claim: big induction in round 2 of blowing up}, and these sections have already been moved away from $C$ above $D^v_\bot$ at the first blowup. In particular, $C [[\graph]]^\hsymb_{W_\graph}$ has an irreducible component isomorphic to $W_\graph \times C$. This component intersects the union of the other components at $W_\graph \times \{p\}$. 
	
	For $Y^\hsymb_{W_\graph}$, it is easy to check that the composition
	\begin{equation}\label{equation: to be immersion of component}
		W_\graph \times C \hookrightarrow  C [[\graph]]^\hsymb_{W_\graph} \xrightarrow{ f^\hsymb|_{C [[\graph]]^\hsymb_{W_\graph}} } Y^\hsymb_{W_\graph}, 
	\end{equation} 
	is the immersion of an irreducible component. For instance, \eqref{equation: to be immersion of component} is obviously `fiberwise' unramified by Proposition \ref{proposition: stabilization of universal curve} (it is fiberwise a closed immersion), hence unramified \cite[\href{https://stacks.math.columbia.edu/tag/0475}{Tag 0475}]{[SP]}, hence a closed immersion \cite[\href{https://stacks.math.columbia.edu/tag/04XV}{Tag 04XV}]{[SP]}, and it must be an irreducible component for dimension reasons.
	\end{proof}
	
	\begin{remark}\label{remark: generalization to any dimension}
		Much of the above could have been carried out for general $X$. Then, we would have obtained generalizations of $W_n$ and $W'_n$, which compactify spaces of configurations of $n$ points (distinct, for $W'_n$) in ${\mathbb A}^d$ (rather than ${\mathbb A}^1$) modulo translation, and even universal families over these compactifications. Thus, these compactifications generalizing $W_n$ and $W'_n$ are obtained from our polydiagonal degenerations in a manner quite similar to that in which the Chen--Gibney--Krashen space $T_{d,n}$ \cite{[CGK09]} is obtained from $X[n]$. On the other hand, generalizing $\overline{P}_n$ and $\overline{Q}_n$ is an extremely subtle issue, which we hope to address in future work.
	\end{remark}
	
	\subsubsection*{Summary} Let us summarize what we have constructed and proved so far. Continuing to use the notation in Lemma \ref{lemma: deleting C}, $f^\hsymb x_i^\hsymb (W_\graph) \subset Z^\hsymb$ holds, so let 
	$$ z_i^\hsymb = (f^\hsymb x_i^\hsymb)|_{W_\graph}: W_\graph \to Z^\hsymb. $$
	Let also $z_\infty^\hsymb: W_\graph \to Z^\hsymb$ be the section whose image is $z^\hsymb_\infty(W_\graph) = (W_\graph \times C) \cap Z^\hsymb$. Note that $Z^\hsymb \to W_\graph$ is a prestable curve of genus $0$. Indeed, flatness away from the section $z_\infty^\hsymb$ follows from flatness of $Y^\hsymb \to C[[\graph]]$, and flatness along the section $z_\infty^\hsymb$ follows from smoothness. The sections $z_1^\hsymb,\ldots,z_n^\hsymb,z_\infty^\hsymb$ are smooth, $z_i^\hsymb(s) \neq z_\infty^\hsymb(s)$ for all $s \in W_\graph$, and, if $\graph$ is complete, then $z_i^\suf(s) \neq z^\suf_j(s)$, for all $s \in W_\graph$ and $i \neq j$, which all follow from Theorem \ref{theorem: polydiagonal degeneration}, Proposition \ref{proposition: stabilization of universal curve}, and Lemma \ref{lemma: deleting C}.
	
	By \ref{item: combinatorics item in polydiagonal theorem} in Theorem \ref{theorem: polydiagonal degeneration}, Proposition \ref{proposition: stabilization of universal curve}, and Lemma \ref{lemma: deleting C}, we can determine the dual tree of $Z^\hsymb_s$, for any $s \in W_\graph(\cc)$. Indeed, it is obtained from the corresponding dual tree of $C[[\graph]]^\star$ over $s \in C[[\graph]]$ by `stabilizing' it (passing from $T({\mathcal H})$ to $T^s({\mathcal H})$, or from $T({\mathcal H},{\mathcal F})$ to $T^s({\mathcal H},{\mathcal F})$, depending on the situation) -- the combinatorial effect of Proposition \ref{proposition: stabilization of universal curve}, and removing the root -- the combinatorial effect of Lemma \ref{lemma: deleting C}. In particular, skipping ahead, combinatorial conditions such as the stability conditions for $\overline{P}_n$ (\cite[Theorem 1.3]{[Za21]}) and $\overline{Q}_n$ (\cite[Definition 2.4]{[GSW17]}), or condition (a) in \cite[Definition 2.4]{[GSW17]} will be trivial to check, so handled implicitly.
	
	The restriction $\phi|_{Z^\fuf}$ of $\phi$ (cf. Proposition \ref{proposition: descending scaling}) to $Z^\fuf$ is a priori a global section of
	$ \omega_{Y^\fuf/C[[\graph]]}^\vee |_{Z^\fuf} \simeq \omega_{Z^\fuf/W_\graph}^\vee(-z^\fuf_\infty) $ (the isomorphism follows from classical facts about dualizing sheaves, e.g. \cite[Chapter X, \S2]{[ACG10]}), but it must vanish along the section one more time since $\phi|_{W_\graph \times C} = 0$ (by Theorem \ref{theorem: polydiagonal degeneration} and Propositions \ref{proposition: descending scaling} and \ref{proposition: stabilization of universal curve}), so, 
	\begin{equation}\label{equation: double vanishing} \phi|_{Z^\fuf} \in H^0(Z^\fuf, \omega_{Z^\fuf/W_\graph}^\vee(-2z^\fuf_\infty)). \end{equation}
	By Theorem \ref{theorem: polydiagonal degeneration}, Proposition \ref{proposition: stabilization of universal curve}, and Proposition \ref{proposition: descending scaling} again, 
	\begin{equation}\label{equation: non-vanishing} \phi|_{Z^\fuf} ( z_i^\fuf(s) ) \neq 0, \end{equation}
	for all $s \in W_\graph$, and $i=1,\ldots,n$. If $\graph$ is complete and $s \in W_\graph(\cc)$, we can determine from Proposition \ref{proposition: descending scaling 2}, whether the restriction of the scaling $\sigma|_{Z^\suf_s}$ on a given component of $Z^\suf_s$ is $0$, generically nonzero and finite, or infinite. In particular, $\sigma|_{Z^\suf_s}$ is infinite at $z_\infty^+(s)$ and finite (possibly $0$) at $z_i^\suf(s)$, for any $s \in W_\graph$. 

	\subsection{The small resolutions of $\overline{P}_n$ and $\overline{Q}_n$} The objects summarized above are `points' of $\overline{P}_n$ (respectively $\overline{Q}_n$) if $\graph$ has no edges (respectively, if $\graph$ is complete, and $\hsymb = \suf$). Thus, we obtain the morphisms in Theorems \ref{theorem: small resolution theorem} and \ref{theorem: small resolution theorem 2}, and all that remains is to check that they are indeed small resolutions.
	
	Recall that $P_n \subset \overline{P}_n$ and $Q_n \subset \overline{Q}_n$ are the open subsets where the parametrized curves are smooth.
	
	\subsubsection*{The case of $\overline{P}_n$} 
	Assume that $\graph$ has no edges. Let $Z=Z^\fuf=Z^\suf$, $z_i = z_i^\fuf=z_i^\suf$, and $z_\infty = z_\infty^\fuf=z_\infty^\suf$. Recall that $W_\graph = W_n$ and $W_n^\circ = W_\graph^\circ \simeq {\mathbb A}^n/{\mathbb G}_a$. With the summary at the end of \S\ref{subsection: wonderful models as subvarieties} in mind, $ (Z,z_1,\ldots,z_n,z_\infty,\phi|_Z) $ induces a morphism
	\begin{equation} \gamma:W_n \to \overline{P}_n \end{equation}
	by \cite[Theorem 1.3]{[Za21]}. Indeed, the first three items in loc. cit. are stated explicitly in our summary, condition (1) is \eqref{equation: double vanishing}, condition (2) is \eqref{equation: non-vanishing}, and the stability condition (3) is clear (for $\cc$-points, hence for all geometric points).
	
	\begin{remark}\label{remark: relation between W curve and P curve}
	Let $s \in W_n \subset C[[K_n^c]]$ be a $\cc$-point, and let
	$$ {\mathcal H} = \{\rho \in L_n: s \in D^v_\rho\} = \{\bot\} \cup \{\rho \in L^*_n: s \in D_\rho\}. $$
	Let us review the combinatorics of the various curves over $s$.
	\begin{itemize}
		\item The dual tree of $C[[K_n^c]]_s^\fuf$ is $T({\mathcal H})$ (Definition \ref{definition: leveled tree representations}) by Theorem \ref{theorem: polydiagonal degeneration}.
		\item If we remove the component isomorphic to $C$ from $C[[K_n^c]]_s^\fuf$, the dual tree is $T({\mathcal H})-\star$. This is still connected and rooted, since $\bot \in {\mathcal H}$, see the discussion in \S\ref{section: combinatorial language} after Definition \ref{definition: leveled tree representations}.
		\item The dual tree of $Z_s$ is $T^s({\mathcal H}) - \star$ by Proposition \ref{proposition: stabilization of universal curve} and the constructions above. By definition, $Z_s$ is also the curve corresponding to $\gamma(s) \in \overline{P}_n$.
	\end{itemize}
	\end{remark}
	
	\begin{proposition}\label{proposition: birational}
		The restriction of $\gamma$ to $W_n^\circ$ is an isomorphism, and $\gamma^{-1}(P_n) = W_n^\circ$. In particular, $\gamma$ is birational.
	\end{proposition}
	
	\begin{proof}
		The fact that $\gamma^{-1}(P_n) = W_n^\circ$ follows from Remark \ref{remark: relation between W curve and P curve} and a simple combinatorial argument. It remains to check that $\gamma$ restricts on $W_n^\circ$ to an isomorphism $W_n^\circ \simeq P_n$. Since $\gamma^{-1}(P_n) = W_n^\circ$ and $P_n$ is smooth (even normal suffices, which follows from \cite[Proposition 6.7]{[Za21]}), it suffices to check that $\gamma$ is injective on $W_n^\circ$. This can be done by an explicit calculation, which is trivial after unwinding the constructions. For $s \in W^\circ_n(\cc)$, we have $C[[K_n^c]]_s^\fuf  = C \cup_p {\mathbb P}^1 = C \cup_p Y_s^\fuf$, so $Y_s^\fuf = {\mathbb P}^1$. With \eqref{equation: fiber of exceptional at p} in mind, $W^\circ_n$ is identified with the open $\cc^n/\cc \subset \Pi$. Assuming that $s = (x_1,\ldots,x_n) + \cc(1,\ldots,1)$, it can be checked by unwinding the definitions that $\gamma(s) = ({\mathbb P}^1,x_1,\ldots,x_n,[1:0],\partial/\partial x) \in P_n(\cc)$ cf. \S\ref{subsection: review of Q and P}, so $\gamma$ is injective on $W_n^\circ$. 
	\end{proof}
	
	It remains to check that $\gamma$ is small, but not an isomorphism if $n \geq 4$.
	
	\begin{proposition}\label{proposition: isomorphism in codimension 2}
		The birational morphism $\gamma$ is small.
	\end{proposition}
	
	\begin{proof}
		Let $R \subset \overline{P}_n$ be the closed set where the fibers of $\gamma$ are positive dimensional and $E = \gamma^{-1}(R)$. The restriction $W_n \backslash E \to \overline{P}_n \backslash R$ of $\gamma$ is a finite (since proper and quasi-finite) birational morphism with normal target (since $\overline{P}_n$ is normal by \cite[Proposition 6.7]{[Za21]}), hence an isomorphism. Thus, if we show that $E$ has codimension at least $2$ in $W_n$, then the proposition follows.
		
		By Proposition \ref{proposition: birational}, $W_n^\circ \cap E = \emptyset$, so $E \subseteq \partial W_n = W_n \backslash W_n^\circ$.
		
		By standard facts about stratifications of moduli stacks of prestable curves, for each $\rho \in L_n^*$ there exists a locally closed subset $P_n[\rho] \subset \overline{P}_n$ which parametrizes curves whose dual tree is a star with $|\blocks(\rho)|$ leaves (that is, $K_{1,|\blocks(\rho)|}$), such that the $\infty$ marking is on the central component, and the tail component corresponding to the block $B \in \blocks(\rho)$ contains the marked points indexed by $B$. We claim that
		\begin{equation}\label{equation: stratum dimension on P side}
			\dim P_n[\rho] = n-2.
		\end{equation}
		Heuristically, there are $|\blocks(\rho)| - 2$ moduli on the backbone ($|\blocks(\rho)|$ nodes, the $\infty$ marking, and $0$ vector field), and $|B|-1$ moduli on the tail corresponding to the block $B$ (one node, $|B|$ marked points, and nonzero vector field), which adds up to $|\blocks(\rho)| - 2 + \sum_{B \in \blocks(\rho)} (|B|-1) = n-2$. 
		More formally, it is obvious that $\dim P_n[\rho] < \dim \overline{P}_n = n-1$. On the other hand, $\dim P_n[\rho] \geq n-2$ can be justified by exhibiting a family of $n$-marked ${\mathbb G}_a$-rational trees over an $(n-2)$-dimensional base, such that the map to $\overline{P}_n$ it induces by \cite[Theorem 1.3]{[Za21]} is quasi-finite and has its image contained in $P_n[\rho] $. For instance, it is easy to construct such a family as divisors in ${\mathbb P}^1 \times {\mathbb P}^1$ consisting of one horizontal line and $|\blocks (\rho)|$ vertical lines in the rulings.
		
		Recall that $D^\circ_\rho$ is irreducible of dimension $n-2$ for any $\rho \in L_n^*$. We have
		\begin{equation}\label{equation: covers codim up to 2}
			\dim \partial W_n \backslash \bigcup_{\rho \in L_n^*} D_\rho^\circ = \dim \bigcup_{\rho_1 \neq \rho_2} (D_{\rho_1} \cap D_{\rho_2}) = n-3, 
		\end{equation}
		provided that $n \geq 3$. Theorem \ref{theorem: polydiagonal degeneration}, Remark \ref{remark: relation between W curve and P curve}, and Lemma \ref{lemma: spider has short legs} imply that
		\begin{equation}\label{equation: compatibility of W and P strata}
			\gamma^{-1}(P_n[\rho]) = D_\rho^\circ
		\end{equation}
		for any $\rho \in L_n^*$. In particular, $\gamma(D_\rho^\circ) = P_n[\rho]$ since $\gamma$ is surjective, so, by \eqref{equation: stratum dimension on P side} and the irreducibility of $D^\circ_\rho$, $P_n[\rho]$ is irreducible of dimension $n-2$. 
		
		Assume by way of contradiction that $E$ was of codimension $1$ in $W_n$. By $E \subseteq \partial W_n$ and \eqref{equation: covers codim up to 2}, there exists $\rho \in L_n^*$ such that $D_\rho^\circ \subseteq E$. Then $\dim P_n[\rho] < \dim D_\rho^\circ = n-2$ by \eqref{equation: compatibility of W and P strata} (keeping in mind that $P_n[\rho]$ is irreducible), which contradicts \eqref{equation: stratum dimension on P side}. 
	\end{proof}
	
	\begin{proposition}\label{proposition: not an isomorphism in general}
		The morphism $\gamma$ is not an isomorphism if $n \geq 4$.
	\end{proposition}
	
	\begin{proof}
		We will use the notation from Lemma \ref{lemma: binary tree}. By Lemma \ref{lemma: binary tree}, there exists a unique $b \in \overline{P}_n(\cc)$ such that the curve over $b$ has dual tree $T^s({\mathcal H}) - \star$. Note that
		$$ Y: = \bigcap_{i=0}^{k-1} D_{\rho_i} \neq \emptyset \quad \text{and} \quad \dim Y = \dim W_n - k = n - 1 - \lceil \log_2 n \rceil, $$
		by \cite[Theorem 1.2]{[Li09]}. Hence, $\dim Y > 0$ for $n \geq 4$. However, $\gamma(Y) = \{b\}$ by the second part of Lemma \ref{lemma: binary tree} and Remark \ref{remark: relation between W curve and P curve}, hence $\gamma$ is not an isomorphism. 
	\end{proof}
	
	\begin{remark}\label{remark: not IH small proof}
		The argument in the proof of Proposition \ref{proposition: not an isomorphism in general} also proves Remark \ref{remark: not IH small}. If $r = n - 1 - \lceil \log_2 n \rceil$, then $\dim \gamma^{-1}(b) \geq r$, so  $ \mathrm{codim} \{p \in \overline{P}_n: \dim \gamma^{-1}(p) \geq r \} \leq n-1 \leq 2r$ 
		for $n \geq 7$. The same argument will prove the analogous statement about $\overline{Q}_n$. 
	\end{remark}
	
	Propositions \ref{proposition: birational}, \ref{proposition: isomorphism in codimension 2}, and \ref{proposition: not an isomorphism in general} complete the proof of Theorem \ref{theorem: small resolution theorem}. Any small birational morphism of smooth projective varieties is an isomorphism \cite[Corollary 2.63]{[KM98]}, which confirms that $\overline{P}_n$ is singular for all $n \geq 4$.
	
	\subsubsection*{The case of $\overline{Q}_n$} The arguments for $\overline{Q}_n$ are in essence analogous, so we will focus on the new ingredients and the few minor differences. 
	
	Assume that $\graph$ is complete. Let $Z=Z^\suf$, $z_i = z_i^\suf$, and $z_\infty =z_\infty^\suf$. We have $W_\graph = W'_n$ and $(W'_n)^\circ=W_\graph^\circ \simeq \mathrm{F}(\aline,n)/{\mathbb G}_a$. With the summary at the end of \S\ref{subsection: wonderful models as subvarieties} in mind, $(Z,z_\infty,z_1,\ldots,z_n,\sigma|_Z)$ induces a morphism
	\begin{equation} \eta:W'_n \to \overline{Q}_n \end{equation}
	by \cite[Definition 2.4 and Theorem 2.5]{[GSW17]} or \cite[Example 4.2.(d)]{[Wo15]} (cf. \cite[Definition 10.2]{[MW10]} also). Indeed, all conditions in loc. cit. are satisfied by our summary. 
	
	\begin{remark}\label{remark: relation between T curve and Q curve}
	Let $s \in W'_n(\cc) \subset C[[K_n]](\cc)$, and 
	\begin{equation*}
	\begin{aligned}
	{\mathcal H} &= \{\rho \in L_n: s \in D^v_\rho\} = \{\bot\} \cup \{\rho \in L^*_n: s \in D_\rho\} \quad \text{and} \\
	{\mathcal F} &= \{ S \subseteq [n]: |S| \geq 2, s \in D^h_S \} =  \{ S \subseteq [n]: |S| \geq 2, s \in E_S \}.
	\end{aligned}
	\end{equation*}
	The combinatorics of the various curves over $s$ is the following.
	\begin{itemize}
		\item The dual tree of $C[[K_n]]_s^\suf$ is $T({\mathcal H},{\mathcal F})$ (Definition \ref{definition: leveled tree representations}) by Theorem \ref{theorem: polydiagonal degeneration}.
		\item If we remove the component isomorphic to $C$ from $C[[K_n]]_s^\suf$ (cf. Lemma \ref{lemma: deleting C}), the dual tree is $T({\mathcal H},{\mathcal F})-\star$. 
		\item The dual tree of $Z_s$ is $T^s({\mathcal H},{\mathcal F}) - \star$ by Proposition \ref{proposition: stabilization of universal curve} and the considerations above. Note that $Z_s$ is also the  curve corresponding to $\eta(s) \in \overline{Q}_n$.
	\end{itemize}
	\end{remark}

	\begin{lemma}\label{lemma: Q normal}
		The variety $\overline{Q}_n$ is normal.
	\end{lemma}
	
	\begin{proof}
		In \cite[Example 4.2.(d)]{[Wo15]}, it is stated and sketched that $\overline{Q}_{n+1}$ is the universal curve over $\overline{Q}_n$. Then we can easily argue inductively that $\overline{Q}_n$ is lci (prestable curves are lci over their base and compositions of lci maps are lci under very mild hypotheses) and regular in codimension $1$, hence normal since Serre's $S_2$ condition is implied by lci (in fact, even by Cohen--Macaulay \cite[\href{https://stacks.math.columbia.edu/tag/0342}{Tag 0342}]{[SP]}). Alternatively, $\overline{Q}_n$ has toric singularities \cite[Corollary 10.6]{[MW10]}, which are Cohen--Macaulay \cite{[Ka94]}.
	\end{proof}
	
	\begin{remark}\label{remark: Q normal}
		We may use an alternate weaker form of Lemma \ref{lemma: Q normal} if we wish to avoid the harder facts about $\overline{Q}_n$ invoked above. In the proofs of Propositions \ref{proposition: birational for q} and \ref{proposition: isomorphism in codimension 2 for q} below, we will only require normality for the open subset of $\overline{Q}_n$ which consists of the codimension $0$ and $1$ strata. However, this open is in fact smooth, which follows by direct calculations from the construction in \cite[\S10]{[MW10]}.
	\end{remark}
	
	\begin{proposition}\label{proposition: birational for q}
		The restriction of $\eta$ to $(W'_n)^\circ$ is an isomorphism, and $ \eta^{-1}(Q_n) = (W'_n)^\circ$. In particular, $\eta$ is birational.
	\end{proposition}
	
	\begin{proof}
	Similarly to the proof Proposition \ref{proposition: birational}, the only step that requires a longer verification is that $\eta$ is injective on $(W'_n)^\circ$. This is not just similar to the injectivity of $\gamma$ on $W_n^\circ$, but actually follows from it, since $(W'_n)^\circ$ can be naturally identified with an open subset of $W_n^\circ$, in a manner which also identifies the families of curves over them, their marked points, and scalings, keeping Proposition \ref{proposition: descending scaling 2} in mind.
	\end{proof}

	\begin{proposition}\label{proposition: isomorphism in codimension 2 for q}
		The birational morphism $\eta$ is small.
	\end{proposition}
	
	\begin{proof}
	This is similar to the proof of Proposition \ref{proposition: isomorphism in codimension 2}, with the minor difference that now we have two types of boundary divisors. As in that proof, let $R \subset W'_n$ be the set of points whose $\eta$-fiber has positive dimension, and $E = \eta^{-1}(R)$. Then $E \subseteq \partial W'_n = W'_n \backslash (W'_n)^\circ$ by Proposition \ref{proposition: birational for q}. 
	
	Consider the following strata of $\overline{Q}_n$. For each $\rho \in L_n^*$, let $Q_n[\rho] \subset \overline{Q}_n$ be the locally closed subset which parametrizes curves which fit the same description as the curves in $P_n[\rho]$ from the proof of Proposition \ref{proposition: isomorphism in codimension 2} (in fact, it can be shown that $Q_n[\rho]$ is an open subset of $P_n[\rho]$, since the marked points are now distinct). For each subset $S \subseteq [n]$, $|S| \geq 2$, we also have a locally closed subset $Q_n[S] \subset \overline{Q}_n$ which parametrizes curves with two components, one of which contains all marked points indexed by $S$ and no others (including $\infty$). We have
		\begin{equation}\label{equation: stratum dimension on Q side}
			\dim Q_n[\rho] = n-2 \quad \text{and} \quad \dim Q_n[S] = n-2.
		\end{equation}
	The former is completely analogous to \eqref{equation: stratum dimension on P side}, while the latter can be proved by similar techniques.
	
	Assume that $E$ had codimension $1$ in $W'_n$. Then, arguments similar to those in the proof of Proposition \ref{proposition: isomorphism in codimension 2} show that $E$ contains either $D_\rho$ for some $\rho \in L_n^*$, or $E_S$ for some $|S| \geq 2$, and also rule out the first alternative. To rule out the second possibility, note that Remark \ref{remark: relation between T curve and Q curve} and Theorem \ref{theorem: polydiagonal degeneration} imply that
	$$ \eta^{-1}(Q_n[S]) = E_S^\circ. $$ 
	Once again, the fact that $E_S^\circ$ and $Q_n[S]$ have the same dimension is at odds with the assumption that $\eta$ maps the former to the latter with positive dimensional fibers. In conclusion, $E$ has codimension at least $2$ in $W'_n$. We conclude the proof using normality, cf. Lemma \ref{lemma: Q normal} and Remark \ref{remark: Q normal}.
	\end{proof}
	
	\begin{proposition}\label{proposition: not an isomorphism in general for q}
		The morphism $\eta$ is not an isomorphism if $n \geq 4$.
	\end{proposition}
	
	\begin{proof}
		The marked points on the `binary curve' (corresponding to $b \in \overline{P}_n$) used in the proof of Proposition \ref{proposition: not an isomorphism in general} are pairwise distinct, so the curve is also in $\overline{Q}_n$ and the same argument applies.
	\end{proof}
	
	Theorem \ref{theorem: small resolution theorem 2} follows from Propositions \ref{proposition: birational for q}, \ref{proposition: isomorphism in codimension 2 for q} and \ref{proposition: not an isomorphism in general for q}.
	
	\begin{remark}\label{remark: final remark}
		Although the use of the auxiliary curve $C$ and polydiagonal degenerations was not strictly necessary, another advantage besides streamlining the constructions is that it allows us to apply similar constructions to the space  $\overline{M}_{n,1}(C)$ from \cite[Example 4.2.(e)]{[Wo15]} and \cite[Definition 2.1]{[GSW17]}. Indeed, using analogous methods, the interested reader can construct a natural morphism $C[[K_n]]_0 \to \overline{M}_{n,1}(C)$. 
	\end{remark}


\appendix
	
		\section{Contracting components of curves in families}\label{appendix: contracting components} Given a (reduced, projective, connected but reducible) curve $Y$ with only nodal singularities and an irreducible component $Z$ of $Y$, we say that $Z$ is a \emph{rational tail} of $Y$ if $Z \simeq {\mathbb P}^1$ and $Z$ contains precisely one node of $Y$; or that $Z$ is a \emph{separating rational bridge} of $Y$ if $Z \simeq {\mathbb P}^1$, $Z$ contains precisely two nodes of $Y$, and $Y \backslash Z$ is disconnected. Note that the union of some rational tails and separating rational bridges of $Y$ must be a disjoint union of rational chains (a chain ${\mathbb P}^1 \cup \cdots \cup {\mathbb P}^1$ with normal crossings) attached to the rest of $Y$ at either one or two points.
		
		\begin{situation}\label{situation: contracting bridges semi-general}
			Let $S$ be a reduced, separated, connected, finite type scheme over $\cc$, and $\pi:C \to S$ a prestable curve over $S$. Let ${\sh L}$ be a line bundle on $C$ such that for any $s \in S(\cc)$ and any irreducible component $\Sigma$ of $C_s$, $\deg {\sh L}_s > 0$ and $\deg {\sh L}_s|_\Sigma \geq 0$, with strict inequality if $\Sigma$ is neither a rational tail, nor a separating rational bridge of $C_s$.
		\end{situation}
		
		The main goal of this section is to prove the following.
		
		\begin{proposition}\label{proposition: contracting bridges semi-general}
			In Situation \ref{situation: contracting bridges semi-general}, there exist a prestable curve $\overline{C} \to S$ and a morphism $f:C \to \overline{C}$ over $S$ which, for every $s \in S(\cc)$, contracts all irreducible components of $C_s$ on which ${\sh L}_s$ has degree $0$ and nothing else.
		\end{proposition}

		The result is clearly well-known, but the author does not know a suitable reference, so we will use a very slight variation on the proof of \cite[Proposition 3.10]{[BM96]}, which, in turn, is related to \cite[\S1]{[Kn83]}. (The lemmas in \cite[\href{https://stacks.math.columbia.edu/tag/0E7B}{Tag 0E7B}]{[SP]} suggest some potential alternate approaches, but we will not take this route.)
		
		\begin{lemma}\label{lemma: Lemma C.1.3}
			Proposition \ref{proposition: contracting bridges semi-general} holds if $S = \spec \cc$.
		\end{lemma}
		
		\begin{proof}
			We can construct $\overline{C}$ by elementary methods. The details are omitted.
		\end{proof}
		
		\begin{situation}\label{situation: contracting bridges semi-general + odd assumption}
			In Situation \ref{situation: contracting bridges semi-general}, assume furthermore that, for any $s \in S(\cc)$, there exist $M \geq 3$ and a line bundle ${\sh G} \in \Pic(C_s)$ such that ${\sh L}_s = {\sh G}^{\otimes M}$, and 
			$$ \deg {\sh G}|_\Sigma \geq 5g + 2\#\text{(irreducible components of $C_s$)} + 4 $$
			for any irreducible component $\Sigma$ such that $\deg {\sh L}_s|_\Sigma > 0$, where $g= p_a(C_s)$.  
		\end{situation}
		
		Until explicitly stated otherwise, we assume we are in Situation \ref{situation: contracting bridges semi-general + odd assumption}. The assumption in Situation \ref{situation: contracting bridges semi-general + odd assumption} is indeed extremely arbitrary, there is no particular significance beyond `nonzero degrees are large'. 
		
		Recall that a coherent ${\sh O}_X$-module ${\sh F}$ is \emph{normally generated} if the natural map $\Gamma(X,{\sh F})^{\otimes k} \to \Gamma(X,{\sh F}^{\otimes k})$ is surjective, for all $k \geq 1$, e.g. \cite[Definition 1.7]{[Kn83]}.
		
		\begin{proposition}\label{proposition: proposition C.1.5}
			Assume that $S = \spec \cc$ and ${\sh L}$ is ample. Then, ${\sh L}$ is generated by global sections, normally generated, and nonspecial (that is, $H^1(C,{\sh L}) = 0$). 
		\end{proposition}
		
		\begin{proof}
			We will deduce the claim artificially from \cite[Theorem 1.8]{[Kn83]}.
			
			First, we claim that the vanishing locus of a general global section of ${\sh G} \otimes \omega_C^\vee$ is $0$-dimensional, reduced, and disjoint from the singular locus of $C$. We sketch a proof of this uninteresting technicality. For $0$-dimensional and reduced, we may restrict attention to sections which vanish at all nodes, and then, after normalizing and restricting to connected components, the claim follows from the fact that a complete linear system of degree at least twice the genus on a smooth (connected, projective) curve is base-point free and its general member is reduced. For `disjoint from the singular locus', let $\nu:C' \to C$ be the partial normalization of $C$ at an arbitrary node with preimages $q_1,q_2 \in C'$. The claim amounts to $h^0(C',\nu^*({\sh G} \otimes \omega^\vee_C) \otimes {\sh O}_{C'} (-q_1-q_2)) = h^0(C,{\sh G} \otimes \omega_C^\vee) - 1$, which follows from the non-speciality of the two line bundles, by Serre duality and some routine manipulations of the bounds.
			
			Let $p_1+\cdots+p_n$ be the vanishing locus of a general global section of ${\sh G} \otimes \omega_C^\vee$. Then $p_1,\ldots,p_n \in C(\cc)$ are smooth and distinct, and ${\sh G} \otimes \omega_C^\vee \simeq {\sh O}_C(p_1 + \cdots + p_n)$, or equivalently, ${\sh G} \simeq \omega_C(p_1 + \cdots + p_n)$. Since ${\sh G}$ is ample, $(C,p_1,\ldots,p_n)$ is stable, and the claim follows from \cite[Theorem 1.8]{[Kn83]}.
		\end{proof}
		
		\begin{corollary}\label{corollary: corollary C.1.6}
			Assume that $S = \spec \cc$. Then, ${\sh L}^{\otimes n}$ is generated by global sections, normally generated, and nonspecial for all $n \geq 1$. 
		\end{corollary}
		
		\begin{proof}
			We may assume without loss of generality that $n=1$. Let $f:C \to \overline{C}$ as provided by Lemma \ref{lemma: Lemma C.1.3}.
			
			\begin{claim}\label{claim: claim C.1.7}
				The ${\sh O}_{\overline{C}}$-module ${\sh J}:=f_*{\sh L}$ is invertible and ample. Moreover, for any integer $k \geq 0$,  $f^*{\sh J}^{\otimes {k}} \simeq {\sh L}^{\otimes k}$, ${\sh J}^{\otimes {k}} \simeq f_*{\sh L}^{\otimes {k}}$, and $R^1f_*{\sh L}^{\otimes k} = 0$.
			\end{claim}
			
			\begin{proof}
				The proof is elementary and left to the reader.
			\end{proof}
			
			The corollary can be deduced from Proposition \ref{proposition: proposition C.1.5} and Claim \ref{claim: claim C.1.7}. Indeed, $\overline{C}$ and $f_*{\sh L}$ satisfy the assumptions of Proposition \ref{proposition: proposition C.1.5}.
		\end{proof}
		
		\begin{lemma}\label{lemma: lemma C.1.10}
			The following hold:
			\begin{enumerate}
				\item\label{item: conclusion 1 in lemma: lemma C.1.10} $\pi_*{\sh O}_C \simeq {\sh O}_S$, while $\pi_*{\sh L}^{\otimes k}$ is locally free, and its formation commutes with base change, for all $k \geq 0$;
				\item\label{item: conclusion 2 in lemma: lemma C.1.10} $\pi^*\pi_*{\sh L}^{\otimes k} \to {\sh L}^{\otimes k}$ is surjective for all $k \geq 0$; and
				\item\label{item: conclusion 3 in lemma: lemma C.1.10} $(\pi_*{\sh L})^{\otimes k} \to \pi_*({\sh L}^{\otimes k})$ is surjective for all $k \geq 0$.
			\end{enumerate}
		\end{lemma}
		
		\begin{proof}
			The fact that $\pi_*{\sh O}_C \simeq {\sh O}_S$ is well-known. If $k \geq 1$, then $H^1(C_s,{\sh L}_s^{\otimes k}) = 0$ by Corollary \ref{corollary: corollary C.1.6} when $s \in S(\cc)$, and hence for all $s \in S$ by semi-continuity, and item \ref{item: conclusion 1 in lemma: lemma C.1.10} follows from the `cohomology and base change' theorem. By item \ref{item: conclusion 1 in lemma: lemma C.1.10}, the formation of the homomorphisms in items \ref{item: conclusion 2 in lemma: lemma C.1.10} and \ref{item: conclusion 3 in lemma: lemma C.1.10} commutes with base change, so their surjectivity can be checked in the special case $S = \spec \cc$, when it follows from Corollary \ref{corollary: corollary C.1.6}.
		\end{proof}
		
		Let
		\begin{equation} \widehat{C} = \underline{\mathrm{Proj}}_S {\sh S}, \quad \text{where} \quad {\sh S} = \bigoplus_{k \geq 0} {\sh S}_k = \bigoplus_{k \geq 0} \pi_* ({\sh L}^{\otimes k}),  \end{equation}
		and $\widehat{\pi}:\widehat{C} \to S$ the natural projection. Let $\widehat{f}$ be the morphism associated to ${\sh L}$ and $\pi^*{\sh S} \to \bigoplus_{k \geq 0} {\sh L}^{\otimes k}$ (in the language of \cite[(3.7.1)]{[EGAII]}), that is, the composition
		\begin{equation}\label{equation: equation C.2}
			C = \underline{\mathrm{Proj}}_C \bigoplus_{k \geq 0} {\sh L}^{\otimes k} \supseteq U \xrightarrow{\tau} \underline{\mathrm{Proj}}_C \pi^*{\sh S} = C \times_S \widehat{C} \to \widehat{C},
		\end{equation}
		where $U \subseteq C$ is the suitable open, e.g. \cite[\href{https://stacks.math.columbia.edu/tag/07ZG}{Tag 07ZG}]{[SP]}.
		
		\begin{proposition}\label{proposition: proposition C.1.11}
			The following hold:
			\begin{enumerate}
				\item\label{item: conclusion 1 in proposition C.1.11} the formation of $\widehat{C}$ commutes with base change;
				\item\label{item: conclusion 2 in proposition C.1.11} the morphism $\widehat{\pi}$ is flat and projective;
				\item\label{item: conclusion 3 in proposition C.1.11} the morphism $\widehat{f}$ is defined everywhere (that is, $U = C$); and
				\item\label{item: conclusion 4 in proposition C.1.11} the formation of $\widehat{f}$ commutes with base change.
			\end{enumerate}
		\end{proposition}
		
		\begin{proof}
			All arguments needed can be found in the proofs of Claims 1--3 inside the proof \cite[Proposition 3.10]{[BM96]}. Item \ref{item: conclusion 1 in proposition C.1.11} follows from \ref{item: conclusion 1 in lemma: lemma C.1.10} in Lemma \ref{lemma: lemma C.1.10}. Flatness of $\widehat{\pi}$ follows from the fact in Lemma \ref{item: conclusion 1 in lemma: lemma C.1.10} that ${\sh S}_k$ is locally free (cf. Claim 2 in loc. cit.). Projectivity of $\widehat{\pi}$ follows from the fact that ${\sh S}$ is generated by ${\sh S}_1$ over ${\sh S}_0 = {\sh O}_S$ by Lemma \ref{lemma: lemma C.1.10} using \cite[Proposition 5.5.1]{[EGAII]} (cf. Claim 2 in \cite{[BM96]} again). Lemma \ref{lemma: lemma C.1.10} also gives that $\pi^*{\sh S}$ is generated by $\pi^*{\sh S}_1$ over $\pi^*{\sh S}_0 = {\sh O}_C$, so, it follows from \cite[\href{https://stacks.math.columbia.edu/tag/07ZK}{Tag 07ZK}]{[SP]} that $\widehat{f}$ is defined everywhere (in fact, $\tau$ is a closed immersion of $U=C$). The fact that the formation of $\widehat{f}$ commutes with base change is visible from \eqref{equation: equation C.2}. 
		\end{proof}
		
		\begin{lemma}\label{lemma: lemma C.1.11 new}
			Assume that $S = \spec {\cc}$. Let $f:C \to \overline{C}$ as provided by Lemma \ref{lemma: Lemma C.1.3}. Then, there exists an isomorphism $\psi:\overline{C} \to \widehat{C}$ such that $\widehat{f} = \psi f$.
		\end{lemma}
		
		\begin{proof}
			Recall that $\widehat{f}$ is defined everywhere, by Proposition \ref{proposition: proposition C.1.11}. We will use the notation $\Gamma_*(C,{\sh L}) = \bigoplus_{k \geq 0} \Gamma(C,{\sh L}^{\otimes k})$ and similar notation in analogous situations. By Claim \ref{claim: claim C.1.7}, $(f_*{\sh L})^{\otimes k} = f_*({\sh L}^{\otimes k})$, so 
			\begin{equation}\label{equation: equation C.3}
				\Gamma_*(\overline{C},f_*{\sh L}) \cong \Gamma_*(C,{\sh L}).
			\end{equation}
			By Claim \ref{claim: claim C.1.7}, $f_*{\sh L}$ is invertible and ample, so, by \cite[Proposition (4.6.3)]{[EGAII]},
			\begin{equation}\label{equation: equation C.4}
				\overline{C} \cong \Proj \Gamma_*(\overline{C},f_*{\sh L}).
			\end{equation}
			By \eqref{equation: equation C.3} and \eqref{equation: equation C.4}, we may define $\psi$ to be the composition
			\begin{equation}
				\overline{C} \cong \Proj \Gamma_*(\overline{C},f_*{\sh L}) \cong \Proj \Gamma_*(C,{\sh L}) = \widehat{C}. 
			\end{equation}
			Checking $\widehat{f} = \psi f$ requires some diagram chasing. Consider the diagram
			\begin{center}
				\begin{tikzpicture}
					\node (a) at (0.5,0) {$\overline{C}$};
					\node (b) at (3,0) {$\Proj \Gamma_*(\overline{C},f_*{\sh L})$};
					\node (c) at (7,0) {$\Proj \Gamma_*(C,{\sh L})$};
					\node (d) at (9.5,0) {$\widehat{C}$};
					
					\node (e) at (3,1) {$\Proj \mathrm{Sym}\mathop{} \Gamma(\overline{C},f_*{\sh L})$};
					\node (f) at (7,1) {$\Proj \mathrm{Sym}\mathop{} \Gamma(C,{\sh L})$};
					
					\node (g) at (5,2) {$C$};
					
					\draw[->] (g) to[out=180,in=90] node[above] {$f$} (a);
					\draw[->] (g) to[out=0,in=90] node[above] {$\widehat{f}$} (d);
					\draw[->] (a) to[out=-15,in=195] node[above] {$\psi$} (d);
					\draw[->] (a) to node[above] {\eqref{equation: equation C.4}} (b);
					\draw[double equal sign distance] (b) to node[above] {\eqref{equation: equation C.3}} (c);
					\draw[double equal sign distance] (c) to (d);
					\draw[right hook->] (b) to (e); \draw[right hook->] (c) to (f); 
					\draw[double equal sign distance] (e) to (f);
					\draw[->] (g) to (f);
					\draw[->] (a) to [out = 60, in =180] (e);
				\end{tikzpicture}
			\end{center}
			in which the maps $C \to \Proj \mathrm{Sym}\mathop{}\Gamma(C,{\sh L})$ and $\overline{C} \to \Proj \mathrm{Sym}\mathop{}\Gamma(\overline{C},f_*{\sh L})$ are the morphisms associated to the complete linear systems $|{\sh L}|$ and $|f_*{\sh L}|$. For instance, the former is $p \mapsto \Gamma(C,{\sh L}(-p))$. It is straightforward to check that all five bounded `faces' in the diagram are commutative. We will only mention that the commutativity of the upper left (concave, curved) square amounts to $\Gamma(C,{\sh L}(-p)) = \Gamma(\overline{C},(f_*{\sh L})(-f(p)))$ as subspaces of $\Gamma(C,{\sh L}) = \Gamma(\overline{C},f_*{\sh L})$. Let $p \in C(\cc)$. The commutativity of the faces implies that $\psi(f(p)),\widehat{f}(p) \in \widehat{C} = \Proj \Gamma_*(C,{\sh L})$ have the same image $\Gamma(C,{\sh L}(-p)) \in \Proj \mathrm{Sym} \mathop{} \Gamma(C,{\sh L})$, but since $\widehat{C} \to \Proj \mathrm{Sym} \mathop{} \Gamma (C,{\sh L})$ is a closed immersion by Corollary \ref{corollary: corollary C.1.6} (${\sh L}$ is normally generated) and \cite[\href{https://stacks.math.columbia.edu/tag/01N1}{Tag 01N1}]{[SP]}, it follows that $\psi(f(p)) = \widehat{f}(p)$, for any $p \in C(\cc)$. Everything is reduced, separated, of finite type over $\cc$, so we conclude $\widehat{f} = \psi f$.
		\end{proof}
		
		\begin{corollary}\label{corollary: corollary C.1.13}
			Proposition \ref{proposition: contracting bridges semi-general} holds (in Situation \ref{situation: contracting bridges semi-general + odd assumption}). 
		\end{corollary}
		
		\begin{proof}
			Define $\overline{C} = \widehat{C}$ and $f=\widehat{f}$. Everything follows from Proposition \ref{proposition: proposition C.1.11} and Lemmas \ref{lemma: lemma C.1.11 new} and \ref{lemma: Lemma C.1.3}.
		\end{proof}
		
		We exit Situation \ref{situation: contracting bridges semi-general + odd assumption} and enter Situation \ref{situation: contracting bridges semi-general}. To complete the proof of Proposition \ref{proposition: contracting bridges semi-general}, apply Corollary \ref{corollary: corollary C.1.13} with ${\sh L}$ replaced by ${\sh L}^{\otimes 3N}$ for $N \gg 0$. Indeed, the number of irreducible components of $C_s$ is bounded.
		
		We conclude with a further technical remark on Proposition \ref{proposition: contracting bridges semi-general}, for which it is convenient to use the following terminology.
		
		\begin{definition}[{\cite[Definition 3.1]{[Za21]}}]\label{definition: rational contraction}
			Let $S$ be a scheme, $\pi:X \to S$ and $\rho:Y \to S$ prestable curves, and 
			$ f:X \to Y$ an $S$-morphism. We say that $f$ \emph{has property R} if $\smash{ f^\#:{\sh O}_Y \to f_*{\sh O}_X }$ is an isomorphism and $R^1f_*{\sh O}_X = 0$, and these hold universally: $f_{S'}^\#:{\sh O}_{Y_{S'}} \to f_{S',*}{\sh O}_{X_{S'}} $ is an isomorphism and $R^1f_{S',*}{\sh O}_{X_{S'}} = 0$ for all $S' \to S$.
		\end{definition}
		
		\begin{lemma}\label{lemma: has property R in general}
			Let $S$ be a scheme of finite type over $\cc$, $\pi:X \to S$ and $\rho:Y \to S$ prestable curves over $S$, and $f:X \to Y$ an $S$-morphism. If $f_s^\#: {\sh O}_{Y_s} \to f_{s,*}{\sh O}_{X_s} $ is an isomorphism and $R^1f_{s,*}{\sh O}_{X_s} = 0$ for all $s \in S(\cc)$, then $f$ has property R.
		\end{lemma}
		
		\begin{proof}
			Follows from \cite[\href{https://stacks.math.columbia.edu/tag/0E88}{Tag 0E88}]{[SP]} and \cite[Lemma 3.3]{[Za21]}.
		\end{proof}

		\begin{remark}\label{remark: has property R in general}
			In particular, the morphism $f$ from Proposition \ref{proposition: contracting bridges semi-general} has property R, since it is elementary to check the property on fibers over $\cc$-points.
		\end{remark}
		
		We will also take this opportunity to prove Proposition \ref{proposition: morphism from Q to P}. 
		
	\begin{proof}[Proof of Proposition \ref{proposition: morphism from Q to P}]
	Let $\overline{Q}_n^+$ and $\overline{P}_n^+$ be the universal families of curves over $\overline{Q}_n$ and $\overline{P}_n$ respectively. Let $q \in \overline{Q}_n(\cc)$. We claim that there exists an elementary \'{e}tale neighbourhood $h:(U,u) \to (\overline{Q}_n,q)$ and a morphism of prestable curves $f:C \to C'$ over $U$ with property R (Definition \ref{definition: rational contraction}), where $C:=U \times_{\overline{Q}_n} \overline{Q}_n^+$, such that $f_s:C_s \to C'_s$ coincides with the morphism $\overline{Q}^+_{n,h(s)} \to \overline{P}^+_{n,\tau(h(s))} $ discussed in \S\ref{subsection: review of Q and P}, for each $s \in U(\cc)$. Let $y_1(q),\ldots,y_m(q) \in \overline{Q}_{n,q}^+$ be smooth points on components of (strictly) finite scaling, at least $2$ on each such component. Note that $(\overline{Q}_{n,q}^+,x_1(q),\ldots,x_n(q),x_\infty(q),y_1(q),\ldots,y_m(q)) \in \overline{M}_{0,m+n+1}(\cc)$. \'{E}tale locally, we may extend to smooth sections $y_1,\ldots,y_m: U \to C$, and since stability is an open condition in families, the data gives a $U$-point of $\overline{M}_{0,m+n+1}$ after shrinking $U$. Stabilizing by $U \to \overline{M}_{0,m+n+1} \to \overline{M}_{0,m+1}$ which forgets the $x_1,\ldots,x_n$ sections of $C$, we obtain a morphism of prestable curves $f:C \to C'$ over $U$. The desired property of $f_s$ is satisfied for $s=u$ by construction; we claim that, after further shrinking $U$ if needed, it will necessarily be satisfied for all $s \in U(\cc)$. Let $\Gamma_s$ be the dual tree of $C_s$ for $s \in U(\cc)$. After shrinking $U$, we have an edge contraction $\Gamma_u \to \Gamma_s$. Note that all vertices of $\Gamma_u$ map to vertices of $\Gamma_s$ of the same scaling or of (strictly) finite scaling, and then the claim follows purely combinatorially. 
	
	Next, we show that the scaling $\sigma:C \to {\mathbb P}(\omega_{C/U} \oplus {\sh O}_C)$ on $C$ descends to a vector field $v$ on $C'$. First, we claim that $C'$ is normal. Indeed, $\overline{Q}_n$ is lci (see the proof of Lemma \ref{lemma: Q normal}), hence $U$ is lci, hence $C'$ is lci since $C'$ is a prestable curve over $U$, and in particular $C'$ is $S_2$. On the other hand, $C'$ is nonsingular in codimension $1$ by a simple dimension count of potential singularities, thus $C'$ is indeed normal. Let $E \subset C$ be the exceptional locus of $f$, $V = C \setminus E$ its complement, $V' = f(V) \simeq V$, and $\sigma_V$ the restriction of $\sigma$ to $V$. Note that the restrictions $\omega_{C/U}|_V$ and $\omega_{C'/U}|_{V'}$ are canonically identified, and $\sigma_V$ is nowhere $0$ by construction, so $1/\sigma_V \in \Gamma(V',\omega_{C'/U})$. However, the complement $C' \setminus V'$ has at least codimension $2$ in $C'$ because its fiber over any point $s \in U$ is finite and its fiber over the generic point is empty, and $C'$ is normal, so $1/\sigma_V$ extends to a global section $v \in \Gamma(C',\omega_{C'/U})$ by what is colloquially known as the Algebraic Hartogs Lemma. The double vanishing of $v$ at $x_\infty$ and its non-vanishing at the sections $x'_1,\ldots,x'_n:U \to C'$ obtained by pushing forward the corresponding sections of $C \to U$ is immediately checked on fibers. By \cite[Theorem 1.3]{[Za21]}, we obtain a morphism $g_U:U \to \overline{P}_n$ and we conclude the proof by \'{e}tale descent, since the morphisms $g_U$ agree on \'{e}tale overlaps given that $g_U(s)$, $s \in U$, is uniquely determined by $h(s) \in \overline{Q}_n(\cc)$. 
	\end{proof}
	
	\section{Base change of wonderful compactifications} The formation of wonderful compactifications commutes with base change by a closed immersion transverse to the entire induced arrangement. The formal statement is the following.
	
	\begin{proposition}\label{proposition: functoriality of wc}
		Let $Y$ be a smooth variety, ${\mathcal G}$ a building set in $Y$ with induced arrangement ${\mathcal S}$, and $Y_{\mathcal G}$ the wonderful compactification of ${\mathcal G}$. Let $Y' \subset Y$ be a smooth closed subvariety whose intersection with any $S \in {\mathcal S}$ is transverse, ${\mathcal S}' = \{ Y' \cap S:S \in {\mathcal S} \}$, and ${\mathcal G}' = \{Y' \cap G: G \in {\mathcal G}\}$. 
		Assume also that all elements of ${\mathcal S}'$ are non-empty and connected.
		
		Then ${\mathcal S}'$ is an arrangement in $Y'$, ${\mathcal G}'$ is a building set for ${\mathcal S}'$, and
		\begin{equation}\label{equation: functoriality of wc}
			Y'_{{\mathcal G}'} = Y' \times_Y Y_{\mathcal G}.
		\end{equation} 
		Moreover, if ${\mathcal D}$ and ${\mathcal D}'$ are the collections of distinguished divisors on $Y_{\mathcal G}$ and $Y'_{{\mathcal G}'}$ respectively (i.e. the divisors in \cite[Theorem 1.2]{[Li09]}), then
		$ {\mathcal D}' = \{Y'_{{\mathcal G}'} \cap D: D \in {\mathcal D} \}$,
		where $Y'_{{\mathcal G}'} \hookrightarrow Y_{\mathcal G}$ by \eqref{equation: functoriality of wc}, and $Y'_{{\mathcal G}'}$ intersects $D$ transversally for any $D \in {\mathcal D}$.
	\end{proposition}
	
	\begin{proof}
	We write $Z' = Y' \cap Z$ for any closed subvariety $Z \subseteq Y$. Clearly, ${\mathcal S}'$ is an arrangement since $S_1'\cap S_2' = (S_1 \cap S_2)'$. Let $S \in {\mathcal S}$, and $G_1,\ldots,G_m$ the ${\mathcal G}$-factors of $S$. It is clear that $S' = G'_1 \cap \cdots \cap G'_m$ transversally. We also have to check that $S' \subseteq G'$ for given $G \in {\mathcal G}$ implies $G \supseteq G_i$ for some $i$. Indeed, the assumption above plus the transversality of $Y'$ with $G$ and $S$ implies $S \subseteq G$, so  $G \supseteq G_i$ for some $i$ by the definition of ${\mathcal G}$-factors. Hence, ${\mathcal G}'$ is a building set of ${\mathcal S}'$.
	
	The claim regarding the wonderful models amounts to the fact that everything in the iterative blowup construction of the wonderful models in \cite[\S2]{[Li09]} `commutes' with the suitable closed immersions, as shown below.
	
	\begin{lemma}\label{lemma: preparation for functoriality of wc}
		Let $F,S,Y'$ be smooth subvarieties of a smooth variety $Y$ such that $S \cap F$ is either smooth and connected, or empty, $S$ is not contained in $F$, and for any
		$ X \in \{F,S,S \cap F\} \backslash \{\emptyset\}$, the intersection of $Y'$ with $X$ is nonempty, connected, and transverse. Consider the blowup of $Y$ at $F$ with exceptional divisor $E$. 
		
		Then, $\widetilde{Y'}$ is isomorphic to the blowup of $Y'$ at $Y' \cap F$, and the intersection of $\widetilde{Y'}$ with $E$ is transverse. Furthermore,
		\begin{equation}\label{equation: commutativity with proper transforms for functoriality} \widetilde{Y' \cap S} = \widetilde{Y'} \cap \widetilde{S}, \end{equation} 
		where the dominant transform on the left hand side is taken relative to the blowup of $Y'$ at $Y' \cap F$, and, moreover, the intersection on the right hand side is transverse.
	\end{lemma}
	
	\begin{proof}
		The claims in the first sentence are clear, and so is everything if $Y' \cap S \cap F = \emptyset$. The most substantial claim is \eqref{equation: commutativity with proper transforms for functoriality}. The `$\subseteq$' inclusion is trivial. Let us compare the fibers of the two sides of \eqref{equation: commutativity with proper transforms for functoriality} over some $p \in Y' \cap S \cap F$. On one hand, the fiber of the left hand side is a projective space of dimension $\ell' = \dim S' - \dim (F' \cap S') - 1$, 
		using the notation $Z' = Y' \cap Z$. On the other hand, $E_p \subset \widetilde{Y'}$ since $Y'$ intersects $F$ transversally, so the fiber over $p$ of the right hand side is simply $\widetilde{S} \cap E_p$, which is a projective space of dimension $\ell = \dim S - \dim (F \cap S) - 1$. However, $\ell = \ell'$ by our transversality assumptions, so \eqref{equation: commutativity with proper transforms for functoriality} follows. The intersection on the right is smooth (therefore clean) by \eqref{equation: commutativity with proper transforms for functoriality}, and of the `expected dimension' by \eqref{equation: commutativity with proper transforms for functoriality} again, hence transverse. 
	\end{proof}
		
	In the situation of \cite[Proposition 2.8]{[Li09]} and of our proposition (so, with the additional data of $Y'$), we have $ \mathrm{Bl}_{F'} Y' = Y' \times_Y \mathrm{Bl}_F Y $ by Lemma \ref{lemma: preparation for functoriality of wc}, because $F$ and $Y'$ intersect transversally. Lemma \ref{lemma: preparation for functoriality of wc} and \cite[Proposition 2.8.(i)]{[Li09]} guarantee that the non-emptiness, connectivity, and transversality hypotheses on the intersection of the subvariety with all varieties in the arrangement continue to hold after blowing up $F$, so we may continue inductively.
	\end{proof}
	
	\section{Some facts related to blowing up} 
	
	\begin{proposition}\label{proposition: pushforward of ideals in blowup}
	Let $W$ be a smooth variety, $Y \subset W$ a smooth closed subvariety, and $\beta: W' \to W$ the blowup of $W$ at $Y$ with exceptional divisor $E \subset W'$. Then, $\beta^\# : {\sh O}_W \to \beta_*{\sh O}_{W'}$ is an isomorphism, and
	\begin{equation*}
		\beta^\#({\sh I}^k_{Y,W}) = \beta_*({\sh I}^k_{E,W'} )
	\end{equation*}
	as subsheaves of $\beta_*{\sh O}_{W'}$, for any integer $k \geq 0$.
	\end{proposition}

	\begin{proof}
		The formal proof will be omitted, but note that this is essentially the fact that, in local coordinates, if $Y=\{x_1=\cdots=x_r=0\} \subset W$ and $f$ is a function on $W$ defined locally, then the order of vanishing of $\beta^*f$ on $E$ is equal to the largest $k$ such that $\displaystyle \frac{\partial^{a_1+\cdots+a_r}f }{\partial x_1^{a_1}\cdots \partial x_r^{a_r}} = 0$ on $Y$ for all $a_1+\cdots+a_r \leq k-1$.
	\end{proof}
	
	\begin{corollary}\label{corollary: twisting for blowup} In the setup of Proposition \ref{proposition: pushforward of ideals in blowup}, let ${\sh L}$ be an invertible ${\sh O}_W$-module. Then, there exists a unique isomorphism
	\begin{equation*}
	{\sh L} \otimes {\sh I}_{Y,W} \simeq \beta_*(\beta^*{\sh L} \otimes {\sh I}_{E,W'} )
\end{equation*}
which restricts to the trivial one on $W \setminus Y$. Moreover, if $s$ is a local section of ${\sh L} \otimes {\sh I}_{Y,W}$ and $s'$ the corresponding section of $\beta^*{\sh L} \otimes {\sh I}_{E,W'}$ under the isomorphism above, then $s'|_E = 0$ if and only if $s$ belongs to ${\sh L} \otimes {\sh I}^2_{Y,W}$.
	\end{corollary}
	
	\begin{proof}
	Uniqueness is trivial, since $\mathrm{Aut}({\sh L} \otimes {\sh I}_{Y,W}) \to \mathrm{Aut}({\sh L}|_{W \backslash Y})$ is injective. For existence, working (Zariski) locally on $W$ (note that `gluing' follows from uniqueness), we may assume that ${\sh L} = {\sh O}_W$. The case $k = 1$ of Proposition \ref{proposition: pushforward of ideals in blowup} proves the desired isomorphism, while the case $k = 2$ proves the second claim.
	\end{proof}
	
	We have also used the following fact in the proof of Theorem \ref{theorem: polydiagonal degeneration}.
	
	\begin{lemma}\label{lemma: technical fact about dominant transforms}
		Let $\pi:W \to V$ be a smooth morphism between two smooth varieties, and $\sigma:V \to W$ a section of $\pi$. Let $R \subset V$ be a smooth closed subvariety, $V' = \mathrm{Bl}_R V$, and $W' = W \times_V V' = \mathrm{Bl}_{\pi^{-1}(R)}W$. Let $\pi':W' \to V'$ be the projection and $\sigma':V' \to W'$ the section of $\pi'$ induced by $\sigma$.
		
		If $A \subset V$ and $B \subset W$ are smooth closed subvarieties such that $\sigma(A) = B$, and $A' \subset V'$ and $B' \subset W'$ are the dominant transforms of $A$ and $B$, then $\sigma'(A') = B'$.
	\end{lemma}
	
	Note that $\mathrm{Bl}_{\pi^{-1}(R)}W = W \times_V V'$ follows from $\pi^*{\sh I}_{R,V}^k = {\sh I}_{\pi^{-1}(R),W}^k$ for all $k \geq 0$ (since $\pi$ is smooth) and the fact that the relative Proj construction commutes with base change.
	
	\begin{proof}
		After some reductions using \cite[\href{https://stacks.math.columbia.edu/tag/054L}{Tag 054L}]{[SP]}, we may assume without loss of generality that $W = V \times {\mathbb A}^m$, $\pi$ is the projection to the first factor, and $\sigma(v) = (v,0,\ldots,0)$ for all $v \in V$. Everything is clear in this situation.
	\end{proof}
	

	\bigskip
	
	\subsection*{Competing interests} No competing interest is declared.
	
	\subsection*{Funding}
	
	We acknowledge the support of the Natural Sciences and Engineering Research Council of Canada (NSERC), RGPIN-2020-05497. Cette recherche a \'{e}t\'{e} financ\'{e}e par le Conseil de recherches en sciences naturelles et en g\'{e}nie du Canada (CRSNG), RGPIN-2020-05497.
	
	\subsection*{Acknowledgements} I guessed that the small resolution of $\overline{P}_n$ (which was starting to take shape at the time) is an augmented wonderful variety when I saw \cite{[Pr22]}. I would like to thank Joel Kamnitzer and Nicholas Proudfoot for asking, respectively answering, the question on mathoverflow. I would also like to thank Joel Kamnitzer for interesting discussions related to this topic. In particular, Remarks \ref{remark: relation to spaces 2} and \ref{remark: flower curves} are based on these discussions.
	
	I am indebted to the anonymous referees not only for comments which improved the quality of the paper, but also for extremely interesting questions which suggest further directions of investigation.
	
	\bibliographystyle{plain}
	\bibliography{smallres_refs}{}

\end{document}